\documentclass[10pt]{article}
\topmargin - 0.7 in \overfullrule = 0pt
\usepackage{graphics}
\usepackage{graphicx}
\usepackage{color}
\usepackage{amsfonts}
\usepackage{dsfont}
\usepackage{amsmath}
\usepackage{amstext}
\usepackage{amsopn}
\usepackage{amsbsy}
\usepackage{amscd}
\usepackage{amsxtra}
\usepackage{amsthm}
\usepackage{amssymb}
\usepackage{esint}
\usepackage{abstract,lipsum,mathrsfs}
\usepackage[new]{old-arrows}

\numberwithin{equation}{section}

\title{Long-time Dynamics of Classical Patlak-Keller-Segel Equation}
\author{Chia-Yu Hsieh \hspace{6pt} and \hspace{6pt} Yong Yu}
\date{}

\newtheorem{thm}{Theorem}[section]

\newtheorem{lem}[thm]{Lemma}
\newtheorem{prop}[thm]{Proposition}

\newtheorem{rmk}[thm]{Remark}
\newtheorem*{thm_a}{Theorem A}

\def\Xint#1{\mathchoice
  {\XXint\displaystyle\textstyle{#1}}%
  {\XXint\textstyle\scriptstyle{#1}}%
  {\XXint\scriptstyle\scriptscriptstyle{#1}}%
  {\XXint\scriptscriptstyle\scriptscriptstyle{#1}}%
  \!\int}
\def\XXint#1#2#3{{\setbox0=\hbox{$#1{#2#3}{\int}$}
  \vcenter{\hbox{$#2#3$}}\kern-.5\wd0}}

\def\dashint{\Xint-}
\newcommand\mystrut{\rule{0pt}{7pt}}
\newcommand\mynewstrut{\rule{0pt}{9pt}}

\begin{document}
\maketitle

\begin{abstract}
\noindent \textbf{Abstract:}  When the spatial dimension $n =2$, it has been well-known that a global mild solution to classical Patlak-Keller-Segel equation (PKS equation for short) exists if and only if its initial total mass is not in supercritical regime. However, to study long-time behavior of a global mild solution to $2$D PKS equation usually requires additional assumptions on initial data. These additional assumptions include a finite-free-energy assumption and a finite-second-moment assumption. Moreover, $u_0 \log u_0$ should be $L^1$-integrable over $\mathbb{R}^2$, where $u_0$ is the initial mass density of cells. In earlier works, within subcritical regime and with these additional assumptions, a global mild solution to $2$D PKS equation is shown to approach a self-similar profile when time is large. It is unclear whether these additional assumptions are necessary to characterize the long-time behavior of global mild solutions to PKS equation in 2D. There are probably two difficulties. Firstly, $L^1$\hspace{0.5pt}-\hspace{0.5pt}integrability of the initial data is not strong enough to induce the finiteness of free energy. Entropy method cannot be applied. Secondly, the gradient of chemoattractant concentration does not satisfy the Carlen-Loss condition for the viscously damped conservation law. Therefore, we have no Carlen-Loss type estimate for PKS equation. In this article, we introduce a novel argument to push and stretch a space-time strip. By this way, we gain $L^1$\hspace{0.5pt}-\hspace{0.5pt}compactness of PKS equation expressed under similarity variables. As a consequence, we obtain global dynamics of 2D PKS equation in subcritical regime with no additional assumptions. As for the higher dimensional case in which the spatial dimension $n \geq 3$, we also characterize the long-time asymptotics of global mild solutions to PKS equation. With a finite-total-mass assumption on density of cells, any global mild solution to PKS equation will approach a self-similar profile when time is large, provided that there is a sequence of time going to infinity on which $L^\infty$\hspace{0.5pt}-\hspace{0.5pt}norm of the density of cells converges to zero. The self-similar profile in the higher dimensional case is given by the function $M \hspace{1pt}\mathcal{G}_n$, where  $M$ is the total mass of cells and $\mathcal{G}_n$ denotes the standard $n$-dimensional Gaussian probability density. Convergence rates to self-similar profiles are also discussed in any dimensions. Particularly in the higher dimensional case, the general  convergence rate for $L^1$\hspace{0.5pt}-\hspace{0.5pt}initial data can be improved if the initial data has a finite second moment. In fact, when time is large and $n \geq 3$, we provide, in an optimal way, a higher-order approximation of global mild solutions to PKS equation if the initial density has a finite second moment. All convergence rates studied in this article are under the $L^p$-norm with $p \in [\hspace{1pt}1, \infty\hspace{1pt}]$.
\end{abstract}

\section{\large Introduction}\label{sec_1}\vspace{0.5pc}

General Patlak-Keller-Segel equation was introduced in \cite{P53} and \cite{KS70} as a fundamental model for chemotaxis in the cell population. In its simplest formulation,  the equation is given by 
\begin{eqnarray}\label{eqn_0_0_0}
\left\{\begin{array}{lcl}
\hspace{4.6pt} \partial_{\hspace{0.5pt}t} u \hspace{1.5pt}= \Delta u - \nabla \cdot \left( u \nabla v \right) \qquad &&\text{on $\mathbb{R}^{n+1}_+$}; \\[3mm]
-\Delta v = u  && \text{on $\mathbb{R}^{n+1}_+$}.
\end{array} \right.
\end{eqnarray}
Here $u$ denotes the mass density of cells. $v$ is the chemoattractant concentration. $\mathbb{R}^{n+1}_+ = \mathbb{R}^n \times \mathbb{R}_+$ with $\mathbb{R}_+ = (0, \infty)$. The equation (\ref{eqn_0_0_0}) is usually referred to as classical Patlak-Keller-Segel equation. This equation is also used in astrophysics  to describe gravitationally interacting massive particles (see e.\hspace{0.5pt}g.\hspace{1.5pt}\cite{B95,B95.2,Chan42}\hspace{1pt}).  Denote by $E_n$ the fundamental solution of $-\Delta$ on $\mathbb{R}^n$ and suppose that $u$ decays sufficiently fast at spatial infinity. The equation (\ref{eqn_0_0_0}) can be written in a non-local form as follows:
\begin{eqnarray}\label{eqn_0_0_1} \partial_{\hspace{0.5pt}t} u \hspace{2pt}= \Delta u - \nabla \cdot \big( u \nabla \big(E_n* u \big) \big) \qquad \text{on $ \mathbb{R}^{n +1}_+$.}
\end{eqnarray}
Here $*$ is the standard convolution on $\mathbb{R}^n$.
\subsection{\normalsize Existing research works}\label{sec_1_1}\vspace{0.5pc} 

Since 1980s,  extensive research works have been devoted to studying  (\ref{eqn_0_0_0}) or (\ref{eqn_0_0_1}). In order to obtain global weak solutions or blow-up solutions to (\ref{eqn_0_0_0}), usually there should have some additional assumptions on the initial density $u_0$. For example, in 2D case, besides the finite-total-mass condition, $u_0$ is also assumed to have a finite free energy and a finite second moment in literatures. In addition, $u_0 \log u_0$ should be $L^1$-integrable over $\mathbb{R}^2$. With these additional assumptions, Blanchet-Dolbeault-Perthame \cite{BDP06} shows the existence of a critical threshold of total mass. If the initial total mass is below the threshold, then it lies in the subcritical regime. In this case there exists a non-negative global weak solution to (\ref{eqn_0_0_0}). If the initial total mass is above the threshold, then it lies in the supercritical regime. Solutions to (\ref{eqn_0_0_0}) with supercritical initial data must blow up in finite time. It is until recently that the additional assumptions used in \cite{BDP06} are dropped. See \cite{W18}. As is well-known, $E_2 * u$ exists if and only if $u \hspace{0.5pt}E_2$ is integrable on $\mathbb{R}^2 \setminus D_2$. Here $D_2$ is the disk in $\mathbb{R}^2$ with center $0$ and radius $2$. Merely with $L^1$-integrability of initial density, it is problematic to define the logarithmic potential $E_2 * u$ in (\ref{eqn_0_0_1}). Therefore in \cite{W18}, the author chooses to study mild solutions of the equation: \begin{align}\label{eqn_0_0_2} \partial_{\hspace{0.5pt}t} u \hspace{2pt}= \Delta u - \nabla \cdot \big( u \nabla E_2* u \big) \qquad \text{on $ \mathbb{R}^{2 +1}_+$.}
\end{align}The critical threshold phenomenon of (\ref{eqn_0_0_2}) purely depends on the initial total mass without any additional assumptions.  \vspace{0.3pc}

In 2D case, there are also many delicated works on specific properties of solutions to (\ref{eqn_0_0_2}). For example, some finite-time blow-up solutions have been found in the supercritical regime. See \cite{CGMN19,RS14b,V02,V04a,V04b}. As for the critical regime where the total mass equals to the critical threshold, solutions to (\ref{eqn_0_0_2}) is shown to exist globally in time (see \cite{BKLN06b} for the radial case and \cite{W18} for the general case). However, in this regime, there coexist both blow-up and non-blow-up solutions. An infinite time blow-up solution was found by Blanchet-Carrillo-Masmoudi  \cite{BCM08}. The solution has a finite free energy and a finite second moment. Moreover, it converges to a Dirac measure as time diverges to infinity. The blow-up rate problem has also been addressed  recently in \cite{DDDMW19} and \cite{GM18} by different methods. As for non-blow-up solutions, we refer readers to \cite{BCC12} where the authors prove the existence of properly dissipative weak solutions to the classical 2D Patlak-Keller-Segel equation. Their solutions have a finite free energy; however, the second moment is infinite. In \cite{BCC12} the authors also suggest that as time goes to infinity, the convergence rate problem of their solutions can be studied by some stability result  of Gagliardo-Nirenberg-Sobolev inequality. This program is finally completed by Carlen-Figalli in \cite{CF13}. In the subcritical regime, Blanchet-Dolbeault-Perthame \cite{BDP06} shows that a global solution to (\ref{eqn_0_0_2}) approaches a self-similar profile under $L^1$-norm as time goes to infinity. The problem on the convergence rate to the self-similar profile has been considered in \cite{BDEF08}, \cite{CD14} and \cite{FM16}. So far, in the subcritical regime, the long-time behaviors discussed in these works rely heavily on some technical assumptions. Therefore, part of our current work is to drop these assumptions and  provide a theory on global dynamics of (\ref{eqn_0_0_2}) in the subcritical regime. To finish our discussions on the 2D case, let us mention that there are also some local well-posedness results of (\ref{eqn_0_0_2}) with measure-valued initial data. See \cite{BM14, BZ15, SS02}. Of particular interest to us is the uniqueness result of Bedrossian-Masmoudi \cite{BM14}. Readers may refer to Sect.\hspace{0.5pt}1.3.3, where we will show an application of their uniqueness result in the study of global dynamics of (\ref{eqn_0_0_2}). \vspace{0.3pc} 

Compared to the $2$D case, many problems are left open if the spatial dimension $n \geq 3$. First of all, let us summarize some known well-posedness results  of (\ref{eqn_0_0_1}) in the higher dimensional case. Local existence of solutions to (\ref{eqn_0_0_1}) is established in \cite{B95,KSY12}, where initial data has $L^p$-integrability over $\mathbb{R}^n$. Here $p \geq n/2$. To obtain global solutions in some functional spaces, the scaling property of (\ref{eqn_0_0_1}) should be utilized. For any function $u$ on $\mathbb{R}_+^{n+1}$ and positive number $\mu$, we define the scaled function
\begin{align}\label{scaling}
u_\mu\left(x,t\right)  :=  \mu^2 u \left( \mu x , \mu^2 t \right) .
\end{align}
Under this transformation,  the equation (\ref{eqn_0_0_1}) is scaling invariant. It is also translational invariant. Therefore, (\ref{eqn_0_0_1}) may admit a global solution if the initial data is small in some translational and scaling invariant space. Here and in what follows (except otherwise stated), all functions in a Banach space $E$ are defined on $\mathbb{R}^n$. A Banach space $E$ equipped with norm $\| \cdot \|_E$ is called scaling invariant if for any $u \in E$ and $\mu > 0$, the scaled function \begin{align*} u_\mu\left(x \right) := \mu^2 u\left(\mu x\right) \end{align*}lies in $ E$ with $\| u_\mu \|_E = \| u \|_E$. In \cite{CPZ04, KSY12}, the authors show that (\ref{eqn_0_0_1}) has a global  weak solution, provided that the initial data is small in $L^{n/2}$-space. A similar global existence result has also been obtained in \cite{KS10} for weak-$L^{n/2}$ initial data. In general, it is shown in \cite{L13} that any Banach space $E$ of tempered distributions on $\mathbb{R}^n$ is embedded into the homogeneous Besov space $\dot{B}_{\infty,\infty}^{-2}$ if the norm of $E$ is translational and scaling invariant.  Therefore, to solve (\ref{eqn_0_0_1}) globally, it is natural to choose the initial data from  $\dot{B}_{\infty,\infty}^{-2}$. By \cite{L13}, a non-negative global mild solution to (\ref{eqn_0_0_1}) can be obtained with the initial data small in  $\dot{B}_{\infty,\infty}^{-2}$. However, similarly as in the work of Bourgain-Pavlovi\'{c} \cite{BP08} for the 3D incompressible Navier-Stokes equation, it shows in \cite{I11} that solutions to (\ref{eqn_0_0_1}) with the initial data in $\dot{B}_{\infty,\infty}^{-2}$ may not be continuously dependent on the initial data.  Besides the spaces mentioned above, some other translational and scaling invariant spaces have also been used to obtain global solutions of (\ref{eqn_0_0_1}). They include Triebel-Lizorkin space $\dot{F}_{\hspace{0.5pt}\infty, 2}^{-2}$, Besov type spaces $\dot{B}_{\hspace{0.5pt}p, \infty}^{-2 + \frac{n}{p}}$ and Morrey space $\dot{M}_{n/2}^1$.  These functional spaces satisfy the following inclusion relations:
\begin{align}\label{embedding}
L_+^{n/2}  \longhookrightarrow \dot{B}_{p,\infty; +}^{-2+ \frac{n}{p}}  \longhookrightarrow \dot{F}_{\infty,2; +}^{-2}  \longhookrightarrow \dot{M}_{n/2; +}^1  = \dot{B}_{\infty,\infty; +}^{-2}, \hspace{15pt} \text{for all $ p \geq n\big/2$.}
\end{align}Here for any functional space $E$, the notation $E_+$ denotes the subset of $E$ which contains all non-negative functions in $E$. Notice that although  (\ref{eqn_0_0_1}) is ill-posed in $\dot{B}_{\infty,\infty}^{-2}$, it is in fact well-posed in $\dot{F}_{\infty,2}^{-2}$. The following result is obtained by Iwabuchi-Nakamura in \cite{IN13}: 
\begin{thm_a}There exist positive constants $\epsilon_0$ and $c$ such that for any initial data $u_0$ with  $\|\hspace{0.5pt} u_0 \hspace{0.5pt}\|_{\dot{F}_{\infty,2}^{-2}} < \epsilon_0$, the equation (\ref{eqn_0_0_1}) has a unique global mild solution, denoted by $u$, on $\mathbb{R}_+^{n+1}$. The solution $u$ takes the initial data $u_0$ at $t = 0$. Moreover, it satisfies
\begin{align}
\sup_{t\hspace{0.5pt}>\hspace{0.5pt}0} \big\|\hspace{0.5pt} u(\cdot,t) \hspace{0.5pt}\big\|_{\dot{F}_{\infty,2}^{-2}} + \sup_{t\hspace{0.5pt}>\hspace{0.5pt}0} \hspace{1pt}t^{1/2} \big\|\hspace{0.5pt} u(\cdot,t) \hspace{0.5pt}\big\|_{\dot{F}_{\infty,2}^{-1}} \hspace{2pt}\leq \hspace{2pt}c\hspace{1pt}\epsilon_0.
\end{align}
\end{thm_a} 
\noindent Compared to the $2$D case where we can use $L^1$\hspace{0.5pt}-\hspace{0.5pt}norm of the non-negative initial data to characterize global solutions and blow-up solutions to (\ref{eqn_0_0_2}), the smallness assumption on $\dot{F}_{\infty,2}^{-2}$\hspace{0.5pt}-\hspace{0.5pt}norm of the initial data in Theorem A is far from being optimal for the study of a similar threshold problem in the higher dimensional case. Recently, radially symmetric global-in-time solutions to  (\ref{eqn_0_0_1}) have been obtained in \cite{BKP19}. In addition to some regularity assumptions on initial data, the authors in \cite{BKP19} require some $\dot{M}_{n/2}^1$\hspace{0.5pt}-\hspace{0.5pt}type norm of the initial data less than twice of the surface area of $\mathbb{S}^{n - 1}$. Although this constant is not small and might be the critical threshold for global and blow-up solutions to (\ref{eqn_0_0_1}) in higher dimensions, at least in the radially symmetric scenario, but so far the proof on the existence of this critical threshold is still open. The main difficulty is to understand the blow-up behavior of solutions at its possible singularity. Readers may refer to \cite{BZ19,BCKSV99,GMS11,N00,SW19}, where concentration and blow-up of solutions to (\ref{eqn_0_0_1}) are discussed for the $3$D case or the case with $n \geq 4$.

\subsection{\normalsize Main results}\label{sec_1_2}\vspace{0.5pc} 

\noindent In this section, we summarize our main results. The first result is concerned about the long-time behavior of global mild solutions to (\ref{eqn_0_0_2}) in which the spatial dimension $n = 2$. \begin{thm}\label{thm_1_2D}
Suppose that $u$ is a non-negative global mild solution to (\ref{eqn_0_0_2}) with the total mass $M$. Then the following three statements are equivalent:
\begin{enumerate}
\item[$\mathrm{1}.$] $M < 8\pi$;
\item[$\mathrm{2}.$] The $L^{\infty}$-norm of $u(\cdot, t)$ satisfies
\begin{eqnarray}\label{t -1 decay of solution} \limsup_{t \hspace{0.5pt}\rightarrow\hspace{0.5pt} \infty} \hspace{2pt} t \hspace{1pt} \big\| u(\cdot, t) \big\|_{\infty} < \infty.
\end{eqnarray}Here and throughout the article, $\| \cdot \|_p$ denotes the standard $L^p$-norm of a function on $\mathbb{R}^n$;
\item[$\mathrm{3}.$] The total mass $M$ is strictly less than $8\pi$. Moreover, it holds \vspace{0.2pc} $$\lim_{t \rightarrow \infty} t^{1 - \frac{1}{p}} \left\| u(\cdot, t) - \dfrac{1}{t} \hspace{1pt}G_M \left( \dfrac{\cdot}{\sqrt{t}} \right) \right\|_{p} = 0, \hspace{15pt}\text{ for any $1 \leq p \leq \infty$.}$$\vspace{0.2pc}\noindent In this limit, $G_M$  is a strictly positive function with the total mass $M$. Among all non-negative functions with the total mass $M$, the function $G_M$ is the unique solution to the following equation with finite free energy:  \begin{align}\label{stationary solution of U in 2d} \Delta_\xi U + U + \dfrac{1}{2}\hspace{1pt}\xi \cdot \nabla_\xi U = \nabla_\xi \cdot \big( U \hspace{1pt}\nabla_\xi E_2 * U \hspace{0.5pt} \big) \hspace{15pt} \text{on $\mathbb{R}^{2}$}.\end{align}In 2D case, the free energy associated with the above equation is given by \begin{align}\label{free energy in 2D}\mathcal{F}\left[ w \right] := \int_{\mathbb{R}^2} w\left(\xi\right)\log w\left(\xi\right)\mathrm{d} \xi + \dfrac{1}{4} \int_{\mathbb{R}^2}w\left(\xi\right)|\hspace{0.5pt}\xi\hspace{0.5pt}|^2 \hspace{2pt}\mathrm{d} \xi - \dfrac{1}{2} \int_{\mathbb{R}^2} w\left(\xi\right)\hspace{1pt} E_2 * w \left(\xi\right)\hspace{2pt}\mathrm{d} \xi.
\end{align}
\end{enumerate}
\end{thm}

A consequence of Theorem \ref{thm_1_2D} is that the supremum norm of a non-negative global mild solution to (\ref{eqn_0_0_2}) satisfies the decay estimate (\ref{t -1 decay of solution}) if and only if the total mass of this solution is strictly less than $8 \pi$. Moreover, the long-time behavior of a non-negative global mild solution to (\ref{eqn_0_0_2}) is uniquely governed by the self-similar profile  \text{\scriptsize $\dfrac{1}{t}G_M\left(\dfrac{\cdot}{ \sqrt{t}}\right)$} if the total mass $M$ of this solution is strictly less than $8 \pi$. \vspace{0.3pc} 

In the next, we introduce our results on the higher dimensional case.

\begin{thm}\label{thm_1}
Suppose that the spatial dimension $n \geq 3$. There exists a small positive constant $\epsilon_0$, which depends only on the spatial dimension $n$, so that the following four statements are equivalent for any non-negative mild solution $u$ of the equation (\ref{eqn_0_0_1}):
\begin{enumerate}
\item[$\mathrm{1}.$] On the existence time interval of $u$, denoted by $[\hspace{1pt}0, T_u\hspace{1pt})$, there is a time $T \in (\hspace{1pt}0, T_u\hspace{1pt})$ so that  $$\big\|u\left(\cdot, T\right)\big\|_{\dot{F}_{\infty,2}^{-2}} < \epsilon_0;$$
\item[$\mathrm{2}.$] The maximal existence time interval of $u$ equals to $[\hspace{1pt}0, \infty \hspace{0.5pt})$. Moreover, the $L^{\infty}$-norm of $u(\cdot, t)$  satisfies
\begin{eqnarray*} \limsup_{t \hspace{0.5pt}\rightarrow\hspace{0.5pt} \infty} \hspace{2pt} t \hspace{1pt} \big\| u(\cdot, t) \big\|_{\infty} < \infty;
\end{eqnarray*}
\item[$\mathrm{3}.$]  The maximal existence time interval of $u$ equals to $[\hspace{1pt}0, \infty \hspace{0.5pt})$. Moreover, the $L^{\infty}$-norm of $u(\cdot, t)$  satisfies
\begin{eqnarray*} \limsup_{t \rightarrow \infty} \hspace{2pt} t^{\frac{n}{2}} \hspace{1pt} \big\| u(\cdot, t) \big\|_{\infty} < \infty;
\end{eqnarray*}
\item[$\mathrm{4}.$] The maximal existence time interval of $u$ equals to $[\hspace{1pt}0, \infty \hspace{0.5pt})$.  In addition, it holds \vspace{0.2pc}$$\lim_{t \rightarrow \infty} t^{\frac{n}{2}\left( 1 - \frac{1}{p}\right)} \big\|\hspace{1pt} u(\cdot, t) - M\hspace{1pt}\Gamma_t\left(\cdot\right) \hspace{1pt}\big\|_p = 0, \hspace{15pt}\text{for any $1 \leq p \leq \infty$.}$$\vspace{0.2pc}\noindent Here $\Gamma_t\left(\cdot\right)$ is the fundamental solution of the heat equation $\partial_t w = \Delta w$ on $\mathbb{R}^n$. 
\end{enumerate}
\end{thm}
\noindent The decay estimate in Statement 4 of Theorem \ref{thm_1} can be improved if the initial density has a finite second moment. In fact, we have the following higher-order approximation of a global mild solution to (\ref{eqn_0_0_1}) near $t = \infty$:
 \begin{thm}\label{thm_2}
In addition to all the assumptions in Theorem \ref{thm_1}, we also assume that the initial density, denoted by $u_0$, of $u$ has a finite second moment. Then Theorem \ref{thm_1} still holds with Statement 4 replaced by the results given below:\vspace{0.5pc}\\
$\mathrm{5.}$ The maximal existence time interval of $u$ equals to $[\hspace{1pt}0, \infty \hspace{0.5pt})$.  Moreover, \begin{enumerate}
\item[$\mathrm{(1)}.$]  if $n \geq 5$, then it holds\begin{eqnarray*}t^{\frac{n}{2}\left(1 - \frac{1}{p}\right)} \hspace{2pt}\big\| u(\cdot, t) - M\hspace{1pt}\Gamma_t + B_0 \cdot \nabla \Gamma_t\hspace{1pt} \big\|_{p} = \mathrm{O}\big(t^{-1}\big), \hspace{15pt} \text{for any $p \in [\hspace{1pt}1, \infty\hspace{0.5pt}]$ and $t \geq 1$.}
\end{eqnarray*}Here $B_0$ is the center of mass defined by \begin{align}\label{center of mass} B_0 := \int_{\mathbb{R}^n} x \hspace{0.5pt}u_0\left(x\right)\mathrm{d} x;
\end{align}
\item[$\mathrm{(2)}.$] if $n = 3$, then for any $p \in [\hspace{1pt}1, \infty\hspace{0.5pt}]$ and $t \geq 1$, it holds  
\begin{eqnarray*}t^{\frac{3}{2}\left(1 - \frac{1}{p}\right)} \hspace{2pt}\left\| u(\cdot, t) - M\hspace{1pt}\Gamma_t + B_0 \cdot \nabla \Gamma_t + M^2 \hspace{0.5pt}\mathcal{W}\left(\cdot, t\right)  + c_1 t^{-\frac{5}{2}}\log t \left( \frac{1}{2} - \frac{|\hspace{0.5pt}x\hspace{0.5pt}|^2}{12t} \right) e^{-\frac{|\hspace{0.5pt}x\hspace{0.5pt}|^2}{4t}}\right\|_{p} = \mathrm{O}\big(t^{-1}\big).\end{eqnarray*}The function $\mathcal{W}$ is the unique global solution to the following initial value problem: \begin{eqnarray*}
\left\{\begin{array}{lcl}
\hspace{4.6pt} \partial_{\hspace{0.5pt}t} \mathcal{W} \hspace{1.5pt}= \Delta \mathcal{W} + \mathrm{div} \big(\hspace{1pt} \Gamma_t \nabla E_3 * \Gamma_t \hspace{1pt}\big) \qquad &&\text{on $\mathbb{R}^{3+1}_+$}; \\[3mm]
\hspace{15.5pt}\mathcal{W} = 0  && \text{on $\mathbb{R}^3 \times \big\{0\big\}$}.
\end{array} \right.
\end{eqnarray*}More explicitly, $\mathcal{W}$ can be represented by \begin{align}\label{representation of function cal W}\mathcal{W}\left(x, t\right) = t^{ - 2} \int_0^\infty e^{\frac{s}{2}} \hspace{1pt}S_3(s) \Big[ \hspace{1pt}\mathrm{div}\big(\mathcal{G}_3 \nabla \mathcal{V}_3\big)\Big] \hspace{2pt}\mathrm{d} s \hspace{2pt}\bigg|_{\frac{x}{\sqrt{t}}}.
\end{align}Here and throughout the remainder of this  article, for any dimension $n$, we use $\mathcal{G}_n$ to denote the following Gaussian probability density: \begin{align}\label{gaussian density def}\mathcal{G}_n\big(\xi\big) := \left(4 \pi\right)^{-\frac{n}{2}} e^{ - \frac{ | \hspace{0.5pt}\xi \hspace{0.5pt}|^2}{4}}, \hspace{15pt}\text{for all $\xi \in \mathbb{R}^n$.}
\end{align}The function $\mathcal{V}_n = E_n * \mathcal{G}_n$. $S_n\left(\tau\right)$ is the semi-group generated by the operator $L_n := \Delta_\xi + \frac{1}{2} \hspace{1pt} \xi \cdot \nabla_\xi + \frac{n}{2}$. The constant $c_1$ is given by \begin{align*} c_1 := (4\pi)^{-\frac{3}{2}} M \int_{\mathbb{R}^3} |z|^2 \hspace{1pt} \mathrm{div} \left( \mathcal{G}_3 \nabla \mathcal{V}_3^{(1)} + \mathcal{G}_3^{(1)} \nabla \mathcal{V}_3 \right) \mathrm{d}z,
\end{align*}where\begin{align*}
\mathcal{G}_3^{(1)} := B_0 \cdot \nabla \mathcal{G}_3 + M^2 \int_0^\infty e^{\frac{s}{2}} S_3(s) \hspace{1pt}\Big[ \mathrm{div}\Big(\hspace{1pt} \mathcal{G}_3 \nabla \mathcal{V}_3\Big) \hspace{1pt}\Big] \hspace{1pt}\mathrm{d} s \quad \mbox{and} \quad \mathcal{V}_3^{(1)} := E_3 * \mathcal{G}_3^{(1)};
\end{align*} 
\item[$\mathrm{(3)}.$] if $n = 4$, then for any $p \in [\hspace{1pt}1, \infty\hspace{0.5pt}]$ and $t \geq 1$, it holds 
\begin{eqnarray}t^{2\left(1 - \frac{1}{p}\right)} \hspace{2pt}\left\| u(\cdot, t) - M\hspace{1pt}\Gamma_t + B_0 \cdot \nabla \Gamma_t   - c_2 \hspace{0.5pt} t^{- 3}\log t \left( \frac{1}{2} - \frac{|\hspace{0.5pt}x\hspace{0.5pt}|^2}{16t} \right) e^{-\frac{|\hspace{0.5pt}x\hspace{0.5pt}|^2}{4t}}\right\|_{p} = \mathrm{O}\big(t^{-1}\big).\end{eqnarray}The constant $c_2$ is given as follows: \begin{align*} c_2 := \left(\frac{M}{4\pi}\right)^2  \int_{\mathbb{R}^4} |z|^2 \hspace{1pt}\mathrm{div} \Big( \mathcal{G}_4 \nabla \mathcal{V}_4 \Big) \mathrm{d}z.
\end{align*} \end{enumerate}
\end{thm}

We would like put some remarks on our results in Theorems \ref{thm_1} and \ref{thm_2}. Suppose that $n \geq 3$ and $u$ is a non-negative global mild solution to (\ref{eqn_0_0_1}). If it satisfies
\begin{align}\label{liminf_condition}
\liminf_{t\rightarrow\infty} \big\| u(\cdot, t) \big\|_{\infty} = 0,
\end{align}
then there is a sequence $\big\{ t_k \big\}$, which diverges to $\infty$ as $k \rightarrow \infty$, such that the $L^\infty$-norm of $u\left(\cdot, t_k\right)$ converges to $0$ as $k \rightarrow \infty$. Since the equation (\ref{eqn_0_0_1}) preserves the total mass, the $L^{n/2}$-norm of $u\left(\cdot, t_k\right)$ also converges to $0$ as $k \rightarrow \infty$. In light of the inclusion relations (\ref{embedding}), Statement 1 in Theorem \ref{thm_1} holds. These arguments induce that the condition (\ref{liminf_condition}) is in fact equivalent to all the statements  in Theorem \ref{thm_1}. Hence for any non-negative global mild solution to (\ref{eqn_0_0_1}), either it decays uniformly to $0$ as $t \rightarrow \infty$, or its $L^\infty$-norm has a strictly positive lower bound near $t = \infty$. In 2D case, when the total mass is strictly less than $8 \pi$, we have solutions uniformly decaying to $0$ as time goes to infinity. However, when the total mass equals to $8\pi$, there exists a global solution which converges to a non-zero stationary solution of (\ref{eqn_0_0_2}) when $t \rightarrow \infty$. See \cite{CF13}. As for the higher dimensional case, little is known about smooth stationary solutions to (\ref{eqn_0_0_1}). Our results in Theorems \ref{thm_1} and \ref{thm_2} describe the long-time behavior of global mild solutions to (\ref{eqn_0_0_1}) which decay to $0$ uniformly as $t \rightarrow \infty$.\vspace{0.3pc}

We also want to point out that in \cite{BKP19}, Biler-Karch-Pilarczyk have obtained some $L^p$\hspace{0.5pt}-\hspace{0.5pt}decay estimate for global solutions to (\ref{eqn_0_0_1}). Their results work for all spatial dimensions $n \geq 3$. However, they need some regularity assumptions on the initial data. And their results are for radially symmetric solutions only.  Our result in Statement 4 of Theorem \ref{thm_1} drops their symmetry assumption on the initial data. More than theirs, our decay rate is optimal and we also show long-time asymptotics of non-negative global mild solutions to (\ref{eqn_0_0_1}) when time is large.  


\subsection{\normalsize Mothodology and structure of the article}\label{sec_1_3}\vspace{0.5pc} 

We briefly discuss the ideas behind the proofs of our main results. In what follows, $u$ always denotes a non-negative global mild solution to either (\ref{eqn_0_0_1}) or (\ref{eqn_0_0_2}). $M$ is the total mass of $u$.

\subsubsection{Temporal decay estimate of $\| \hspace{1pt}u\left(\cdot, t\right)\|_\infty$} \vspace{0.5pc}

With an approximation argument, we can assume our initial density $u_0 \in L^1_+\left(\mathbb{R}^n\right)\cap L^\infty\left(\mathbb{R}^n\right)$. For any positive constant $K'$, we define \begin{align*}
T_0 := \sup \left\{ T > 0 : \big\|\hspace{0.5pt} u(\cdot,t) \hspace{0.5pt}\big\|_{\infty} \hspace{1.5pt}\leq\hspace{1.5pt} \dfrac{K'}{t}  \hspace{3pt} \mbox{ for all } t \in (\hspace{0.5pt}0,T\hspace{0.5pt}] \hspace{2pt}\right\}.
\end{align*}The time $T_0$ must be strictly positive or $\infty$. Therefore, to obtain Statement 2 in Theorems \ref{thm_1_2D} and \ref{thm_1}, we should show $T_0 = \infty$ for some $K'$ suitably chosen. Suppose that $T_0 < \infty$. By Struwe's arguments for heat flow of harmonic maps, we can translate and scale the solution $u$ to a new function $\widehat{u}$ so that the supremum norm of $\widehat{u}$ is bounded from above by $5$ on $\mathbb{R}^n \times [-1, 0\hspace{1pt}]$. The equations (\ref{eqn_0_0_1}) and (\ref{eqn_0_0_2}) are translational invariant. Moreover, they are also invariant under the scaling transformation (\ref{scaling}). Hence, the function $\widehat{u}$ satisfies \begin{align}\label{scaled equation for hat u}
\partial_\tau \widehat{u} - \Delta_\xi \widehat{u} + \big(\nabla_\xi E_n * \widehat{u}\hspace{1pt} \big) \cdot \nabla_\xi \widehat{u} - \widehat{u}^2 = 0 \qquad \mbox{on } P_1 := B_1 \times [-1, 0\hspace{1pt}].
\end{align}In (\ref{scaled equation for hat u}), $B_1$ is the unit ball in $\mathbb{R}^n$ with center $0$ and $(\xi, \tau) \in \mathbb{R}^n \times \mathbb{R}_+$ are new space-time variables. To apply local maximum estimate for parabolic equations to (\ref{scaled equation for hat u}), we need control the coefficients of $\nabla_\xi \widehat{u}$ in (\ref{scaled equation for hat u}). By a potential estimate, it satisfies \begin{align}\label{supre norm est of potential}\big\|\hspace{1pt} \nabla_\xi E_n * \widehat{u} \hspace{1pt}\big\|_\infty \hspace{2pt}\lesssim\hspace{2pt} \left\|\hspace{1pt} \widehat{u} \hspace{1pt}\right\|_1^{\frac{1}{n}} \left\|\hspace{1pt} \widehat{u}\hspace{1pt} \right\|_{\infty}^{1- \frac{1}{n}}, \hspace{15pt}\text{for all $n \geq 2$.}
\end{align}Here and in what follows, $A \lesssim B$ denotes $A \leq c B$, where $c > 0$ is a universal constant depending only on $n$.  In 2D case, the total mass is scaling and translational invariant. In conjunction with the supremum norm estimate of $\widehat{u}$, the estimate (\ref{supre norm est of potential}) induces the supremum norm estimate of $\nabla_\xi E_2 * \widehat{u}$ on $\mathbb{R}^2 \times [-1, 0\hspace{1pt}]$.  However, when the spatial dimension $n \geq 3$, this method fails. The reason is that $L^1$-norm is no longer scaling invariant and we have no control on the scaling used in the definition of $\widehat{u}$. It is also because of this reason that we need to estimate the supremum norm of $\nabla_\xi E_n * \widehat{u}$ in terms of some scaling and translational invariant norm. Notice that in higher dimensional case the Newtonian potential $\widehat{v} := E_n * \widehat{u}$ is well-defined if $\widehat{u}$ has enough integrability. Moreover, it also satisfies \begin{align*} - \Delta_\xi \widehat{v} = \widehat{u} \hspace{15pt}\text{on $\mathbb{R}^n \times [-1, 0\hspace{1pt}]$.}
\end{align*}With Morrey's inequality, standard $W^{2, p}$-estimate for elliptic equations and John-Nirenberg inequality, we can imply from the above equation that \begin{align*}\big\|\hspace{0.5pt} \nabla_\xi \widehat{v}\left(\cdot,\tau\right)\hspace{0.5pt} \big\|_{\infty} \hspace{2pt}&\lesssim\hspace{3pt} \big\|\hspace{0.5pt}\widehat{u}(\cdot, \tau)\hspace{0.5pt}\big\|_{\infty}  + \big\|\hspace{1pt} \widehat{v}(\cdot,\tau) \hspace{0.5pt}\big\|_{\mathrm{BMO}}, \hspace{10pt}\text{for all $\tau \in [-1, 0\hspace{1pt}]$.}\end{align*}As we can see, the BMO-norm of $\widehat{v}$ plays a key role in the estimate of the supremum norm of $\nabla_\xi \widehat{v}$, particularly in the higher dimensional case. If we can control $\mathrm{BMO}$-norm of $\widehat{v}$, then in conjunction with the supremum norm estimate of $\widehat{u}$, we can obtain the supremum norm estimate of $\nabla_\xi \widehat{v}$ on $\mathbb{R}^n \times [-1, 0\hspace{1pt}]$. At this stage, we should mention the work of Gotoh \cite{G87} on the BMO property of Newtoniam potential. Indeed, Gotoh shows that the BMO-norm of $\widehat{v}$ can be estimated in terms of some Morrey norm of $\widehat{u}$. More precisely, we have from Gotoh's result that $$\big\|\hspace{1pt} \widehat{v}\left(\cdot, \tau\right)\hspace{1pt}\big\|_\mathrm{BMO} \hspace{2pt}\lesssim\hspace{2pt} \big\|\hspace{1pt} \widehat{u}\left(\cdot, \tau\right)\hspace{1pt}\big\|_{\dot{M}^1_{n/2}}, \hspace{15pt}\text{for all $\tau \in [-1, 0\hspace{1pt}]$.}$$ Since the $\dot{M}^1_{n/2}$-norm is scaling and translational invariant, by Theorem A and (\ref{embedding}), the above estimate yields the uniform boundedness of $\| \hspace{1pt}\widehat{v}\left(\cdot, \tau\right)\hspace{1pt}\|_{\mathrm{BMO}}$ on $[ -1, 0\hspace{1pt}]$, which infers the uniform bound of the supremum norm of $\nabla_\xi \widehat{v}$ on $\mathbb{R}^n \times [-1, 0\hspace{1pt}]$. Now we can apply local maximum estimate for parabolic equations to (\ref{scaled equation for hat u}) and obtain \begin{align}\label{contradiction to get result}
1 = \widehat{u}\left(0,0\right) \hspace{2pt}\lesssim &\hspace{2pt} \int_{P_{1}} \widehat{u}  \hspace{2pt}\lesssim\hspace{2pt} \int_{e_0^{-1/2}}^{\sqrt{2} \hspace{0.5pt}e_0^{-1/2}} \dfrac{\Phi_{z_1}(\rho)}{\rho} \hspace{2pt}\mathrm{d}\rho.
\end{align}In the above estimates, the first equality holds by our construction. $e_0$ is a positive number. $z_1 = (y_0, s_1)$ is a point in spacetime. The density function $\Phi_{z_1}$ is given by
\begin{align*}
\Phi_{z_1}(\rho) := \left(4 \pi\right)^{- \frac{n}{2}} \rho^{2-n} \int_{\mathbb{R}^n}  \widetilde{u}\left(y,s_1 - \rho^2\right)   \hspace{1pt}e^{- \frac{\left|y-y_0 \right|^2}{4 \rho^2 }} \hspace{2pt} \mathrm{d}y,
\end{align*}where $\widetilde{u}$ is another function obtained by scaling and translating the original solution $u$. We note that various density functions have already been used in the study of heat flow of harmonic maps (see \cite{S88}) and mean curvature flow (see \cite{W05}). It is the first time that we apply a similar concept of density function to study regularity of mild solutions to the classical Patlak-Keller-Segel equation.\vspace{0.3pc}

To contradict the assumption $T_0 < \infty$, we are left to prove smallness of the last integral in (\ref{contradiction to get result}). Here our arguments are also different in 2D and higher dimensions. In 2D case, it shows in Lemma \ref{monomonicity of phi function} that  the density function $\Phi_{z_1}$ satisfies the ODE inequality:\begin{align*}
\dfrac{\mathrm{d}}{\mathrm{d}\rho} \Phi_{z_1}(\rho) \geq \left( 1 - \dfrac{M}{8\pi} \right) \dfrac{2}{\rho} \hspace{2pt}\Phi_{z_1} (\rho).
\end{align*}Due to this inequality, if $M < 8\pi$, then $\Phi_{z_1}(\rho)$ decays to $0$ as $\rho \rightarrow 0$. The decay rate  is of the order $\mathrm{O}\hspace{1pt}\big(\rho^{2 - \frac{M}{4 \pi}}\big)$ as $\rho \rightarrow 0$. Therefore, we can obtain the smallness of $\Phi_{z_1}$ on \text{\footnotesize $\left[ e_0^{- 1/2}, \sqrt{2}\hspace{1pt}e_0^{- 1/2}\right]$}, provided that  $e_0$ is large. The largeness of $e_0$ can be attained if $K'$ is large.  For the higher dimensional case,  we use the fact that $\dot{B}_{\infty, \infty}^{-2}$-norm of a function $w$ can be characterized by the superemum norm of $t \hspace{1pt}e^{t \Delta} w$ on $\mathbb{R}_+^{n+1}$. The density function $\Phi_{z_1}$ defined above immediately satisfies\begin{align*}\Phi_{z_1}\left(\rho\right) \hspace{2pt}\leq\hspace{2pt} \left\|\hspace{0.5pt} \widetilde{u}\left(\cdot,s_1-\rho^2\right) \right\|_{\dot{B}_{\infty,\infty}^{-2}} \hspace{2pt}\lesssim\hspace{2pt}   \epsilon_0.
\end{align*}The last estimate above is a consequence of Theorem A and (\ref{embedding}). In any case, we have smallness of the integral on the most-right-hand side of  (\ref{contradiction to get result}), which by our previous arguments, infers the $t^{-1}$-decay rate of the supremum norm of $u$ as $t \rightarrow \infty$. Readers may refer to Sect.\hspace{1pt}2 for the details of the proof. Notice that, for the classical 2D Patlak-Keller-Segel equation, the same $t^{-1}$-decay rate has been obtained in \cite{BDEF08} with several additional assumptions on the initial data. Our method here only depends on the finite-total-mass condition. Moreover, with smallness assumption on the $\dot{F}_{\infty, 2}^{-2}$-norm of initial data, our method can be applied to the higher dimensional case. We also want to point out that near $t = \infty$, the temporal decay rate of the supremum norm of $u$ can be improved from $t^{-1}$ to $t^{- \frac{n}{2}}$ if the spatial dimension $n \geq 3$. This result is attained in Sect.\hspace{3pt}3 by a bootstrap argument. Compared to the linear part of the equation (\ref{eqn_0_0_1}) which is governed by the heat equation on $\mathbb{R}^n$, our $t^{-\frac{n}{2}}$-decay rate is optimal. 

\subsubsection{Long-time asymptotic behavior in higher dimensions}\vspace{0.5pc}

In light of the works \cite{BDEF08, BDP06, CD14, FM16} on the classical 2D Patlak-Keller-Segel equation and the works \cite{GW02, GW05} by Gallay-Wayne on the 2D incompressible Navier-Stokes equation, to study the long-time behavior of global mild solutions to (\ref{eqn_0_0_1}) or (\ref{eqn_0_0_2}), it is convenient to write the equation under similarity variables. For any spatial dimension $n$ and function $u$, we define a new function $U$ by
\begin{align}\label{def_U}
U(\xi,\tau) := e^{\frac{n}{2} \tau} u \left( e^{\frac{\tau}{2}} \xi , e^{\tau} \right), \hspace{15pt}\text{where $(\xi, \tau) \in \mathbb{R}^n \times \mathbb{R}_+$.}
\end{align}
If $u$ is a solution of \hspace{1pt}(\ref{eqn_0_0_1}) or (\ref{eqn_0_0_2}), then $U$ satisfies the equation:
\begin{align}\label{eqn_1}
\partial_{\hspace{0.5pt}\tau}U  + f_n \nabla_\xi \cdot \big( U \nabla_\xi E_n * U  \big) = L_n U,  \hspace{15pt}\text{on $\mathbb{R}_+^{n+1}$, where $L_n := \Delta_\xi + \dfrac{1}{2}\hspace{1pt} \xi \cdot \nabla_\xi + \dfrac{n}{2}$.}
\end{align}
In (\ref{eqn_1}), the function $f_n$ is given by
\begin{eqnarray}\label{def_f_n}\label{f_n}f_n = f_n(\tau) := \exp\left\{ \left( 1 - \frac{n}{2}\right)\tau\right\}, \qquad \text{for all $\tau \in \mathbb{R}$.}\end{eqnarray}

To prove Statement 4 in Theorem \ref{thm_1}, we equivalently need to study  the strong $L^p$-convergence of the flow $\big\{U\left(\cdot, \tau\right)\big\}$ as $\tau \rightarrow \infty$. Here $p \in [1, \infty]$. The temporal decay rate that we have discussed in Sect.\hspace{1pt}1.3.1 are now crucial  for the sake of obtaining  the strong $L^p$-compactness of $U$. In light of (\ref{def_U}), a direct consequence of the $t^{- \frac{n}{2}}$-temporal decay rate of $u$ is the uniform boundedness of the rescaled solution $U$ on  $\mathbb{R}_+^{n + 1}$. Due to this uniform boundedness,  the rescaled solution $U\left(\cdot, \tau\right)$ and all its spatial derivatives are precompact in $L_{\mathrm{loc}}^\infty\left(\mathbb{R}^n\right)$ as $\tau \rightarrow \infty$. The proof is a standard application of regularity theory for parabolic equations and Arzel\` {a}-Ascoli theorem. Therefore, we only need to show that the flow $\big\{ U\left(\cdot, \tau\right)\big\}$ is spatially localized in $L^p\left(\mathbb{R}^n\right)$ with the spatial localization uniform for all $\tau$ large. Then, the strong $L^p$-convergence of the flow $\big\{U\left(\cdot, \tau\right)\big\}$ follows. By Duhamel principle, we decompose $U$ into $$U = U_1 + U_2, \hspace{15pt}\text{where $U_1(\tau) := S_n(\tau) \hspace{1pt}\big[\hspace{0.5pt}U\left(\cdot, 0\right)\big]$.}$$ In this decomposition, $S_n(\tau)$ is the semi-group generated by $L_n$. It can be shown that the flow $\big\{U_1\left(\cdot, \tau\right)\big\}$ is strongly $L^p$-compact for large $\tau$. Moreover, it converges to $M\hspace{1pt}\mathcal{G}_n$ strongly in $L^p\left(\mathbb{R}^n\right)$ as $\tau \rightarrow \infty$. We are therefore guided to study strong $L^p$-compactness of $U_2$. By an interpolation argument,  the $L^1$ and $L^\infty$-compactness of $U_2$ are of most importance.  One of our main estimates in Sect.\hspace{0.5pt}4.1 is to control the $L^1$-norm of $U_2$ as follows: \begin{align}\label{important est on L1 of U2}\int_{\mathbb{R}^n} \big|\hspace{1pt} U_2\left(\xi,\tau\right)\big|  \hspace{2pt} \mathrm{d}\xi \hspace{2pt}&\lesssim \hspace{2pt}   M^{1+\frac{1}{n}} \left[\hspace{1.5pt}\sup_{t \geq 0} \left( 1 + t \right)^{\frac{n}{2}}\big\|\hspace{1pt}u\left(\cdot, t\right)\big\|_{\infty}\right]^{1-\frac{1}{n}} \int_0^\tau \dfrac{f_n(s) e^{-\frac{\tau-s}{2}}}{ \left( 1 - e^{- (\tau -s )}\right)^{\frac{1}{2}}} \mathrm{d}s.
\end{align}As one can see, the $t^{- \frac{n}{2}}$-temporal decay rate of $u$ also plays a key role in the above estimate. Now we would like to emphasize that  when $n = 2$, the last integral on the right-hand side of (\ref{important est on L1 of U2}) is only uniformly bounded for all $\tau$ large. However, if the spatial dimension $n \geq 3$, then this integral indeed converges to $0$ as $\tau \rightarrow \infty$.  It is this difference that makes the long-time behaviors of the classical Patlak-Keller-Segel equation different in $2$D and higher dimensions. The assumption  $n \geq 3$ is also used in our study of the strong $L^\infty$-convergence of $U_2$ to $0$ as $\tau \rightarrow \infty$. Due to our previous arguments in this section, Statement 4 in Theorem \ref{thm_1} follows. Readers may refer to Sect.\hspace{0.5pt}4.1 for more details.\vspace{0.3pc}

If the initial density has a finite second moment, then the global solution to (\ref{eqn_1}) with this initial data may have a better convergence rate as $\tau \rightarrow \infty$ than the case when the initial density is merely $L^1$-integrable.  Generally a better convergence rate can be obtained via the entropy method. Readers may refer to \cite{AMTU01,BD00, GW05} for various applications of the entropy method on drift-diffusion type equations and $2$D incompressible Navier-Stokes equation. In the following, we briefly explain an application of the entropy method on the classical Patlak-Keller-Segel equation. Given a positive function $w$ on $\mathbb{R}^{n}$ with the total mass $M$, we define the free energy of $w$ associated with (\ref{eqn_1}) as follows:  \begin{align*}
\mathcal{H}\left[\hspace{0.5pt}w\hspace{0.5pt}\right] := \int_{\mathbb{R}^n} w \log \left(\dfrac{w}{\mathcal{G}_n}\right)  + \dfrac{1}{2} \hspace{1pt}f_n\hspace{1pt} \big| \nabla_\xi E_n * w \big|^2 \hspace{2pt}\mathrm{d}\xi - M \log M.
\end{align*}By carefully checking the validity of several integrations by parts, our solution $U$ to (\ref{eqn_1}) with a finite second moment satisfies the free energy identity: \begin{align*}\dfrac{\mathrm{d}}{\mathrm{d} \tau} \mathcal{H}   +  \int_{\mathbb{R}^n} \dfrac{n - 2}{4}  f_n \big| \nabla_\xi V \big|^2 + \dfrac{\mathcal{G}_n^{\hspace{0.5pt}2}}{U} \left| \nabla_\xi \left( \dfrac{U}{\mathcal{G}_n} \right) \right|^2 =   \int_{\mathbb{R}^n} f_n^{\hspace{1pt}2} \hspace{1pt}U \big| \nabla_\xi V \big|^2.
\end{align*}Here $\mathcal{H} = \mathcal{H}\left(\tau\right) = \mathcal{H}\left[\hspace{0.5pt}U\hspace{0.5pt}\right]$, $V = E_n * U$ and $n \geq 3$. The Gaussian probability density $\mathcal{G}_n$ is given in (\ref{gaussian density def}). Although $\mathcal{H}\left(\tau\right)$ is not a Lyapunov function since in general it is not monotonically decreasing with respect to $\tau$, we can still apply Gross' logarithmic Sobolev inequality in \cite{G75} and Csisz\'{a}r-Kullback inequality in \cite{C67,K67} to control the $L^1$ distance between $U$ and $M \mathcal{G}_n$. When $n \geq 4$, this method can give us the following optimal approximation of $U$ up to the order $\mathrm{O}(1)$: \begin{align*} U = M \hspace{0.5pt}\mathcal{G}_n + \mathrm{O}\big(e^{ - \frac{\tau}{2}}\big) \hspace{15pt}\text{near $\tau = \infty$ and under $L^1$-norm.}
\end{align*} However, for $n = 3$, the $L^1$-convergence rate of $U$ to $M \mathcal{G}_n$ obtained from this entropy method is only $e^{ - \frac{\tau}{4}}$ when $\tau \rightarrow \infty$. It is not optimal. This is one drawback of the entropy method. Another drawback is that we can only control the distance of $U$ with its $0$th order approximation via the entropy method. But as is well-known, if the second moment of initial data is finite, the solution to heat equation under similarity variables can be expanded near $\tau = \infty$ up to the order $e^{- \frac{\tau}{2}}$ with error bounded by $e^{- \tau}$. We expect a similar expansion to be true for the classical Patlak-Keller-Segel equation with error also bounded by $e^{- \tau}$.  Therefore, instead of using the entropy method, in Sect.\hspace{0.5pt}4.2, we utilize an integral representation of $U$, by which we provide a higher-order approximation of  the solution $U$ to (\ref{eqn_1}) with error also bounded by $e^{- \tau}$. See Propositions \ref{estimate of U1 up to second order} and \ref{optimal decay rate U2}. This approximation is optimal under the finite-second-moment assumption of the initial density. With the higher-order approximation of $U$, by (\ref{def_U}), Theorem \ref{thm_2} follows. We remark here that the higher-order  approximations of $u$ shown in Theorem \ref{thm_2} are different between the 3D case, the 4D case  and the  higher dimensional case with $n \geq 5$. The reason is due to   the advection term in (\ref{eqn_1}). For example, when $n = 3$, the advection term is of order $e^{ - \frac{\tau}{2}}$. Therefore, the effect from the advection term must be included in the approximation of $U$ when we expand it up to the order $e^{- \frac{\tau}{2}}$. But when $n \geq 4$, the advection term is a higher-order term. Its contribution is minor when we only expand $U$ up to the order $e^{- \frac{\tau}{2}}$. We also want to emphasize  two null structures associated with the classical Patlak-Keller-Segel equation. One is \begin{align*} \int_{\mathbb{R}^n} u \hspace{1pt}\partial_j v = 0, \hspace{15pt}\text{for all $j = 1, ..., n$.}
\end{align*} Here $v$ satisfies $- \Delta v = u$ on $\mathbb{R}^n$ with $u$ having sufficiently fast decay at spatial infinity. Another null structure is \begin{align*} \int_{\mathbb{R}^n} u_1 \hspace{1pt}\partial_j v_2 + u_2 \hspace{1pt}\partial_j v_1 = 0, \hspace{15pt}\text{for all $j = 1, ..., n$.}
\end{align*} Here for $k = 1, 2$, $v_k$ satisfies the Poisson equation $- \Delta v_k = u_k$  on $\mathbb{R}^n$. Meanwhile, $u_1$ and $u_2$ also have sufficiently fast decays at spatial infinity. With these two null structures obeyed by suitable functions, enough temporal decays can be gained near the temporal infinity which are crucial in our proof of Theorem \ref{thm_2}. Readers may refer to Sect.\hspace{0.5pt}4.2 for the details of the proof of Theorem \ref{thm_2}.

\subsubsection{Long-time asymptotic behavior in 2D}\vspace{0.5pc}

Compared to the higher dimensional case, it is difficult to study global dynamics of solutions to (\ref{eqn_0_0_2}) in 2D. The reason is that in 2D case, the last integral in (\ref{important est on L1 of U2}) is only uniformly bounded in $L^1$ for all $\tau \geq 0$. The compactness result of $S_2(\tau)$ at $\tau = \infty$ cannot be straightforwardly applied to obtain the strong $L^1$-compactness of a bounded global solution $U$ to \begin{align*}\partial_\tau U + \nabla_\xi \cdot\big(U \nabla_\xi E_2 * U\big) = L_2 U\hspace{15pt}\text{on $\mathbb{R}_+^{2+1}$}.\end{align*} By standard parabolic regularity result and Arzel\` {a}-Ascoli theorem, the uniform boundedness of $U$ can imply locally uniform convergence of $U$ as $\tau \rightarrow \infty$. To obtain the strong  $L^1$-compactness of $U$, we need \begin{align}\label{uniform smallness of integral of U} \sup_{\tau \hspace{1pt}\geq \hspace{1pt}0} \int_{B\mystrut^c_{R_\epsilon}} U\left(\cdot, \tau\right) < \epsilon,
\end{align}for arbitrary $\epsilon > 0$ and some $R_\epsilon > 0$ depending on $\epsilon$. In \cite{GW05}, a Carlen-Loss type estimate (see \cite{CL95}) has been used to study the global dynamics of 2D incompressible Navier-Stokes equation. In this work, the incompressibility condition of velocity satisfies Carlen-Loss type condition for the viscously damped conservation law and plays a key role in order to obtain a similar estimate as (\ref{uniform smallness of integral of U}) for vorticity.   However, the vector field $\nabla_\xi E_2* U$ in the classical $2$D Patlak-Keller-Segel equation satisfies \begin{align*} - \mathrm{div}_\xi \big(\nabla_\xi E_2* U \big) =  U \hspace{15pt}\text{on $\mathbb{R}^2$,}
\end{align*} which does not fulfill any Carlen-Loss type condition for the viscously damped conservation law if $U > 0$. The result in \cite{CL95} cannot be applied. The second approach to obtain a similar Carlen-Loss type estimate for parabolic type equations is to employ the Gaussian estimate of Aronson \cite{A68} for  fundamental solutions of parabolic equations. But Aronson's upper bound works only locally in time. One can extend Aronson's result globally in time. But in general (see \cite{K15}), a fundamental solution, denoted by $p(s, x; t, y)$, to a parabolic equation with bounded coefficients satisfies \begin{align*} p(s, x; t, y) \hspace{2pt}\lesssim\hspace{2pt} \dfrac{e^{c\hspace{0.5pt}\left(t - s\right)}}{(t - s)^{\frac{n}{2}}} e^{- \frac{\gamma  \left| x - y \right|^2}{t - s}}.
\end{align*}Here $c$ and $\gamma$ are positive constants. The exponential function $e^{c\left(t - s\right)}$ is harmful while we study long-time behavior of solutions to parabolic equations.\vspace{0.3pc}

To overcome the difficulty mentioned above, in Sect.\hspace{0.5pt}5.1, we introduce the quantity \begin{align}\label{def of tau 0 in introduction}
\tau_0\left(R\right) := \sup \left\{ T > 0 : \int_{B\mystrut^c_R} U\left(\xi,\tau\right) \mathrm{d}\xi  \hspace{2pt}<\hspace{2pt} \epsilon, \mbox{ for all } \tau \in [\hspace{1pt}0,T\hspace{1pt}] \right\}.
\end{align}Our goal is to show $\tau_0\left(R\right) = \infty$, provided that $\epsilon$ is small and $R$ is large. In the following arguments, we simply use $\tau_0$ to denote $\tau_0(R)$. Note that in the decomposition $U = U_1 + U_2$, where $U_1 = S_2\left(\tau\right) U_0$, the component $U_1$ is uniformly concentrated under the $L^1$-norm for all $\tau \geq 0$. The main arguments in Sect.\hspace{0.5pt}5.1 are devoted to estimating the $L^1$-norm of $U_2$ outside a large ball $B_R$.   Sufficiently we are led to estimate \begin{align}\label{est need to do in sect 5.1} \int_{B\mystrut^c_R} \mathrm{d} \xi \int_0^{\tau_0} \mathrm{d} s \int_{\mathbb{R}^2} \Omega\left(\xi, \eta, s\right) \hspace{2pt}\mathrm{d} \eta,\end{align}where\begin{align*} \Omega\left(\xi, \eta, s\right) :=  \dfrac{e^{-\frac{\tau_0 - s}{2}}}{ \left( 1 - e^{ - (\tau_0 - s)}\right)^2}  \hspace{1pt} U\left(\eta,s\right) \Big|\nabla_\eta E_2 * U \hspace{1pt}\Big| \left(\eta,s\right)  \left|\hspace{1pt}\xi -  e^{- \frac{\tau_0 - s }{2}} \eta \hspace{1pt}\right| e^{ -\frac{\big| \xi - e^{-\frac{\tau_0 - s}{2}} \eta \hspace{1pt} \big|^2}{4  \left(1 - e^{- (\tau_0-s)}\right)}}. 
\end{align*}Fixing an arbitrary $R_{1, \epsilon} > 0$, we decompose the integral domain of $\eta$ variable, i.e. $\mathbb{R}^2$, into $\omega_{1}^\epsilon \cup \omega_{2}^\epsilon$. Here $\omega_{1}^\epsilon$ denotes the ball $B_{R + R_{1, \epsilon}}$, while $\omega_{2}^\epsilon$ is the complement set of $\omega_{1}^\epsilon$ in $\mathbb{R}^2$. By (\ref{def of tau 0 in introduction}), the integral of $U(\cdot, s)$ on $\omega_2^\epsilon$ is of order at most $\epsilon$ for any $R_{1, \epsilon} > 0$ and $s \in [\hspace{1pt}0, \tau_0\hspace{1pt}]$. Utilizing the potential estimate, we can pick up $R_{1, \epsilon}$ sufficiently large, so that the $L^\infty$-norm of $\nabla E_2 * U$ on $\omega_2^\epsilon$ is of order at most $\epsilon^{1/6}$. Therefore, the integral \begin{align*} \int_{B\mystrut^c_R} \mathrm{d} \xi \int_0^{\tau_0}  \mathrm{d} s \int_{\omega_2^\epsilon} \Omega\left(\xi, \eta, s\right) \mathrm{d}\eta
\end{align*}is of order at most $\epsilon^{7/6}$, which is a higher-order term of $\epsilon$. When we integrate $\Omega\left(\xi, \eta, s\right)$ with respect to the variable $\eta$ on $\omega_1^\epsilon$, it is generally impossible to obtain any smallness from $U$ and $\nabla E_2 * U$ since in this case the radii $R$ and $R_{1, \epsilon}$ are large. Instead, we have to use the $\xi$ variable, more precisely the function \begin{align*}\left|\hspace{1pt}\xi -  e^{- \frac{\tau_0 - s }{2}} \eta \hspace{1pt}\right| e^{ -\frac{\big| \xi - e^{-\frac{\tau_0 - s}{2}} \eta \hspace{1pt} \big|^2}{4 \left(1 - e^{- (\tau_0-s)}\right)}}
\end{align*} in the integrand $\Omega\left(\xi, \eta, s\right)$. If there were  some positive constant $\sigma_\epsilon$ so that
\begin{align}\label{preferable estimate of xi}
\left|\hspace{1pt} \xi - e^{-\frac{\tau_0 - s}{2}} \eta \hspace{1pt}\right| \hspace{2pt}\geq\hspace{2pt} \sigma_\epsilon\hspace{1pt} |\hspace{0.5pt}\xi\hspace{0.5pt}| \hspace{15pt} \mbox{for $\xi \in B\mynewstrut^c_R$, $\eta \in \omega_1^\epsilon$ and $s \in [\hspace{1pt}0, \tau_0\hspace{1pt}]$,}
\end{align} then we could take $R$ large enough to obtain \begin{align*} \int_{B\mystrut^c_R} \mathrm{d} \xi \int_0^{\tau_0}  \mathrm{d} s \int_{\omega_1^\epsilon} \Omega\left(\xi, \eta, s\right) \mathrm{d}\eta < \epsilon^2.
\end{align*}However, (\ref{preferable estimate of xi}) is obviously false at $s = \tau_0$ in that the domain of $\xi $ variable and $\omega_1^\epsilon$ have non-empty intersection. But if we fix a $t_\epsilon > 0$ and push down the upper limit of $s$ variable in (\ref{est need to do in sect 5.1}) from $\tau_0$ to $\tau_0 - \log\left( 1 + t_\epsilon \right)$, then we can find a $R$ large enough so that the inequality in  (\ref{preferable estimate of xi}) holds for some $\sigma_\epsilon > 0$, all $\left(\xi, \eta\right) \in B\mynewstrut^c_R \times \omega_1^\epsilon$ and all $s \in [\hspace{1pt}0, \tau_0 - \log\left(1 + t_\epsilon\right)]$. See (\ref{id_7_12}). By this way, we can keep taking $R$ large enough such that  \begin{align*} \int_{B\mystrut^c_R} \mathrm{d} \xi \int_0^{\tau_0 - \log \left( 1 + t_\epsilon\right)}  \mathrm{d} s \int_{\omega_1^\epsilon} \Omega\left(\xi, \eta, s\right) \mathrm{d}\eta 
\end{align*}is a higher-order term of $\epsilon$. Our arguments are finally accomplished by choosing carefully a constant $t_\epsilon$ sufficiently small with which the integral \begin{align*} \int_{B\mystrut^c_R} \mathrm{d} \xi \int_{\tau_0 - \log\left(1 + t_\epsilon\right)}^{\tau_0 }  \mathrm{d} s \int_{\omega_1^\epsilon} \Omega\left(\xi, \eta, s\right) \mathrm{d}\eta 
\end{align*}is also a higher-order term of $\epsilon$. In summary, the $L^1$-compactness of $U$ near $\tau = \infty$ then follows by a bootstrap argument and this pushing-and-stretching spacetime approach.\vspace{0.3pc}

The next step to study the global dynamics of (\ref{eqn_0_0_2}) is to determine the $\omega$-limit set of $U$. To answer this question, the 2D case is also different from the higher dimensional case. For the higher dimensional case, since $U_2$ converges to $0$ strongly in $L^1$ as $\tau \rightarrow \infty$, we can easily infer from the decomposition $U = U_1 + U_2$ that the function $M \mathcal{G}_n$ is the unique element contained in the $\omega$-limit set of $U$. But in the 2D case, generally the $\omega$-limit set of $U$ may contain multiple elements. Associated with a sequence $\big\{\tau_k\big\}$ which diverges to $\infty$ as $k \rightarrow \infty$, the limit of $\big\{ U\left(\cdot, \tau_k + \tau\right)\big\}$ as $k \rightarrow \infty$, denoted by $U_\star\left(\tau\right)$, should be a function depending on the given variable $\tau \in \mathbb{R}$. Rewriting $U_\star$ under the $(x, t)$-variables by setting \begin{align*}u_\star\left(x, t\right) := \dfrac{1}{t} \hspace{1pt}U_\star\left(\dfrac{x}{\sqrt{t}}, \hspace{1pt}\log t \right),
\end{align*}in Sect.\hspace{0.5pt}5.2, we show that $u_\star$ solves the classical Patlak-Keller-Segel equation (\ref{eqn_0_0_2}) smoothly on $\mathbb{R}_+^{2+1}$. Meanwhile, $u_\star\left(\cdot, t\right)$ converges to $M \delta_0$ under the weak-$\mathrm{*}$ topology as $t \rightarrow 0$. Here, $\delta_0$ is the standard delta measure concentrated at $0$. Therefore, if $M < 8\pi$, then by the uniqueness result of Bedrossian-Masmoudi \cite{BM14} on the classical 2D Patlak-Keller-Segel equation with  measure-valued initial data, we obtain \begin{align*} u_\star\left(x, t\right) = \dfrac{1}{t} \hspace{1pt}G_M\left(\dfrac{x}{\sqrt{t}}\right),
\end{align*}where $G_M$ is the unique solution to (\ref{stationary solution of U in 2d}) with the total mass $M$ and a finite free energy. Equivalently, this shows that $U_\star = G_M$ if $M \in (0, 8\pi)$. We are left to show $M < 8 \pi$ if $U$ is uniformly bounded on $\mathbb{R}_+^{2+1}$. Our strategy is as follows. Firstly, we show that $u_\star$ has finite second moments for all $t > 0$. Moreover, since $u_\star\left(\cdot, t\right)$ converges to $M \delta_0$ under the weak-$\mathrm{*}$ topology as $t \rightarrow 0$, we can further induce \begin{align}\label{seco mom es of u star} \int_{\mathbb{R}^2} u_\star\left(x, t\right) |\hspace{1pt}x\hspace{1pt}|^2\hspace{2pt}\mathrm{d} x \hspace{2pt}\lesssim\hspace{2pt}\big( M + M^2\big)\hspace{0.5pt}t, \hspace{15pt}\text{for all $t > 0$.}
\end{align}If   the total mass $M = 8\pi$, then obviously 
\begin{align*}
\dfrac{\mathrm{d}}{\mathrm{d} t} \int_{\mathbb{R}^2}  u_\star(x,t) \hspace{1pt}|\hspace{1pt}x\hspace{1pt}|^2 \hspace{2pt}\mathrm{d}x = 4M \left( 1 - \frac{M}{8\pi} \right) = 0, \hspace{15pt} \mbox{for all } t > 0.
\end{align*}Hence, the left-hand side of (\ref{seco mom es of u star}) must be a constant, which in fact should be $0$ if we take $t \rightarrow 0$ on both sides of (\ref{seco mom es of u star}). By this way, we obtain a contradiction since $u_\star \equiv 0$ on $\mathbb{R}_+^{2 + 1}$. The total mass of $u_\star$ cannot be $8 \pi$. Finally we finish the proof of Theorem \ref{thm_1_2D} in Sect.\hspace{0.5pt}5.3 by showing the strong $L^\infty$-convergence of $U\left(\cdot, \tau\right)$ to $G_M$ as $\tau \rightarrow \infty$. 
\begin{rmk}In \cite{BM14, BZ15, SS02}, the classical 2D Patlak-Keller-Segel equation is shown to admit a local-in-time weak/mild solution if the initial data is a non-negative finite measure with all its atoms of mass strictly less than $8 \pi$. If atom of mass equals to $8 \pi$ at some location, our arguments above show that there can have no local-in-time solution smooth near $t = 0$ to the classical 2D Patlak-Keller-Segel equation. 
\end{rmk}

\subsection{\normalsize Notations}\label{sec_1_4}\vspace{0.5pc} 
\begin{itemize}
\item In this article, we simply use $X$ to denote a functional space $X\left(\mathbb{R}^n\right)$ in which all functions are defined on $\mathbb{R}^n$. The norm in $X$-space is denoted by $\| \cdot \|_X$;
\item Given $\Omega \subseteq \mathbb{R}^n$ and $p \geq 1$, the norm of $L^p\big(\hspace{1pt}\Omega; \hspace{1pt}\mathrm{d} x\hspace{1pt}\big)$ is denoted by $\| \cdot \|_{p, \mathrm{d} x; \hspace{1pt}\Omega}$, or simply by $\| \cdot \|_{p; \hspace{1pt}\Omega}$ when there is no confusion. Notice that we have also introduced the notation $\| \cdot \|_p$ which in fact is the norm $\| \cdot \|_{p; \hspace{1pt}\mathbb{R}^n}$;
\item In this article, $B_R(x)$ is the ball in $\mathbb{R}^n$ with center $x$ and radius $R$. $B_R$ is a simple notation for $B_R(0)$. Given $\rho > 0$ and $(x_0, t_0) \in \mathbb{R}_+^{n+1}$, we use $P_\rho(x_0,t_0)$ to denote the closed parabolic cylinder:
\begin{align*}
P_\rho(x_0,t_0) := \Big\{  \left(x,t\right) : |\hspace{0.5pt}x-x_0\hspace{0.5pt}| \hspace{1pt}\leq\hspace{1pt} \rho,\,\hspace{2pt} t_0 - \rho^2 \hspace{0.5pt}\leq\hspace{0.5pt} t \hspace{0.5pt}\leq\hspace{0.5pt} t_0 \Big\}.
\end{align*}For simplicity, we also use $P_\rho$ to denote $P_\rho(0, 0)$ in the following arguments. Aside from the closed parabolic cylinder $P_\rho(x_0, t_0)$, another closed cylinder $Q_R$ will be used in Sect.\hspace{0.5pt}5.2. It is defined as follows: \begin{align*} Q_R := \Big\{ (x, t) \in \mathbb{R}^{2 + 1} : | \hspace{0.5pt}x \hspace{0.5pt}| \leq R  \hspace{5pt}\text{and}\hspace{5pt} | \hspace{1pt}t \hspace{1pt}| \leq R^2 \Big\};
\end{align*}
\item We use $\Gamma_t\left(x\right)$ to denote the fundamental solution of the heat equation $\partial_t w = \Delta w$ on $\mathbb{R}^n$. On occasions that we need to substitute different values into the $t$-variable, we also use $\Gamma_n\left(x, t\right)$ to denote the heat kernel $\Gamma_t\left(x\right)$ interchangeably;
\item The notation $A \lesssim B$ means $A \leq c B$, where $c > 0$ is a universal constant depending only on the spatial dimension $n$. If there is a set of parameters, for example $M_1$, ..., $M_k$, then we use the notation $A \lesssim_{\hspace{1pt}M_1, ..., M_k} B$ to denote  $A \leq CB$, where $C > 0$ a constant depending on $M_1$, ..., $M_k$. Note that the constant $C$ can also depend on the spatial dimension $n$.
\end{itemize}

\section{\large Temporal decay estimate of global mild solutions}\label{sec_2}\vspace{0.5pc}

In this section, if $n \geq 3$, then we assume that the initial density $u_0$ has  small $\dot{F}_{\infty,2}^{-2} $-norm. If $n = 2$, then the total mass of $u_0$ is assumed to be in the interval $(0, 8\pi)$. Under these assmptions, we prove Statement 2 in Theorems \ref{thm_1_2D} and \ref{thm_1}.  Firstly, we estimate the supremum norm of $\nabla E_n * u$ in terms of $u$.

\begin{lem}\label{lem_2_1}
Suppose that $u \in L^1 \cap L^\infty$. Then it holds \begin{align}\label{estimate_grad_v}
\big\|\hspace{1pt} \nabla E_n * u \hspace{1pt}\big\|_\infty \hspace{2pt}\lesssim\hspace{2pt} \|\hspace{0.5pt} u \hspace{0.5pt}\|_1^{\frac{1}{n}} \hspace{1pt} \|\hspace{0.5pt} u \hspace{0.5pt}\|_{\infty}^{1- \frac{1}{n}}, \hspace{15pt}\text{for all $n \geq 2$.}
\end{align}
\end{lem}

\begin{proof}[\bf Proof] We assume that $u \not \equiv 0$ on $\mathbb{R}^n$. Otherwise, (\ref{estimate_grad_v}) holds trivially. For any $x \in \mathbb{R}^n$ and $R > 0$, we can estimate $\nabla E_n * u$ pointwisely as follows: 
\begin{align*}
\big|\hspace{1pt} \nabla E_n * u \hspace{1pt}\big|\left(x\right)  &\hspace{2pt}\lesssim\hspace{2pt} \int_{B_R(x)} \dfrac{\big|\hspace{0.5pt}u(z)\hspace{0.5pt}\big|}{|x - z|^{n-1}}   \hspace{2pt}\mathrm{d} z +  \int_{B\mystrut^c_R(x)} \dfrac{\big|\hspace{0.5pt}u(z) \hspace{0.5pt}\big|}{| x - z|^{n-1}} \hspace{2pt}\mathrm{ d} z \notag\\[2mm]
&\hspace{2pt}\lesssim\hspace{2pt} \| \hspace{0.5pt} u \hspace{0.5pt} \|_{\infty} \hspace{1pt} \big\| | x - z|^{1-n} \big\|_{1, \hspace{0.5pt}  \mathrm{d} z; B_R(x)}  +   \|\hspace{0.5pt} u \hspace{0.5pt}  \|_{\frac{3}{2}} \hspace{1pt} \big\|  | x - z|^{1-n} \big\|_{3, \hspace{0.5pt}\mathrm{d} z;  B\mystrut^c_R(x)}. \end{align*}Let  $R = \| u \|_1^{\frac{1}{n}} \|\hspace{0.5pt} u \hspace{0.5pt}\|_{\infty}^{- \frac{1}{n}}$ and take supreme over $x \in \mathbb{R}^n$ in the last estimate. It turns out \begin{align*}
 \big\| \nabla E_n * u \big\|_\infty \hspace{2pt}\lesssim\hspace{2pt} R \hspace{1pt}   \|\hspace{0.5pt} u \hspace{0.5pt} \|_{\infty} + R^{1- \frac{2}{3} n}  \|\hspace{0.5pt}u\hspace{0.5pt} \|_1^{\frac{2}{3}}\hspace{1pt}  \|\hspace{0.5pt} u \hspace{0.5pt} \|_{\infty}^{\frac{1}{3}}  \hspace{2pt}\lesssim\hspace{2pt} \| u \|_1^{\frac{1}{n}} \hspace{1pt}\| u \|_{\infty}^{1- \frac{1}{n}}.
\end{align*}The proof finishes.
\end{proof}

Now we consider the higher dimensional case with the additional assumption that $u_0 \in L^\infty$.
\begin{prop}\label{prop_2}
Suppose that the spatial dimension $n \geq 3$. There exists a positive constant $\epsilon_0 = \epsilon_0(n) $ such that for any  $u_0 \in L_+^1 \cap L^\infty$, if it satisfies \begin{align}\label{small assumption on initial}\big\| u_0 \big\|_{\dot{F}_{\infty,2}^{-2}} < \epsilon_0,\end{align} then (\ref{eqn_0_0_1}) has a unique global mild solution $u$ with initial data $u_0$ at $t = 0$. Moreover, we have
\begin{align}\label{estimate_prop_2}
\big\|\hspace{0.5pt} u(\cdot,t) \hspace{0.5pt}\big\|_{\infty} \hspace{1.5pt}\leq\hspace{1.5pt} \dfrac{5}{t}, \hspace{15pt}\text{for all $t > 0 $.}
\end{align}
\end{prop}
\begin{proof}[\bf Proof]
In light that $u_0$ satisfies (\ref{small assumption on initial}),  if we take $\epsilon_0$ suitably small, then Theorem A in the introduction yields the existence of a unique global mild solution, denoted by $u$, to (\ref{eqn_0_0_1}) with the initial data $u_0$ at $t = 0$. We assume that $u_0 \not\equiv 0$ on $\mathbb{R}^n$. Otherwise, $u \equiv 0$ on $\mathbb{R}_+^{n +1}$. The estimate (\ref{estimate_prop_2}) holds trivially. Notice that locally in time near $t = 0$, the $L^\infty$\hspace{0.5pt}-\hspace{0.5pt}norm of $u$ is finite in that $u_0 \in L^\infty$. Therefore, we can define 
\begin{align*}
T_0 := \sup \left\{ T > 0 : \big\|\hspace{0.5pt} u(\cdot,t) \hspace{0.5pt}\big\|_{\infty} \hspace{1.5pt}\leq\hspace{1.5pt} \dfrac{5}{t} \hspace{3pt} \mbox{ for all } t \in (\hspace{0.5pt}0,T\hspace{0.5pt}] \hspace{2pt}\right\}.
\end{align*}$T_0$ must be strictly positive or $\infty$. If $T_0 < \infty$, then for any $\delta \in (0, 1)$, there exists $\left(x_0,t_0\right) \in \mathbb{R}^n \times (\hspace{0.5pt}0,T_0\hspace{0.5pt}]$ depending on $\delta$ such that
\begin{align}\label{id_3_05}
\sup_{\rho \hspace{1pt} \in \hspace{1pt} \left[\hspace{0.5pt}0,\sqrt{t_0}\hspace{0.5pt}\right]} \left( \sqrt{t_0} - \rho \right)^2 \sup_{P_\rho\left(x_0,t_0 \right)} u \hspace{2pt}\geq \hspace{2pt} \sup_{\left(x, t \right) \hspace{0.5pt}\in\hspace{0.5pt} \mathbb{R}^n\times \left[\hspace{0.5pt}0, T_0 \hspace{0.5pt}\right]} \hspace{1pt}\sup_{\rho \hspace{1pt}\in \hspace{1pt} \left[\hspace{0.5pt}0,\sqrt{t}\hspace{0.5pt}\right]}  \big( \sqrt{t} - \rho \hspace{1pt}\big)^2 \sup_{P_\rho(x,t)} u - \delta.
\end{align}
 With $\left(x_0, t_0\right)$ in (\ref{id_3_05}), we translate and rescale $(u, v) = (u, E_n *u)$ by defining
\begin{align*}
&\widetilde{u}\left(y,s\right) := t_0 \hspace{1pt} u \left(x_0 + \sqrt{t_0} \hspace{1pt}y , t_0 + t_0 \hspace{0.5pt}s\right),  \\[1mm]
&\widetilde{v} \left(y,s\right)  := \hspace{10pt}v \left(x_0 + \sqrt{t_0} \hspace{1pt}y , t_0 + t_0 \hspace{0.5pt}s\right), \hspace{15pt}\text{for all $(y, s) \in \mathbb{R}^n \times [ -1, 0\hspace{0.5pt}]$.}
\end{align*}
Therefore, the left-hand side of (\ref{id_3_05}) satisfies
\begin{align}\label{id_2_06}
\sup_{\rho \hspace{1pt}\in\hspace{1pt} \left[\hspace{0.5pt}0,1\right]} (1 - \rho)^2 \sup_{P_\rho } \widetilde{u} \hspace{2pt}= \sup_{\rho \hspace{1pt}\in\hspace{1pt} \left[\hspace{0.5pt}0,\sqrt{t_0}\right]} \left( \sqrt{t_0} - \rho \hspace{0.5pt}\right)^2 \sup_{P_\rho(x_0,t_0)} u.
\end{align}In the remainder of the proof, we denote by $\rho_0$ a number in $[\hspace{0.5pt}0, 1)$ so that 
\begin{align}\label{id_2_07}
\sup_{\rho \hspace{1pt}\in\hspace{1pt} \left[\hspace{0.5pt}0,1\hspace{0.5pt}\right]} (1 - \rho)^2 \sup_{P_\rho} \widetilde{u} \hspace{2pt} \hspace{2pt}=\hspace{2pt} \left(1 - \rho_0\right)^2 e_0 \hspace{2pt}:=\hspace{2pt} (1 - \rho_0)^2 \sup_{P_{\rho_0}} \widetilde{u}.
\end{align}
Moreover, we let $\left(y_0,s_0\right)$ be a point in $P_{\rho_0}$ such that   $e_0 = \widetilde{u} \left(y_0,s_0\right)$. We claim that \begin{align}\label{claim of upper bound of e_0} e_0 \hspace{2pt}<\hspace{2pt} \dfrac{4}{\left(1 - \rho_0\right)^2}.\end{align} If on the contrary that (\ref{claim of upper bound of e_0}) fails, then  it holds \begin{align}\label{basic_ineq_relation} - \dfrac{(1 + \rho_0)^2}{4} \leq - \rho_0^2 - \dfrac{(1 - \rho_0)^2}{4} \leq  s_0 - \dfrac{1}{e_0} \leq s_0 + \dfrac{\tau}{e_0} \leq 0, \qquad \text{for any $\tau \in [-1, - e_0 \hspace{0.5pt}s_0\hspace{0.5pt}]$}.\end{align} Therefore, we can keep translating and rescaling $\left(\widetilde{u}, \widetilde{v}\right)$ by defining
\begin{align*}
&\widehat{u} \left(\xi,\tau\right) := \dfrac{1}{e_0} \widetilde{u} \left( y_0 + \frac{\xi}{\sqrt{e_0}}, \hspace{1pt}s_0 + \dfrac{\tau}{e_0} \right), \\[1mm]
&\widehat{v} \left(\xi,\tau\right) := \hspace{13pt}\widetilde{v} \left( y_0 + \frac{\xi}{\sqrt{e_0}} , s_0 + \frac{\tau}{e_0} \right), \hspace{15pt}\text{for all $\left(\xi, \tau\right) \in \mathbb{R}^n \times [ -1, - e_0 \hspace{0.5pt}s_0 \hspace{0.5pt}]$.}
\end{align*}
With the definitions of $\widehat{u}$ and $\widetilde{u}$, it follows   from (\ref{basic_ineq_relation}) that
\begin{align*}
\sup_{\mathbb{R}^n \times [-1, - e_0\hspace{0.5pt} s_0\hspace{0.5pt}]} \widehat{u} \hspace{3pt} \leq \hspace{3pt}  \dfrac{4}{e_0 (1 - \rho_0)^2} \sup_{x \hspace{1pt}\in\hspace{1pt}\mathbb{R}^n} \left( \sqrt{t_0} - \sqrt{t_0}\hspace{1pt} \dfrac{1 + \rho_0}{2} \right)^2 \sup_{P_{\sqrt{t_0} \hspace{1pt}\frac{1 + \rho_0}{2}}\left(x, \hspace{1pt}t_0\right)} u.
\end{align*}
In view of (\ref{id_3_05})--(\ref{id_2_07}), this estimate can be reduced to \begin{align}\label{id_2_08}
\sup_{\mathbb{R}^n \times [-1, - e_0\hspace{0.5pt} s_0\hspace{0.5pt}]} \widehat{u} &\hspace{3pt} \leq \hspace{3pt}  \dfrac{4}{e_0 (1 - \rho_0)^2} \sup_{\left(x, t \right) \hspace{0.5pt}\in\hspace{0.5pt} \mathbb{R}^n\times \left[\hspace{0.5pt}0, T_0 \hspace{0.5pt}\right]} \hspace{1pt}\sup_{\rho \hspace{1pt}\in \hspace{1pt} \left[\hspace{0.5pt}0,\sqrt{t}\hspace{0.5pt}\right]}  \big( \sqrt{t} - \rho \hspace{1pt}\big)^2 \sup_{P_\rho(x,t)} u \notag\\[2mm]
&\hspace{3pt}\leq \hspace{3pt}\dfrac{4}{e_0 (1 - \rho_0)^2}  \hspace{1pt}\big[\hspace{1pt}e_0 (1 - \rho_0)^2 + \delta\hspace{2pt}\big] \hspace{3pt}\leq\hspace{3pt}5.
\end{align} 
The last inequality above have used the assumption that (\ref{claim of upper bound of e_0}) fails and $\delta < 1$.\vspace{0.3pc}

By the invariance of the $\dot{F}_{\infty,\hspace{1pt}2}^{-2}$\hspace{0.5pt}-\hspace{0.5pt}norm under the scaling transformation (\ref{scaling}), Theorem A infers 
\begin{align*}
\sup_{\tau \hspace{1pt}\in\hspace{1pt} \left[-1, - e_0 \hspace{0.5pt}s_0\right]} \|\hspace{1pt} \widehat{u}\left(\cdot,\tau\right) \|_{\dot{F}_{\infty,2}^{-2}}  \hspace{2pt}\leq\hspace{2pt}\sup_{t \hspace{1pt}\in\hspace{1pt} [\hspace{0.5pt}0,T_0]} \|\hspace{0.5pt} u\left(\cdot,t\right) \hspace{0.5pt}\|_{\dot{F}_{\infty,2}^{-2}} \hspace{2pt}\leq \hspace{2pt} c\hspace{1pt}\epsilon_0.
\end{align*} 
In conjunction with the last two relationships in (\ref{embedding}), the above estimate induces
\begin{align}\label{id_2_08_1}
\sup_{\tau \hspace{1pt}\in\hspace{1pt} \left[-1,- e_0 \hspace{0.5pt}s_0\right]}  \|\hspace{1pt} \widehat{u}\left(\cdot,\tau\right) \|_{\dot{M}^1_{n/2}} \hspace{2pt}\lesssim\hspace{2pt}\epsilon_0,
\end{align}
where for any locally integrable function $w$ in $\dot{M}^1_{n/2}$, its $\dot{M}^1_{n/2}$\hspace{0.5pt}-\hspace{0.5pt}norm is given by
\begin{align*}
\| w \|_{\dot{M}^1_{n/2}} := \sup_{(x, r)\hspace{0.5pt}\in\hspace{0.5pt}\mathbb{R}_+^{n + 1}} r^{2-n} \int_{B_r(x)} |\hspace{0.5pt}w \hspace{0.5pt}|.
\end{align*} Since $\widehat{v} = E_n * \widehat{u}$, it follows from (\ref{id_2_08_1}) and the BMO property of the Newtonian potentials  in  \cite{G87}  that
\begin{align*}
\sup_{\tau \hspace{1pt}\in\hspace{1pt} \left[-1, -e_0 \hspace{0.5pt} s_0 \right]} \big\|\hspace{0.5pt} \widehat{v}\left(\cdot,\tau\right) \big\|_{\mathrm{BMO}} \hspace{2pt}\lesssim\hspace{2pt} \sup_{\tau \hspace{1pt}\in\hspace{1pt} \left[-1, - e_0 \hspace{0.5pt}s_0\right]} \big\|\hspace{0.5pt} \widehat{u}\left(\cdot,\tau\right) \big\|_{\dot{M}^1_{n/2}} \hspace{2pt}\lesssim\hspace{2pt} \epsilon_0.
\end{align*}
Hence, for any fixed $\xi \in \mathbb{R}^n$ and $\tau \in \left[-1,-e_0 \hspace{0.5pt}s_0\right]$, by using Morrey's ineqaulity, standard $W^{2, p}$-estimate for elliptic equations, John-Nirenberg inequality, (\ref{id_2_08}) and the last estimate on the BMO\hspace{0.5pt}-\hspace{0.5pt}norm of $\widehat{v}$, we can bound the gradient of $\widehat{v}$ as follows:
\begin{align*}
\left\|\hspace{0.5pt} \nabla_\xi \widehat{v}\left(\cdot,\tau\right)\hspace{0.5pt} \right\|_{\infty; \hspace{1pt}B_1(\xi)} \hspace{2pt}&\lesssim\hspace{3pt} \left\| \hspace{0.5pt}\nabla_\xi^2\hspace{0.5pt} \widehat{v}\left(\cdot, \tau\right) \hspace{0.5pt}\right\|_{2n; \hspace{1pt}B_1(\xi)} + \left\| \hspace{0.5pt}\nabla_\xi \widehat{v}\left(\cdot, \tau\right) \hspace{0.5pt}\right\|_{2n; \hspace{1pt}B_1(\xi)}\\[2mm]
&\lesssim \hspace{2pt}\big\|\hspace{0.5pt}\widehat{u}(\cdot, \tau)\hspace{0.5pt}\big\|_{2n; \hspace{1pt}B_2(\xi)} + \left\|\hspace{0.5pt} \widehat{v}\left(\cdot,\tau\right) - \dashint_{B_2(\xi)} \widehat{v} \hspace{1pt} \right\|_{2n; \hspace{1pt}B_2(\xi)}\\[2mm]
&\lesssim \hspace{2pt}1 + \left\|\hspace{1pt} \widehat{v}(\cdot,\tau) \hspace{0.5pt}\right\|_{\mathrm{BMO}}  \hspace{2pt}\lesssim\hspace{2pt}1.
\end{align*}
Here $\displaystyle \dashint_{B_2(\xi)} \widehat{v} $ denotes the average of $\widehat{v}$ on the ball $B_2(\xi)$. Since $(\xi,\tau)$ is arbitrary, we conclude that
\begin{align}\label{id_2_09}
\sup_{\tau \hspace{1pt}\in\hspace{1pt} \left[-1,-e_0 \hspace{0.5pt}s_0\right]} \big\|\hspace{0.5pt} \nabla_\xi \widehat{v}\left(\cdot,\tau\right) \big\|_{\infty} \hspace{2pt}\lesssim\hspace{2pt}1.
\end{align}

Notice that $P_1$ is contained in $\mathbb{R}^n \times \big[-1, -e_0 \hspace{0.5pt}s_0\big]$ in that $s_0 \leq 0$. By the equation of $u$ and the definition of $\widehat{u}$, it turns out
\begin{align}\label{id_2_10}
\partial_\tau \widehat{u} - \Delta_\xi \widehat{u} + \nabla_\xi \widehat{v} \cdot \nabla_\xi \widehat{u} - \widehat{u}^2 = 0 \qquad \mbox{on } P_1.
\end{align}
In light of (\ref{id_2_08}) and (\ref{id_2_09}), we can apply Theorem 7.36 in \cite{L05} to $(\ref{id_2_10})$ and  obtain
\begin{align}\label{id_2_11}
1 = \widehat{u}\left(0,0\right) \hspace{2pt}\lesssim &\hspace{2pt} \int_{P_{1}} \widehat{u}.
\end{align} 
On the other hand, we let $z_1 := \left(\hspace{0.5pt}y_0,s_1\right) := \left(\hspace{0.5pt}y_0,s_0 + e_0^{-1}\right)$ and define the backward heat kernel
\begin{align*}
\widetilde{\Gamma}_{n; \hspace{1pt}z_1}\big(y, s \big)  := \dfrac{1}{\left( 4\pi \left(s_1 - s\right) \right)^{\frac{n}{2}}} \hspace{1pt}e^{- \frac{\left|y-y_0 \right|^2}{4 \left(s_1 - s \right)}} \hspace{20pt}\text{for $s < s_1$.}
\end{align*}
Then the definition of $\widehat{u}$ yields
\begin{align}\label{id_2_15}
\int_{P_1} \widehat{u} &\hspace{2pt}=\hspace{2pt}e_0^{n/2}\int_{P_{e_0^{-1/2}}\left(y_0, s_0\right)} \widetilde{u} \hspace{2pt}\lesssim\hspace{2pt} \int_{P_{e_0^{-1/2}}\left(y_0,s_0\right)} \widetilde{u} \hspace{2pt} \widetilde{\Gamma}_{n; \hspace{1pt}z_1}   \hspace{2pt}\lesssim\hspace{2pt} \int_{e_0^{-1/2}}^{\sqrt{2} \hspace{0.5pt}e_0^{-1/2}} \dfrac{\Phi_{z_1}(\rho)}{\rho} \hspace{2pt}\mathrm{d}\rho,
\end{align}
where in the last estimate,
\begin{align}\label{def_phi_z_1}
\Phi_{z_1}(\rho) := \rho^2 \int_{\mathbb{R}^n} \widetilde{u}\left(\cdot,s\right) \hspace{1pt} \widetilde{\Gamma}_{n; \hspace{1pt}z_1}\left(\cdot,s\right)\hspace{2pt} \bigg|_{s = s_1 - \rho^2}.
\end{align}
Notice that the norm in $\dot{B}_{\infty,\infty}^{-2}$ can be characterized by the heat kernel as follows:
\begin{align*}
\|\hspace{0.5pt} w \hspace{0.5pt}\|_{\dot{B}_{\infty,\infty}^{-2}} = \hspace{2pt}\sup_{t>0} \hspace{2pt}t \left\|\hspace{1pt} e^{t\Delta} w \hspace{1pt}\right\|_{\infty}, \qquad \text{for any $w \in \dot{B}_{\infty,\infty}^{-2}$.}
\end{align*}
Therefore, by (\ref{id_2_08_1}) and the scaling invariance of $\dot{B}_{\infty, \infty}^{-2}$\hspace{0.5pt}-\hspace{0.5pt}norm, it satisfies
\begin{align}\label{id_2_16}
\Phi_{z_1}\left(\rho\right) \hspace{2pt}\leq\hspace{2pt} \left\|\hspace{0.5pt} \widetilde{u}\left(\cdot,s_1-\rho^2\right) \right\|_{\dot{B}_{\infty,\infty}^{-2}} \hspace{2pt}\lesssim\hspace{2pt}   \epsilon_0, \qquad \text{for any $\rho \in \left[ e_0^{-1/2} , \sqrt{2} \hspace{1pt}e_0^{-1/2} \right]$.}
\end{align}
Applying (\ref{id_2_16}) to the most-right-hand side of (\ref{id_2_15}) infers
\begin{align}\label{id_2_17}
\int_{P_1} \widehat{u} \hspace{2pt}\lesssim\hspace{2pt}\epsilon_0.
\end{align}
Combining (\ref{id_2_11}) and (\ref{id_2_17}) yields $\epsilon_0 \gtrsim 1$. However, this is impossible if we choose  $\epsilon_0$ suitably small. Hence, (\ref{claim of upper bound of e_0}) follows, provided that $\epsilon_0$ is small.
In light of (\ref{id_3_05})--(\ref{claim of upper bound of e_0}), it holds
\begin{align*}
T_0 \sup_{x \hspace{1pt}\in\hspace{1pt}\mathbb{R}^n} u\left(x,T_0\right) \hspace{1.5pt}\leq\hspace{1.5pt} (1 - \rho_0)^2 e_0 + \delta \hspace{1.5pt}<\hspace{1.5pt} 4 + \delta.
\end{align*}
Since $\delta$ is an arbitrary number in $(0, 1)$, the last estimate further implies \begin{align*}
T_0 \sup_{x \hspace{1pt}\in\hspace{1pt}\mathbb{R}^n} u \left(x, T_0\right) \hspace{1.5pt}\leq\hspace{1.5pt} 4.
\end{align*}
This is a contradiction to our definition of $T_0$. Therefore, $T_0 = \infty$  and the proof is completed.
\end{proof}

In the next, we consider the $2$D case with the additional assumption that $u_0 \in L^\infty $.

\begin{prop}\label{prop_3}
Suppose that $n = 2$. For any $u_0 \in L_+^1  \cap L^\infty$ with $M  = \| u_0\|_1  \in (0, 8\pi)$, the equation (\ref{eqn_0_0_2})  has a unique global mild solution $u$ with initial data $u_0$ at $t = 0$. Moreover, there exists a constant $K$ depending only on $M$ such that the $L^\infty$\hspace{0.5pt}-\hspace{0.5pt}norm of $u$ satisfies
\begin{align}\label{estimate_prop_3}
\big\|\hspace{0.5pt} u(\cdot,t) \hspace{0.5pt}\big\|_{\infty} \hspace{1.5pt}\leq\hspace{1.5pt} \dfrac{K}{t}, \qquad \text{for all $t > 0 $.}
\end{align}The constant $K = K(M)$ is monotonically increasing with respect to the total mass $M \in (0, 8 \pi)$. It diverges to $\infty$ as $M$ approaches $8 \pi$.
\end{prop}\vspace{0.1pc}

\begin{proof}[\bf Proof]
By $M < 8\pi$, the equation (\ref{eqn_0_0_2}) admits a unique global mild solution $u$ with the initial data $u_0$ at $t = 0$. See \cite{W18}. In light that $u_0 \in L^\infty$, locally in time near $t = 0$, the $L^\infty$\hspace{0.5pt}-\hspace{0.5pt}norm of $u$ is finite. Therefore, we can define 
\begin{align*}
T_0 := \sup \left\{ T > 0 : \big\|\hspace{0.5pt} u(\cdot,t) \hspace{0.5pt}\big\|_{\infty} \hspace{1.5pt}\leq\hspace{1.5pt} \dfrac{K}{t} \hspace{3pt} \mbox{ for all } t \in (\hspace{0.5pt}0,T\hspace{0.5pt}] \hspace{2pt}\right\},
\end{align*}
where $K \geq 32$ is a constant to be determined later. $T_0$ must be positive or $\infty$. If $T_0 < \infty$, then similarly as in the proof of Proposition \ref{prop_2}, for any $\delta \in (0, 1)$, there exists $\left(x_0,t_0\right) \in \mathbb{R}^2 \times (\hspace{0.5pt}0,T_0\hspace{0.5pt}]$ depending on $\delta$ such that
\begin{align}\label{id_5_01}
\sup_{\rho \hspace{1pt} \in \hspace{1pt} \left[\hspace{0.5pt}0,\sqrt{t_0}/2\hspace{0.5pt}\right]} \left( \sqrt{t_0}/2 - \rho \right)^2 \sup_{P_\rho\left(x_0,t_0 \right)} u \hspace{2pt}\geq\hspace{2pt} \sup_{\left(x, t\right) \hspace{1pt}\in\hspace{1pt} \mathbb{R}^2 \times \left[\hspace{0.5pt}0, T_0 \hspace{0.5pt}\right]} \hspace{1pt}\sup_{\rho \hspace{1pt}\in \hspace{1pt} \left[\hspace{0.5pt}0,\sqrt{t}/2\hspace{0.5pt}\right]}  \big( \sqrt{t}/2 - \rho \hspace{1pt}\big)^2 \sup_{P_\rho(x,t)} u - \delta.
\end{align}
With $\left(x_0, t_0\right)$ in (\ref{id_5_01}), we translate and rescale $u$ as follows: 
\begin{align*}
\widetilde{u}\left(y,s\right) := t_0 \hspace{1pt} u \left(x_0 + \sqrt{t_0} \hspace{1pt}y , t_0 + t_0 \hspace{0.5pt}s\right), \quad\text{for all $(y, s) \in \mathbb{R}^2 \times [-1, 0\hspace{0.5pt}]$.}
\end{align*}
Therefore, the left-hand side of (\ref{id_5_01}) satisfies
\begin{align}\label{id_5_02}
\sup_{\rho \hspace{1pt}\in\hspace{1pt} \left[\hspace{0.5pt}0,1/2\right]} (1/2 - \rho)^2 \sup_{P_\rho } \widetilde{u} \hspace{2pt}= \sup_{\rho \hspace{1pt}\in\hspace{1pt} \left[\hspace{0.5pt}0,\sqrt{t_0}/2\right]} \left( \sqrt{t_0}/2 - \rho \hspace{0.5pt}\right)^2 \sup_{P_\rho(x_0,t_0)} u.
\end{align}
We then denote by $\rho_0$ a number in $[\hspace{0.5pt}0, 1/2)$ so that 
\begin{align}\label{id_5_03}
\sup_{\rho \hspace{1pt}\in\hspace{1pt} \left[\hspace{0.5pt}0,1/2\hspace{0.5pt}\right]} (1/2 - \rho)^2 \sup_{P_\rho} \widetilde{u} \hspace{2pt} \hspace{2pt}=\hspace{2pt} \left(1/2 - \rho_0\right)^2 e_0 \hspace{2pt}:=\hspace{2pt} (1/2 - \rho_0)^2 \sup_{P_{\rho_0}} \widetilde{u}.
\end{align}
Moreover, we let $\left(y_0,s_0\right)$ be a point in $P_{\rho_0}$ such that   $e_0 = \widetilde{u} \left(y_0,s_0\right)$. We claim that
\begin{align}\label{id_5_04}
e_0 \hspace{2pt}<\hspace{2pt} \dfrac{K}{2\left(1  - 2 \rho_0\right)^2}, \hspace{10pt}\text{provided that $K$ is suitably large.}
\end{align}
If on the contrary that (\ref{id_5_04}) fails, then by $K \geq 32$, it holds
\begin{align}\label{id_5_05}
- \dfrac{(1/2 + \rho_0)^2}{4} \leq - \rho_0^2 - \dfrac{(1/2 - \rho_0)^2}{4} \leq - \rho_0^2 - \dfrac{2\left(1  - 2 \rho_0\right)^2}{K} \leq s_0 - \dfrac{1}{e_0} \leq s_0 + \dfrac{\tau}{e_0} \leq 0,
\end{align}
for any $\tau \in [-1, - e_0 \hspace{0.5pt}s_0\hspace{0.5pt}]$. Then we further translate and rescale $\widetilde{u}$ by defining
\begin{align*}
\widehat{u} \left(\xi,\tau\right) := \dfrac{1}{e_0} \widetilde{u} \left( y_0 + \frac{\xi}{\sqrt{e_0}}, \hspace{1pt}s_0 + \dfrac{\tau}{e_0} \right), \hspace{5pt}\qquad \text{for all $\left(\xi, \tau\right) \in \mathbb{R}^2 \times [ -1, - e_0 \hspace{0.5pt}s_0 \hspace{0.5pt}]$.}
\end{align*}
By the definition of $\widehat{u}$ and $\widetilde{u}$, it follows from (\ref{id_5_05}) that
\begin{align*}
\sup_{\mathbb{R}^2 \times [-1, - e_0\hspace{0.5pt} s_0\hspace{0.5pt}]} \widehat{u} \hspace{3pt} \leq \hspace{3pt}  \dfrac{4}{e_0 (1/2 - \rho_0)^2} \sup_{x \hspace{1pt}\in\hspace{1pt}\mathbb{R}^2} \left( \frac{\sqrt{t_0}}{2} - \sqrt{t_0}\hspace{1pt} \dfrac{1/2 + \rho_0}{2} \right)^2 \sup_{P_{\sqrt{t_0} \hspace{1pt}\frac{1/2 + \rho_0}{2}}\left(x, \hspace{1pt}t_0\right)} u.
\end{align*}
In view of (\ref{id_5_01})--(\ref{id_5_03}), this estimate can be reduced to \begin{align}\label{id_5_06}
\sup_{\mathbb{R}^2 \times [-1, - e_0\hspace{0.5pt} s_0\hspace{0.5pt}]} \widehat{u} &\hspace{3pt} \leq \hspace{3pt}  \dfrac{4}{e_0 (1/2 - \rho_0)^2} \sup_{\left(x, t\right) \hspace{1pt}\in\hspace{1pt} \mathbb{R}^2 \times \left[\hspace{0.5pt}0, T_0 \hspace{0.5pt}\right]} \hspace{1pt}\sup_{\rho \hspace{1pt}\in \hspace{1pt} \left[\hspace{0.5pt}0,\sqrt{t}/2\hspace{0.5pt}\right]}  \big( \sqrt{t}/2 - \rho \hspace{1pt}\big)^2 \sup_{P_\rho(x,t)} u \notag\\[2mm]
&\hspace{3pt}\leq \hspace{3pt}\dfrac{4}{e_0 (1/2 - \rho_0)^2}\hspace{1pt}\big[\hspace{1pt} \left(1/2 - \rho_0\right)^2 e_0 + \delta \hspace{2pt}\big] \hspace{3pt}\leq\hspace{3pt}5.
\end{align}The last inequality above have used the assumption that (\ref{id_5_04}) fails, $K \geq 32$ and $\delta < 1$.
Using (\ref{id_5_06}) and the scaling invariance of the $L^1$-norm in 2D, we can deduce from (\ref{estimate_grad_v}) that 
\begin{align}\label{id_5_07}
\sup_{\tau \hspace{1pt}\in\hspace{1pt} \left[-1,-e_0 \hspace{0.5pt}s_0\right]} \big\|\hspace{0.5pt} \nabla E_2 * \widehat{u}\left(\cdot,\tau\right) \big\|_{\infty} \hspace{2pt}\lesssim\hspace{2pt}1.
\end{align}
By the equation of $u$, it turns out
\begin{align}\label{id_5_08}
\partial_\tau \widehat{u} - \Delta_\xi \widehat{u} + \big(\hspace{1pt}\nabla_\xi E_2 * \widehat{u} \hspace{1pt}\big)\cdot \nabla_\xi \widehat{u} - \widehat{u}^2 = 0 \quad\quad \mbox{on } P_1.
\end{align}
In view of (\ref{id_5_06})--(\ref{id_5_08}), the local maximum estimate yields
\begin{align}\label{id_5_09}
1 = \widehat{u}\left(0,0\right) \hspace{2pt}\lesssim &\hspace{2pt} \int_{P_{1}} \widehat{u}.
\end{align}
Similarly as in (\ref{id_2_15}), we can  let $z_1 := \left(\hspace{0.5pt}y_0,s_1\right) := \left(\hspace{0.5pt}y_0,s_0 + e_0^{-1}\right)$ and obtain
\begin{align}\label{id_5_10}
\int_{P_1} \widehat{u} \hspace{2pt}\lesssim\hspace{2pt} \int_{e_0^{-1/2}}^{\sqrt{2} \hspace{0.5pt}e_0^{-1/2}} \dfrac{\Phi_{z_1}(\rho)}{\rho} \hspace{2pt}\mathrm{d}\rho.
\end{align}Here $\Phi_{z_1}(\rho)$ is  same as (\ref{def_phi_z_1}) with $n = 2$ here. Direct calculations (see Lemma \ref{monomonicity of phi function}) induce
\begin{align}\label{id_5_11}
\dfrac{\mathrm{d}}{\mathrm{d}\rho} \Phi_{z_1}(\rho) \geq \left( 1 - \dfrac{M}{8\pi} \right) \dfrac{2}{\rho} \hspace{2pt}\Phi_{z_1} (\rho) \quad \mbox{for all} \hspace{4pt} \rho \in (0,\rho_*). \hspace{5pt}\text{Here $\rho_* := \left( 1 + s_0 + e_0^{-1} \right)^{\frac{1}{2}}$.}
\end{align}Equivalently it follows
\begin{align*}
\dfrac{1}{\Phi_{z_1}(\rho)} \dfrac{\mathrm{d}}{\mathrm{d}\rho} \Phi_{z_1}(\rho) \hspace{2pt}\geq\hspace{2pt} \left( 1 - \dfrac{M}{8\pi} \right) \dfrac{2}{\rho}.
\end{align*}
Integrating the above inequality on both sides over the interval $[\hspace{0.5pt}\rho,\rho_*\hspace{0.5pt}]$, we get
\begin{align*}
\log \dfrac{\Phi_{z_1}(\rho_*)}{\Phi_{z_1}(\rho)} \hspace{2pt}\geq\hspace{2pt} 2\left( 1 - \dfrac{M}{8\pi} \right) \log \dfrac{\rho_*}{\rho}, \hspace{10pt} \text{for all $\rho \leq \rho_*$,}
\end{align*}
which furthermore implies
\begin{align}\label{id_5_12}
\Phi_{z_1}(\rho) \hspace{2pt}\leq\hspace{2pt} \Phi_{z_1}(\rho_*) \left( \dfrac{\rho}{\rho_*}\right)^{2\left( 1 - \frac{M}{8\pi} \right)} \hspace{2pt}\lesssim\hspace{2pt} M \rho^{2\left( 1 - \frac{M}{8\pi} \right)}, \hspace{10pt} \text{for all $\rho \leq \rho_*$.}
\end{align}
In the last estimate, we have used 
\begin{align*}
\rho_* = \left( 1 + s_0 + e_0^{-1} \right)^{\frac{1}{2}} \geq \big( 1 - \rho_0^2 \hspace{1pt}\big)^{\frac{1}{2}} \geq \dfrac{\sqrt{3}}{2}
\end{align*}
and the bound
\begin{align*}
\Phi_{z_1}(\rho_*) = \dfrac{1}{4\pi} \int_{\mathbb{R}^2} \widetilde{u} \left( y , -1 \right) e^{-\frac{|y-y_0|^2}{4 \rho_*^2}} \mathrm{d}y \hspace{2pt}\leq\hspace{2pt} \dfrac{M}{4\pi}.
\end{align*}

Applying (\ref{id_5_12}) to the most-right-hand side of (\ref{id_5_10}) infers
\begin{align}\label{id_5_13}
\int_{P_1} \widehat{u} \hspace{2pt} \lesssim \hspace{2pt} M \int_{e_0^{-1/2}}^{\sqrt{2} \hspace{0.5pt}e_0^{-1/2}} \rho^{1 - \frac{M}{4\pi}}\hspace{1pt}\mathrm{d}\rho \hspace{2pt}\lesssim\hspace{2pt} \dfrac{M}{2 - \frac{M}{4\pi}} \left( \dfrac{1}{e_0}\right)^{1 - \frac{M}{8\pi}} \hspace{2pt}\lesssim\hspace{2pt} \dfrac{M}{2 - \frac{M}{4\pi}} \left( \dfrac{1}{K}\right)^{1 - \frac{M}{8\pi}}.
\end{align}
By choosing $K$ depending on $M$ sufficiently large, (\ref{id_5_09}) and (\ref{id_5_13}) then yield a contradiction. Therefore, (\ref{id_5_04}) holds. In light of (\ref{id_5_01})--(\ref{id_5_04}), we have
\begin{align*}
T_0 \sup_{x \hspace{1pt}\in\hspace{1pt}\mathbb{R}^2} u\left(x,T_0\right) \hspace{1.5pt}\leq\hspace{1.5pt} (1  - 2\rho_0)^2 e_0  + 4\delta \hspace{1.5pt}<\hspace{1.5pt} \frac{K}{2} + 4\delta.
\end{align*}Since $\delta$ is an arbitrary number in $(0, 1)$, the above estimate further implies \begin{align*} T_0 \sup_{x \hspace{1pt}\in\hspace{1pt}\mathbb{R}^2} u\left(x,T_0\right) \hspace{1.5pt}\leq\hspace{1.5pt} \dfrac{K}{2}.
\end{align*}
This estimate contradicts our definition of $T_0$. Therefore, $T_0 = \infty$ and the proof is completed.
\end{proof}

In the next, we drop the finiteness assumption on the $L^\infty$-norm of  $u_0$.\vspace{0.3pc}

\begin{proof}[\bf Proof of $1 \Rightarrow 2$ in Theorems \ref{thm_1_2D} and \ref{thm_1}]\ \vspace{0.8pc}
\\  
Fixing a $L > 0$, we denote by $u_{0; L} $ the function $\min \big\{ u_0 , L \big\}$.  By Propositions \ref{prop_2} and \ref{prop_3}, there is a unique global mild solution, denoted by $u_L$,  to either (\ref{eqn_0_0_1}) for the higher dimensional case or (\ref{eqn_0_0_2}) for the 2D case. The solution $u_L$ takes the initial data $u_{0; L}$ at $t = 0$. Moreover, by the decay estimates in (\ref{estimate_prop_2}) and (\ref{estimate_prop_3}), the solution $u_L$ satisfies
\begin{align}\label{id_2_18}
\big\|\hspace{0.5pt}u_L\left(\cdot, t\right)\hspace{1pt}\big\|_\infty \hspace{2pt}\leq\hspace{2pt} 
\dfrac{K'}{t}, \hspace{15pt}\text{for all $t > 0$.}
\end{align}
Here $K' = 5$ if $n \geq 3$. Note that the total mass of $u_{0; L}$ cannot be greater than the total mass of $u_0$. If $n = 2$, by the monotonicity of the constant $K$ in Proposition \ref{prop_3} with respect to the total mass of initial density, the constant $K'$ can be independent of $L$ and equals to $K(M)$ with $M$ being the total mass of $u_0$.  Using (\ref{id_2_18}) and standard regularity theories for parabolic and elliptic equations, we can extract a subsequence, still denoted by $L$, such that $$\big(u_L, \nabla E_n *  u_L\hspace{1pt}\big) \longrightarrow \big(u, \nabla E_n * u\hspace{1pt}\big), \hspace{15pt}\text{ pointwisely on $\mathbb{R}^{n+1}_+$ as $L \rightarrow \infty$.}$$ The function $u$ is a non-negative function on $\mathbb{R}_+^{n+1}$.  Moreover, it is in fact a smooth solution to either (\ref{eqn_0_0_1}) or (\ref{eqn_0_0_2}) on $\mathbb{R}^{n+1}_+$. In the next, we show that $u$ is the unique global mild solution to (\ref{eqn_0_0_1}) or (\ref{eqn_0_0_2}) with initial data $u_0$. \vspace{0.3pc}

 Since $u_L$ is a global mild solution to either (\ref{eqn_0_0_1}) or (\ref{eqn_0_0_2}), for any $t > 0$ and $\epsilon \in (0, t/2)$, it satisfies  \begin{align}\label{rep of error epsilon}
\Big|\hspace{1pt} \mathrm{Err}_\epsilon\big[ u_L; u_{0; L}\big]\hspace{1pt} \Big|  =  \dfrac{1}{2} \left| \int_0^\epsilon \int_{\mathbb{R}^n} u_L(z, s) \hspace{1pt} \Gamma_n(x - z, t - s) \hspace{1pt} \nabla E_n * u_L \hspace{1pt}\bigg|_{ (z ,s)} \cdot \dfrac{z - x}{ t - s }  \hspace{2pt}\mathrm{d} z \hspace{1pt} \mathrm{d} s \right|. 
\end{align}Here $\Gamma_n\left(x, t\right) := \Gamma_t\left(x\right)$ is the standard heat kernel on $\mathbb{R}^n$. Given $\epsilon \geq 0$ and two functions $w$ and $w_0$, the notation $\mathrm{Err}_\epsilon[w; w_0]$ is defined as follows: \begin{align*}\mathrm{Err}_\epsilon\hspace{0.5pt}\big[\hspace{0.5pt} w; w_0\hspace{0.5pt}\big] &\hspace{2pt}:=\hspace{2pt} w(x,t) - \int_{\mathbb{R}^n} w_0(z) \hspace{1pt}\Gamma_n(x - z, t)  \hspace{2pt}\mathrm{d} z \\[2mm]
&\hspace{2pt}+\hspace{2pt}  \dfrac{1}{2} \int_\epsilon^t \int_{\mathbb{R}^n} w(z, s) \hspace{1pt} \Gamma_n(x - z, t - s) \hspace{1pt} \nabla E_n * w \hspace{1pt}\bigg|_{ (z ,s)} \cdot \dfrac{z - x}{ t - s }  \hspace{2pt}\mathrm{d} z \hspace{1pt} \mathrm{d} s.
\end{align*} Taking $L \rightarrow \infty$ on both sides of (\ref{rep of error epsilon}), we obtain  \begin{align*}
\Big|\hspace{1pt}\mathrm{Err}_\epsilon \big[ u; u_0 \big]\hspace{1pt}\Big|  \hspace{2pt}\leq \hspace{2pt}\liminf_{L \rightarrow \infty}\left| \int_0^\epsilon \int_{\mathbb{R}^n} u_L(z, s) \hspace{1pt} \Gamma_n(x - z, t - s) \hspace{1pt} \nabla E_n * u_L \hspace{1pt}\bigg|_{(z ,s)} \cdot \dfrac{z - x}{ t - s }  \hspace{2pt}\mathrm{d} z \hspace{1pt} \mathrm{d} s \right|,
\end{align*}which furthermore infers  \begin{align*}
\Big|\hspace{1pt}\mathrm{Err}_0 \big[ u; u_0 \big]\hspace{1pt}\Big| &\hspace{2pt}\leq \hspace{3pt}\left| \int_0^\epsilon \int_{\mathbb{R}^n} u(z, s) \hspace{1pt} \Gamma_n(x - z, t - s) \hspace{1pt} \nabla E_n * u \hspace{1pt}\bigg|_{(z ,s)} \cdot \dfrac{z - x}{ t - s }  \hspace{2pt}\mathrm{d} z \hspace{1pt} \mathrm{d} s \right| \notag\\[2mm]
&\hspace{2pt}+ \hspace{3pt}\liminf_{L \rightarrow \infty}\left| \int_0^\epsilon \int_{\mathbb{R}^n} u_L(z, s) \hspace{1pt} \Gamma_n(x - z, t - s) \hspace{1pt} \nabla E_n * u_L \hspace{1pt}\bigg|_{(z ,s)} \cdot \dfrac{z - x}{ t - s }  \hspace{2pt}\mathrm{d} z \hspace{1pt} \mathrm{d} s \right|. 
\end{align*}Now we integrate both sides of the above estimate on $\mathbb{R}^n$. By Fatou's lemma, it turns out \begin{align*}
\int_{\mathbb{R}^n}\Big| \mathrm{Err}_0 \big[u; u_0\big] \Big| \hspace{2pt}\mathrm{d} x &\hspace{2pt}\leq \hspace{3pt}\int_{\mathbb{R}^n}\left| \int_0^\epsilon \int_{\mathbb{R}^n} u(z, s) \hspace{1pt} \Gamma_n(x - z, t - s) \hspace{1pt} \nabla E_n * u \hspace{1pt}\bigg|_{ (z ,s)} \cdot \dfrac{z - x}{ t - s }  \hspace{2pt}\mathrm{d} z \hspace{1pt} \mathrm{d} s \right| \hspace{2pt}\mathrm{d} x \notag\\[2mm]
&\hspace{2pt}+ \hspace{3pt}\liminf_{L \rightarrow \infty}\int_{\mathbb{R}^n}\left| \int_0^\epsilon \int_{\mathbb{R}^n} u_L(z, s) \hspace{1pt} \Gamma_n(x - z, t - s) \hspace{1pt} \nabla E_n * u_L \hspace{1pt}\bigg|_{(z ,s)} \cdot \dfrac{z - x}{ t - s }  \hspace{2pt}\mathrm{d} z \hspace{1pt} \mathrm{d} s \right| \hspace{2pt}\mathrm{d} x. 
\end{align*}In light of (\ref{estimate_grad_v}) and (\ref{id_2_18}), we can apply Fubini's theorem  to the right-hand side of the above estimate  and obtain \begin{align*}\int_{\mathbb{R}^n}\Big| \hspace{1pt}\mathrm{Err}_0\big[u; u_0\big]\hspace{1pt} \Big| \hspace{2pt}\mathrm{d} x &\hspace{2pt}\lesssim \hspace{3pt} M^{1 + \frac{1}{n}} \left(K'\right)^{1 - \frac{1}{n}} \int_0^\epsilon s^{- 1 + \frac{1}{n}}\left(t - s \right)^{- \frac{1}{2}} \hspace{2pt}\mathrm{d} s\\[2mm]
&\hspace{2pt}\lesssim\hspace{2pt}M^{1 + \frac{1}{n}}  \left(K'\right)^{1 - \frac{1}{n}}  t^{- \frac{1}{2}} \epsilon^{\frac{1}{n}} \longrightarrow 0 \qquad \text{as $\epsilon \rightarrow 0$.}
\end{align*}Hence, $\mathrm{Err}_0\big[u; u_0\big] \equiv 0$ on $\mathbb{R}_+^{n+1}$. Then $u$ is a global mild solution to either (\ref{eqn_0_0_1}) or (\ref{eqn_0_0_2}) with initial data $u_0$ at $t = 0$. The proof is completed since $u$ satisfies the same estimate as (\ref{id_2_18}) for the approximating solutions $u_L$.
\end{proof}

In the end, we complete this section by showing the proof of (\ref{id_5_11}). \begin{lem}\label{monomonicity of phi function} Suppose that $n = 2$ and $\widetilde{u}$ is the function given in the proof of Proposition \ref{prop_3}. Then (\ref{id_5_11}) holds for the density function $\Phi_{z_1}\left(\cdot\right)$ defined in (\ref{def_phi_z_1}).
\end{lem}
\begin{proof}[\bf Proof] By the definition of $\Phi_{z_1}$,
\begin{align*}
\Phi_{z_1}(\rho) = \dfrac{1}{4\pi} \int_{\mathbb{R}^2} \widetilde{u} \left( y , s_1 - \rho^2 \right) e^{-\frac{|y-y_0|^2}{4\rho^2}} \mathrm{d}y = \dfrac{\rho^2}{4\pi} \int_{\mathbb{R}^2} \widetilde{u} \left( y_0 + \rho z , s_1 - \rho^2 \right) e^{-\frac{|z|^2}{4}} \mathrm{d}z.
\end{align*}
Using the equation of $\widetilde{u}$, i.\hspace{0.5pt}e.
\begin{align*}
\partial_s \widetilde{u} = \Delta_y \widetilde{u} - \nabla_y \cdot \left( \widetilde{u} \hspace{2pt}\widetilde{V} \right), \hspace{15pt}\text{where $\widetilde{V}\left(y, s\right) := \nabla  E_2 * \widetilde{u} \hspace{2pt} \bigg|_{(y ,s)}$,}
\end{align*}
we have
\begin{align*}
\dfrac{\mathrm{d}}{\mathrm{d}\rho} \Phi_{z_1}(\rho) &= \dfrac{1}{4\pi} \int_{\mathbb{R}^2} \left( -2\rho \hspace{1pt} \partial_s \widetilde{u} + \dfrac{|y-y_0|^2}{2\rho^3} \widetilde{u} \right) e^{-\frac{|y-y_0|^2}{4\rho^2}} \mathrm{d}y \\[1mm]
&= \dfrac{1}{4\pi} \int_{\mathbb{R}^2} \left( -2\rho \hspace{1pt} \Delta_y \widetilde{u} + 2 \rho \hspace{2pt} \nabla_y \cdot \left( \widetilde{u} \hspace{2pt}\widetilde{V} \right)  + \dfrac{|y-y_0|^2}{2\rho^3} \widetilde{u} \right) e^{-\frac{|y-y_0|^2}{4\rho^2}} \mathrm{d}y.
\end{align*}
Here $\widetilde{u}$ and $\widetilde{V}$ are evaluated at $s_1 - \rho^2$. Through integration by parts, it turns out
\begin{align*}
\int_{\mathbb{R}^2} \Delta_y \widetilde{u} \hspace{2pt}e^{-\frac{|y-y_0|^2}{4\rho^2}} \mathrm{d}y = \int_{\mathbb{R}^2} \widetilde{u} \hspace{2pt} \left( \dfrac{|y-y_0|^2}{4\rho^4} - \dfrac{1}{\rho^2} \right) e^{-\frac{|y-y_0|^2}{4\rho^2}} \mathrm{d}y 
\end{align*}
and
\begin{align*}
\int_{\mathbb{R}^2} \nabla_y \cdot \big( \widetilde{u} \hspace{2pt} \widetilde{V} \big)\hspace{2pt}e^{-\frac{|y-y_0|^2}{4\rho^2}} \mathrm{d}y = \int_{\mathbb{R}^2} \widetilde{u} \hspace{2pt}  \widetilde{V}  \cdot \dfrac{y-y_0}{2\rho^2} \hspace{1.5pt}e^{-\frac{|y-y_0|^2}{4\rho^2}} \hspace{2pt} \mathrm{d}y.
\end{align*}
Combining the last three identities induces
\begin{align}\label{id_6_01}
\dfrac{\mathrm{d}}{\mathrm{d}\rho} \Phi_{z_1}(\rho) &= \dfrac{1}{4\pi} \int_{\mathbb{R}^2} \left( \dfrac{2}{\rho} + \dfrac{y-y_0}{\rho} \cdot  \widetilde{V}  \right) \widetilde{u}  \hspace{2pt} e^{-\frac{|y-y_0|^2}{4\rho^2}} \hspace{2pt} \mathrm{d}y \notag\\[1mm]
&= \dfrac{2}{\rho}\hspace{1pt} \Phi_{z_1}(\rho) + \dfrac{1}{4\pi\rho} \int_{\mathbb{R}^2} \Big(\left( y-y_0 \right) \cdot  \widetilde{V}\Big) \hspace{2pt}\widetilde{u} \hspace{2pt} e^{-\frac{|y-y_0|^2}{4\rho^2}} \hspace{2pt} \mathrm{d}y \notag\\[1mm]
&= \dfrac{2}{\rho} \hspace{1pt}\Phi_{z_1}(\rho) + \dfrac{\rho^2}{4\pi} \int_{\mathbb{R}^2} \widetilde{u} \left( y_0 + \rho z , s_1 - \rho^2 \right) z \cdot  \widetilde{V} \hspace{1.5pt} \bigg|_{y = y_0 + \rho z} e^{-\frac{|z|^2}{4}} \mathrm{d}z.
\end{align}
The last integral in  (\ref{id_6_01}) can be further computed as follows:
\begin{align}\label{id_6_02}
& \dfrac{\rho^2}{4\pi} \int_{\mathbb{R}^2} \widetilde{u} \left( y_0 + \rho z , s_1 - \rho^2 \right) z \cdot  \widetilde{V} \hspace{1.5pt}\bigg|_{ y =  y_0 + \rho z} e^{-\frac{|z|^2}{4}} \mathrm{d}z 
\\[1mm]
&\hspace{25pt} = \hspace{2pt} -\dfrac{\rho^2}{8\pi^2} \int_{\mathbb{R}^2} \int_{\mathbb{R}^2}  z \cdot \dfrac{y_0 + \rho z - y'}{\big|\hspace{1pt}y_0 + \rho z - y' \hspace{1pt}|^2} \hspace{2pt}\widetilde{u} \left( y_0 + \rho z , s_1 - \rho^2 \right) \widetilde{u} \left( y' , s_1 - \rho^2 \right) e^{-\frac{|z|^2}{4}} \hspace{2pt}\mathrm{d} y' \hspace{1.5pt}\mathrm{d}z  \notag\\[1mm]
&\hspace{25pt}= \hspace{2pt} -\dfrac{\rho^3}{8\pi^2} \int_{\mathbb{R}^2} \int_{\mathbb{R}^2} z \cdot \dfrac{z - z'}{|z - z'|^2} \hspace{1.5pt}  \widetilde{u} \left( y_0 + \rho z , s_1 - \rho^2 \right) \widetilde{u} \left( y_0 + \rho z' , s_1 - \rho^2 \right) e^{-\frac{|z|^2}{4}} \hspace{2pt}\mathrm{d}z' \hspace{2pt}\mathrm{d}z \notag\\[1mm]
&\hspace{25pt}=\hspace{2pt} -\dfrac{\rho^3}{16\pi^2} \int_{\mathbb{R}^2} \int_{\mathbb{R}^2} \widetilde{u} \left( y_0 + \rho z , s_1 - \rho^2 \right) \widetilde{u} \left( y_0 + \rho z' , s_1 - \rho^2 \right) \left[ e^{-\frac{|z|^2}{4}} z - e^{-\frac{|z'|^2}{4}} z' \right] \cdot \dfrac{z - z'}{|z - z'|^2} \hspace{2pt} \mathrm{d}z' \hspace{2pt} \mathrm{d}z. \notag
\end{align}
Notice that for any $z, z' \in \mathbb{R}^2$,
\begin{align*}
\left[ e^{-\frac{|z|^2}{4}} z - e^{-\frac{|z'|^2}{4}} z' \right] \cdot \dfrac{z - z'}{|z - z'|^2} = \dfrac{1}{2} \left[ e^{-\frac{|z|^2}{4}} + e^{-\frac{|z'|^2}{4}} \right] + \dfrac{1}{2} \left[ e^{-\frac{|z|^2}{4}} - e^{-\frac{|z'|^2}{4}} \right] \dfrac{|z|^2 - |z'|^2}{|z - z'|^2}.
\end{align*}
Thus (\ref{id_6_02}) implies
\begin{align*}
& \dfrac{\rho^2}{4\pi} \int_{\mathbb{R}^2} \widetilde{u} \left( y_0 + \rho z , s_1 - \rho^2 \right) z \cdot  \widetilde{V} \hspace{1.5pt}\bigg|_{y =  y_0 + \rho z} e^{-\frac{|z|^2}{4}} \hspace{2pt} \mathrm{d}z \notag\\[1mm]
&\hspace{25pt}\geq\hspace{2pt} -\dfrac{\rho^3}{32\pi^2} \int_{\mathbb{R}^2} \int_{\mathbb{R}^2} \widetilde{u} \left( y_0 + \rho z , s_1 - \rho^2 \right) \widetilde{u} \left( y_0 + \rho z' , s_1 - \rho^2 \right) \left[ e^{-\frac{|z|^2}{4}} + e^{-\frac{|z'|^2}{4}} \right] \hspace{2pt} \mathrm{d}z' \hspace{2pt}\mathrm{d}z \hspace{2pt}=\hspace{2pt} -\frac{M}{4\pi\rho} \Phi_{z_1}(\rho).
\end{align*}
Putting the last inequality into (\ref{id_6_01}), we conclude (\ref{id_5_11}).
\end{proof}\vspace{1pc}


\section{\large Optimal time decay estimate}\label{sec_3}\vspace{0.5pc}
In this section, we focus on the higher dimensional case with $n \geq 3$. The main result is to improve the time decay rate obtained in Proposition \ref{prop_2} from $t^{-1}$ to $t^{- \frac{n}{2}}$. Our proof is a bootstrap argument which relies on the following lemma.

\begin{lem}\label{lem_2_2}
Let $n \geq 3$ and $u$ be a global mild solution to (\ref{eqn_0_0_1}) with initial data $u_0$ and the total mass $M$. If it satisfies
\begin{align}\label{id_lem_2_2_1}
\big\| \hspace{1pt}u\big(\cdot, t\big) \hspace{1pt}\big\|_\infty \leq   c_0 \hspace{1pt}t^{-1}  \hspace{15pt}\mbox{on }\hspace{4pt} 0 < t \leq 1 \hspace{20pt}\text{and}\hspace{20pt}\big\| \hspace{1pt}u\big(\cdot, t\big) \hspace{1pt}\big\|_\infty \leq  c_0 \hspace{1pt}t^{-\alpha} \hspace{15pt} \mbox{on} \hspace{6pt}  t > 1,
\end{align}
for some constants $c_0 > 0$ and $\alpha \in \left[ \hspace{0.5pt} 1 \hspace{0.5pt} , \frac{n}{2} \hspace{0.5pt}\right)$, then we have
\begin{align}\label{id_lem_2_2_0}
\big\| \hspace{1pt}u\big(\cdot, t\big) \hspace{1pt}\big\|_\infty \leq c_0^* \hspace{1pt} t^{-\beta}, \hspace{15pt}\mbox{for all } t > 0.
\end{align}
Here $\beta = \min \left\{ \dfrac{n}{2} , \left( 2 - \dfrac{1}{n} \right) \alpha - \dfrac{1}{2} \right\}$.  $c_0^*$ is a positive constant depending only on $n$, $c_0$ and $M$.
\end{lem}

\begin{proof}[\bf Proof]
Since $\beta > \alpha$, we only need to show (\ref{id_lem_2_2_0}) for $t > 2$. Notice that
\begin{align}\label{id_3_04_0}
u(x,t) &\hspace{2pt} = \hspace{2pt}  \int_{\mathbb{R}^n} u_0 (z) \hspace{2pt}\Gamma_n \left(x - z, t \right) \mathrm{d}z \notag \\[2mm]
&\hspace{2pt} - \hspace{2pt} \frac{1}{2}\int_0^t \int_{\mathbb{R}^n} u(z,s) \hspace{2pt} \Gamma_n\left(x - z, t - s \right) \hspace{1pt}\nabla E_n * u \hspace{1pt}\bigg|_{ (z,s)} \cdot \dfrac{z-x}{t-s} \hspace{2pt} \mathrm{d}z \hspace{1pt} \mathrm{d}s. 
\end{align}Here $\Gamma_n$ is the heat kernel on $\mathbb{R}^n$. The first integral on the right-hand side of (\ref{id_3_04_0}) can be estimated by
\begin{align}\label{id_3_04}
\left| \hspace{2pt} \int_{\mathbb{R}^n} u_0 (z) \hspace{1pt}\Gamma_n\left(x - z, t \right) \mathrm{d}z \hspace{2pt}\right| \hspace{2pt}\leq\hspace{2pt} \frac{1}{(4\pi t)^{\frac{n}{2}}} \int_{\mathbb{R}^n} u_0  \hspace{2pt}\leq\hspace{2pt} M\hspace{1pt}t^{- \frac{n}{2}}.
\end{align}
As for the last integral in (\ref{id_3_04_0}), we divide it into three parts. By using (\ref{estimate_grad_v}) and (\ref{id_lem_2_2_1}),
\begin{align}\label{id_3_05_1}
&\ \left| \int_{t/2}^t \int_{\mathbb{R}^n} u(z,s) \hspace{2pt}\Gamma_n\left(x - z, t - s \right) \nabla E_n * u \hspace{1pt}\bigg|_{ (z,s)} \cdot \dfrac{z-x}{t-s} \hspace{2pt} \mathrm{d}z \hspace{1pt}\mathrm{d}s \right| \notag\\[1mm]
&\hspace{35pt}\lesssim_{\hspace{0.5pt}M,\hspace{0.5pt}c_0} \hspace{2pt}  \int_{t/2}^t \frac{\mathrm{d}s}{s^{\alpha\left(2- \frac{1}{n}\right)} (t-s)^{\frac{1}{2}}}\int_{\mathbb{R}^n} \Gamma_n\left(x - z, t - s \right) \dfrac{|z-x|}{(t-s)^{\frac{1}{2}}} \mathrm{d}z \notag\\[2mm]
&\hspace{35pt}\lesssim_{\hspace{0.5pt}M,\hspace{0.5pt}c_0}\hspace{2pt} t\mystrut^{\hspace{0.5pt}\frac{1}{2} - \alpha \left(2-\frac{1}{n}\right) }.
\end{align}
Still using (\ref{estimate_grad_v}) and (\ref{id_lem_2_2_1}), we have
\begin{align}\label{id_3_06}
&\ \left| \int_1^{t/2} \int_{\mathbb{R}^n} u(z,s) \hspace{2pt}\Gamma_n\left(x - z, t - s \right) \nabla E_n * u \hspace{1pt}\bigg|_{ (z,s)} \cdot \dfrac{z-x}{t-s} \hspace{2pt} \mathrm{d}z \hspace{2pt} \mathrm{d}s \right| \notag\\[1mm]
&\hspace{35pt}\lesssim_{\hspace{0.5pt}M,\hspace{0.5pt}c_0} \hspace{2pt}  \int_1^{t/2} \dfrac{\mathrm{d}s}{(t-s)^{\frac{n+1}{2}} s^{\alpha \left(1-\frac{1}{n}\right)}}  \notag\\[1mm]
&\hspace{35pt}\lesssim_{\hspace{0.5pt}M,\hspace{0.5pt}c_0} \hspace{2pt}  t^{- \frac{n+1}{2}}  \int_1^{t/2} \dfrac{\mathrm{d} s}{s^{\alpha \left(1-\frac{1}{n}\right)}}   \hspace{4pt}\lesssim_{\hspace{0.5pt}M,\hspace{0.5pt}c_0}\hspace{4pt}
\begin{cases}
 t^{ \frac{1-n}{2} - \alpha \left(1- \frac{1}{n}\right)} \hspace{15pt} & \mbox{if } \alpha < \dfrac{n}{n-1}; \\[3mm]
t^{- \frac{n+1}{2}} \log \dfrac{t}{2} \hspace{15pt} & \mbox{if } \alpha = \dfrac{n}{n-1}; \\[3mm]
t^{- \frac{n+1}{2}} \hspace{15pt} & \mbox{if } \alpha > \dfrac{n}{n-1}.
\end{cases}
\end{align}
In the last, the integration on the time interval $[\hspace{0.5pt}0, 1]$ is estimated as follows:
\begin{align}\label{id_3_07}
&\ \left| \int_0^1 \int_{\mathbb{R}^n} u(z, s) \hspace{2pt}\Gamma_n\left(x - z, t - s \right) \nabla E_n * u \hspace{1pt}\bigg|_{ (z,s)} \cdot \dfrac{z-x}{t-s} \hspace{2pt} \mathrm{d}z \hspace{2pt} \mathrm{d}s \right| \notag\\[1mm]
&\hspace{50pt}\lesssim_{\hspace{0.5pt}M,\hspace{0.5pt}c_0} \hspace{2pt} \int_0^1 \dfrac{1}{(t-s)^{\frac{n+1}{2}} s^{1-\frac{1}{n}}} \hspace{2pt} \mathrm{d}s \notag\\[1mm]
&\hspace{50pt}\lesssim_{\hspace{0.5pt}M,\hspace{0.5pt}c_0} \hspace{4pt} t^{- \frac{n+1}{2}}.
\end{align}
Here we have also used (\ref{estimate_grad_v}) and (\ref{id_lem_2_2_1}). Therefore, we obtain (\ref{id_lem_2_2_0}) by (\ref{id_3_04_0})--(\ref{id_3_07}).
\end{proof}

As a direct consequence of Lemma \ref{lem_2_2}, we have

\begin{proof}[\bf Proof of $2 \Rightarrow 3$ in Theorem \ref{thm_1}]\
\\[2mm]
Assuming the $t^{-1}$-decay rate in Statement 2 of Theorem \ref{thm_1}, we can repeatedly apply Lemma \ref{lem_2_2} finitely many times by initially setting $\alpha = 1$. Statement 3 in Theorem \ref{thm_1} then follows.
\end{proof}\vspace{0.2pc}


\section{\large Long-time stability: higher dimensional case}\label{sec_4}\vspace{0.5pc}

We are concerned about in this section the long-time asymptotics of global mild solutions to (\ref{eqn_0_0_1}). Particularly, the $L^\infty$-norms of global mild solutions are assumed to satisfy the $t^{- \frac{n}{2}}$-temporal decay rate as $t$ goes to infinity. Throughout the section, the spatial dimension $n \geq 3$.

\subsection{\normalsize Proof of $3 \Rightarrow 4$ in Theorem \ref{thm_1}}\label{sec_4_1}\vspace{0.5pc}

Our main result in this section is \begin{prop}\label{asymptotics of u L1 n greater than 3}
Let $n \geq 3$ and suppose that $u$ is a global mild solution to (\ref{eqn_0_0_1}). Moreover, $u$ is assumed to have the finite total mass $M$ and satisfy \begin{align}\label{unif_bound_u_near t infity}
[\hspace{1pt} u \hspace{1pt}]_\infty := \sup_{t \hspace{0.5pt}\geq \hspace{0.5pt}0} \big( 1 + t \big)^{\frac{n}{2}}\big\| \hspace{1pt}u\big(\cdot, t\big) \hspace{1pt}\big\|_\infty < \infty.
\end{align}
Then we have \begin{align}\label{general p_L1_n=3}
\lim_{t \rightarrow \infty} t^{\frac{n}{2} \left( 1 - \frac{1}{p}\right)} \hspace{2pt}\big\|\hspace{1pt} u(\cdot, t) - M\hspace{1pt} \Gamma_t\left(\cdot\right) \big\|_p = 0, \hspace{15pt} \text{for any $1 \leq p \leq \infty$.}
\end{align}
Here $\Gamma_t$ is the fundamental solution of the heat equation $\partial_t w = \Delta w$ on $\mathbb{R}^n$.
\end{prop} 
\noindent The proof of $3 \Rightarrow 4$ in Theorem \ref{thm_1} follows easily from Proposition \ref{asymptotics of u L1 n greater than 3} above.\vspace{0.5pc}

In light of the change of variables (\ref{def_U}), the uniform bound in (\ref{unif_bound_u_near t infity}) induces 
\begin{align}\label{unif_bound_U}
\big\| \hspace{1pt}U (\cdot, \tau) \hspace{1pt}\big\|_{\infty} \leq [ \hspace{1pt}u \hspace{1pt}]_\infty < \infty, \qquad \text{for all $\tau \in \mathbb{R}$.}
\end{align}Moreover, the total mass of $U$ is conserved in the sense that
\begin{align}\label{mass_conservation}
\big\| U(\cdot,\tau) \big\|_{1} = \big\| u\left(\cdot,e^\tau\right) \big\|_{1} = M, \qquad \text{for all $\tau \in \mathbb{R}$.}
\end{align}To prove Proposition \ref{asymptotics of u L1 n greater than 3}, we need to study the strong $L^1$-compactness of the flow $\big\{U\left(\cdot, \tau\right)\big\}$ as $\tau \rightarrow \infty$.
\begin{lem}\label{lem_3_1}
Let $n \geq 3$ and suppose that $U$ is given in (\ref{def_U}) with $u \geq 0$ a classical global mild solution to (\ref{eqn_0_0_1}). If $U$ satisfies (\ref{unif_bound_U})--(\ref{mass_conservation}), then for any sequence $\big\{\tau_k\big\}$ which diverges to $\infty$ as $k \rightarrow \infty$, there is a subsequence, still denoted by $\big\{\tau_k\big\}$, such that  $ \big\{ U\left(\cdot, \tau_k\right)\big\}$ converges strongly in $L^1$ as $k \rightarrow \infty$.
\end{lem}

\begin{proof}[\bf Proof]The operator $L_n$ in (\ref{eqn_1}) generates the following semi-group: \begin{align}\label{semi group of U variable}
 S_n(\tau) \hspace{0.5pt}f :=  \big(4\pi \hspace{0.5pt} a(\tau)\big)^{ -\frac{n}{2}} \int_{\mathbb{R}^n} f(\eta)\hspace{1pt}e^{-\frac{\big|\hspace{0.5pt} \cdot - e^{- \frac{\tau}{2}} \hspace{0.5pt} \eta \hspace{0.5pt}\big|^2}{4a(\tau)} }  \mathrm{d}\eta, \qquad\text{where $a(\tau) := 1 - e^{-\tau}$.}
\end{align}
Using this linear semi-group, we decompose $U$ into
\begin{align}\label{id_7_01}
U = U_1 + U_2, \hspace{15pt}\text{where $U_1 := S_n(\tau) \hspace{1pt}U_0$.}
\end{align}
In (\ref{id_7_01}),  $U_0\left(\cdot\right) := U\left( \cdot, 0 \right)$. The function $U_2$ is given by
\begin{align}\label{def_U_2}
 U_2\left(\xi, \tau\right) := \dfrac{1}{2\left(4\pi \right)^{\frac{n}{2}}}  \int_0^{\tau} \dfrac{f_n\left(s\right) e^{- \frac{\tau - s}{2}}}{ a\left(\tau - s\right)^{1 + \frac{n}{2}}} \hspace{2pt} \mathrm{d} s\int_{\mathbb{R}^n}  U\left(\eta,s\right) \nabla_\eta V\left(\eta,s\right) \cdot \left(\xi -  e^{- \frac{\tau - s }{2}} \eta  \right) e^{ -\frac{\big| \xi - e^{-\frac{\tau - s}{2}} \eta \hspace{1pt} \big|^2}{4a\left(\tau-s\right)}} \mathrm{d} \eta.
\end{align}Here we use $V$ to denote the Newtonian potential $E_n * U$.
\vspace{0.3pc} 

To estimate $U_1$, we fix an arbitrary $R > 0$ and define
\begin{align}\label{def W1 W2}
W_1\left(\xi, \tau\right) := a\left(\tau\right)^{-\frac{n}{2}} \int_{B_{R / 2}} U_0\left(\eta\right)e^{-\frac{\big|\hspace{0.5pt} \xi - e^{- \frac{\tau}{2}} \hspace{0.5pt} \eta \hspace{0.5pt}\big|^2}{4\hspace{0.5pt}a(\tau)}} \mathrm{d}\eta,  \hspace{20pt}W_2 := \big(4 \pi\big)^{\frac{n}{2}}U_1 - W_1.
\end{align}
Since $2\left| \xi - e^{-\frac{\tau}{2}} \eta \right| \geq \big|\hspace{0.5pt}\xi\hspace{0.5pt}\big|$ for any $\xi \in B\mynewstrut^c_R$, $\eta \in B_{R/2}$ and $\tau \geq 0$, it turns out
\begin{align}\label{est of W1}
W_1\left(\xi, \tau\right) \leq M a\left(\tau\right)^{-\frac{n}{2}} e^{-\frac{|\xi|^2}{16\hspace{0.5pt}a(\tau)}}, \qquad \mbox{for all}\hspace{3pt} \tau > 0 \hspace{3pt}\text{ and }\hspace{3pt} \xi \in B\mynewstrut^c_R.
\end{align}
This estimate induces
\begin{align*}
\int_{B\mystrut^c_R} W_1\left(\xi, \tau\right) \mathrm{d}\xi \hspace{2pt} \leq \hspace{2pt} M \int_{B\mystrut^c_{R\hspace{0.5pt}a(\tau)^{- 1/2}}} e^{-\frac{|\xi|^2}{16}} \mathrm{d}\xi \hspace{2pt}\leq\hspace{2pt} M \int_{B\mystrut^c_R} e^{-\frac{|\xi|^2}{16}} \mathrm{d}\xi, \hspace{15pt}\text{for all $\tau > 0$.}
\end{align*}
On the other hand, by Fubini's theorem, it turns out
\begin{align*}
\int_{B\mystrut^c_R} W_2\left(\xi, \tau\right) \mathrm{d}\xi \hspace{2pt}\lesssim\hspace{2pt} \int_{B\mystrut^c_{R/2}} U_0, \hspace{15pt}\text{for all $\tau > 0$.}
\end{align*}
By the last two estimates and $U_1 = (4 \pi)^{-\frac{n}{2}}\left(W_1 + W_2\right)$, for any $\epsilon > 0$, there is a $R_\epsilon > 0$ such that   
\begin{align}\label{id_7_04_1}
\sup_{\tau \hspace{0.5pt}> \hspace{0.5pt} 0} \hspace{1.5pt}\int_{B\mystrut^c_{R_\epsilon}} U_1\left(\xi, \tau\right) \mathrm{d}\xi  \hspace{2pt}< \hspace{2pt} \dfrac{\epsilon}{2}.
\end{align}

In light of Lemma \ref{lem_2_1}  and (\ref{unif_bound_U})--(\ref{mass_conservation}), the $L^\infty$-norm of $\nabla_\xi V$ satisfies
\begin{align}\label{L inf est of nabla V}
\big\| \nabla_\xi V \big\|_\infty \hspace{2pt}\lesssim\hspace{2pt} M^{\frac{1}{n}} \big\| U  \big\|_\infty^{1 - \frac{1}{n}} \lesssim M^{\frac{1}{n}}  [ \hspace{1pt}u \hspace{1pt}]_\infty^{1 - \frac{1}{n}}, \hspace{15pt}\text{for all $\tau \geq 0$.}
\end{align}Utilizing the above estimate, (\ref{mass_conservation}) and  Fubini's theorem, we have
\begin{align}\label{conv of U2 in L1}
\int_{\mathbb{R}^n} \big|\hspace{1pt} U_2\left(\xi,\tau\right)\big|  \hspace{2pt} \mathrm{d}\xi \hspace{2pt}&\lesssim \hspace{2pt} \int_0^\tau \dfrac{f_n(s) e^{-\frac{\tau-s}{2}}}{ a\left(\tau - s\right)^{\frac{1}{2}}}  \hspace{2pt}\mathrm{d} s \int_{\mathbb{R}^n} U\left(\eta,s\right) \Big|\nabla_\eta V \left(\eta,s\right) \Big| \mathrm{d}\eta \notag\\[2mm]
&\lesssim_{\hspace{0.5pt}M, \hspace{1pt}[\hspace{0.5pt}u\hspace{0.5pt}]_{\infty}}  \int_0^\tau \dfrac{f_n(s) e^{-\frac{\tau-s}{2}}}{ a\left(\tau - s\right)^{\frac{1}{2}}} \mathrm{d}s \longrightarrow 0 \qquad \mbox{as } \tau \rightarrow \infty.
\end{align}  Here we have used the assumption $n \geq 3$ so that the convergence in (\ref{conv of U2 in L1}) holds. By (\ref{conv of U2 in L1}), for any $\epsilon > 0$, there is $\tau_\epsilon > 0$ such that
\begin{align}\label{id_7_06}
\sup_{\tau \hspace{0.5pt}\geq\hspace{0.5pt} \tau_\epsilon} \int_{\mathbb{R}^n} \big| \hspace{1pt} U_2\left(\xi, \tau\right) \big| \hspace{2pt} \mathrm{d}\xi \hspace{2pt} < \hspace{2pt} \dfrac{\epsilon}{2}.
\end{align}
Combining (\ref{id_7_01}), (\ref{id_7_04_1}) and (\ref{id_7_06}), for any $\epsilon > 0,$ we can find $R_\epsilon > 0$ and $\tau_\epsilon > 0$ so that \begin{align}\label{L1 estimate outside the ball}\sup_{\tau \hspace{0.5pt}\geq\hspace{0.5pt} \tau_\epsilon} \int_{B\mystrut^c_{R_\epsilon}} U\left(\xi, \tau\right)  \hspace{2pt} \mathrm{d}\xi \hspace{2pt} < \hspace{2pt} \epsilon.
\end{align}

Now we rewrite (\ref{eqn_1}) as follows: \begin{align*} - \partial_\tau U + \Delta_\xi U + \left( \dfrac{1}{2} \xi - f_n \nabla_\xi V\right) \cdot \nabla_\xi U + \left(\dfrac{n}{2} + f_n U \right) U = 0.
\end{align*}In light of (\ref{unif_bound_U}) and (\ref{L inf est of nabla V}), we can apply Theorem 7.22 in \cite{L05} to the above equation and obtain \begin{align*} \big\|\hspace{1pt} D_\xi^2 U \hspace{1pt}\big\|_{2n, P_{R}\left(0, T\right)} + \big\|\hspace{1pt} \partial_\tau U \hspace{1pt}\big\|_{2n, P_{R}\left(0, T\right)} \hspace{2pt}\leq\hspace{2pt}C, \hspace{15pt}\text{for all $T \geq 10 R^2$.}
\end{align*}Here $R > 0$ is arbitrarily given. The constant $C$ depends on $n$, $R$, $M$ and $[u]_\infty$. By a standard interpolation inequality, the last estimate yields \begin{align*} \big\|\hspace{1pt} D_\xi U \hspace{1pt}\big\|_{2n, P_{R}\left(0, T\right)} + \big\|\hspace{1pt} \partial_\tau U \hspace{1pt}\big\|_{2n, P_{R}\left(0, T\right)} \hspace{2pt}\leq\hspace{2pt}C, \hspace{15pt}\text{for all $T \geq 10 R^2$.}
\end{align*}Therefore, by Morrey's inequality, the last estimate induces \begin{align}\label{Holder norm est of U} \big\|\hspace{1pt} U \hspace{1pt}\big\|_{C^{0, \frac{n-1}{2n}}\left(P_{R}\left(0, T\right)\right)} \leq C, \hspace{15pt}\text{for all $T \geq 10 R^2$.}
\end{align}Given a sequence $\big\{\tau_k\big\}$ with $\tau_k \rightarrow \infty$ as $k \rightarrow \infty$, by (\ref{Holder norm est of U}), $U\left(\cdot, \tau_k\right)$ is uniformly bounded in $C^{0, \frac{n-1}{2n}}\left(B_{R}\right)$, provided that $k$ is large. Utilizing Arzel\` {a}-Ascoli theorem, we then can extract  a subsequence, still denoted by $\big\{\tau_k\big\}$, such that $U(\cdot, \tau_k ) \longrightarrow \overline{U}$ in $L^\infty\left( B_R \right)$, as $k \rightarrow \infty$. Since $R$ is arbitrary, the definition of $\overline{U}$ can be extended to $\mathbb{R}^n$ by a diagonal argument. Moreover, we can also extract a subsequence, still denoted by $\big\{ \tau_k \big\}$, so that \begin{align}\label{loc unif conv of U}
U\left(\cdot, \tau_k\right) \longrightarrow \overline{U} \hspace{15pt}\text{in $L_{\mathrm{loc}}^\infty\left(\mathbb{R}^n\right)$, as $k \rightarrow \infty$.}
\end{align}In light of (\ref{unif_bound_U})--(\ref{mass_conservation}) and Fatou's lemma, it holds $\overline{U} \in L_+^1 \cap L^\infty $. Then for any $\epsilon > 0$, we can choose $R_* > R_\epsilon$ sufficiently large such that
\begin{align}\label{id_7_20}
\int_{B\mystrut^c_{R_*}} \overline{U} (\xi) \hspace{2pt} \mathrm{d}\xi < \epsilon.
\end{align}
Here $R_\epsilon$ is given in (\ref{L1 estimate outside the ball}). On the other hand, by Lebesgue's dominated convergence theorem and (\ref{loc unif conv of U}), it follows
\begin{align*}
\big\| U(\cdot,\tau_k) - \overline{U} \big\|_{1, B_{R_*}} \longrightarrow 0 \qquad \mbox{as } k \rightarrow \infty.
\end{align*}The proof is then completed by the last convergence, (\ref{L1 estimate outside the ball}) and (\ref{id_7_20}).
\end{proof}
In the remainder of this section, we prove Proposition \ref{asymptotics of u L1 n greater than 3}.

\begin{proof}[\bf Proof of Proposition \ref{asymptotics of u L1 n greater than 3}] \vspace{0.2pc}
Recall the decomposition (\ref{id_7_01}). By Lebesgue's dominated convergence theorem, $U_1$ satisfies
\begin{align*}
\lim_{\tau\rightarrow\infty} U_1(\xi,\tau) = \lim_{\tau\rightarrow\infty}\left(4\pi\right)^{-\frac{n}{2}} a\left(\tau\right)^{-\frac{n}{2}} \int_{\mathbb{R}^n} U_0\left(\eta\right)e^{-\frac{\big|\hspace{0.5pt} \xi - e^{- \frac{\tau}{2}} \hspace{0.5pt} \eta \hspace{0.5pt}\big|^2}{4a(\tau)} } \mathrm{d}\eta = M \mathcal{G}_n(\xi).
\end{align*}Here  $\mathcal{G}_n$ is the $n$-dimensional Gaussian probability density given in (\ref{gaussian density def}). Combining this convergence with (\ref{conv of U2 in L1}), we have $U\left(\cdot, \tau\right) \longrightarrow M \mathcal{G}_n$ almost everywhere, as $\tau \rightarrow \infty$. The $L^1$-compactness result in Lemma \ref{lem_3_1} then induces 
\begin{align}\label{prop L1 case}
\lim_{\tau \rightarrow \infty} \big\| U(\cdot,\tau) - M \mathcal{G}_n \big\|_1 = 0.
\end{align}
Equivalently, (\ref{general p_L1_n=3}) holds for $p=1$. Here we have used the change of variables in (\ref{def_U}). \vspace{0.3pc}

Now we fix an arbitrary $R > 0$ and still divide $U_1$ into $U_1 = (4\pi)^{-\frac{n}{2}} \left( W_1 + W_2 \right)$. The functions  $W_1$ and $W_2$ have been defined in (\ref{def W1 W2}). For $W_2$, it can be bounded from above pointwisely as follows:
\begin{align*}
W_2\left(\xi, \tau\right) \hspace{2pt}\lesssim\hspace{2pt} \int_{B\mystrut^c_{R/2}} U_0  \hspace{15pt} \mbox{on } \mathbb{R}^n \times [1,\infty).
\end{align*}This estimate and (\ref{est of W1}) then infer 
\begin{align}\label{id_7_24}
\big\| U_1 \big\|_{\infty;\hspace{1pt}B\mystrut^c_R \times [\hspace{0.5pt}1,\infty)} \longrightarrow 0 \qquad \mbox{as } R \rightarrow \infty.
\end{align} 
In light of (\ref{unif_bound_U})--(\ref{mass_conservation}) and (\ref{L inf est of nabla V}), $U_2$ can also be estimated pointwisely by
\begin{align}\label{U2 uniform estimate}
\big|\hspace{1pt}U_2(\xi,\tau) \hspace{1pt}\big| \hspace{2pt}&\lesssim \hspace{2pt}\int_0^{\frac{3\tau}{4}}\dfrac{f_n\left(s\right) e^{- \frac{\tau - s}{2}}}{ a\left(\tau - s\right)^{1 + \frac{n}{2}}} \hspace{2pt} \mathrm{d} s \int_{\mathbb{R}^n}  U\left(\eta,s\right) \Big| \nabla_\eta V\left(\eta,s\right) \Big|  \left|\xi -  e^{- \frac{\tau - s }{2}} \eta  \right| e^{ -\frac{\big| \xi - e^{-\frac{\tau - s}{2}} \eta \hspace{1pt} \big|^2}{4a\left(\tau-s\right)}} \mathrm{d} \eta  \notag\\[2mm]
&+ \hspace{4pt}\int_{\frac{3\tau}{4}}^\tau \hspace{2pt}\dfrac{f_n\left(s\right) e^{- \frac{\tau - s}{2}}}{ a\left(\tau - s\right)^{1 + \frac{n}{2}}} \hspace{3pt}  \mathrm{d} s\int_{\mathbb{R}^n} U\left(\eta,s\right) \hspace{0.5pt} \Big| \nabla_\eta V\left(\eta,s\right) \Big|   \left|\xi -  e^{- \frac{\tau - s }{2}} \eta  \right| e^{ -\frac{\big| \xi - e^{-\frac{\tau - s}{2}} \eta \hspace{1pt} \big|^2}{4a\left(\tau-s\right)}} \mathrm{d} \eta \notag\\[2mm]
&\lesssim_{\hspace{1pt}M, \hspace{1pt}[\hspace{0.5pt}u\hspace{0.5pt}]_\infty}\hspace{2pt} \int_0^{\frac{3\tau}{4}} \dfrac{f_n\left(s\right) e^{- \frac{\tau - s}{2}}}{ a\left(\tau - s\right)^{\frac{n+1}{2}}} \hspace{2pt} \mathrm{d}s  + \int_{\frac{3\tau}{4}}^{\tau} \dfrac{f_n\left(s\right) e^{\frac{n-1}{2}(\tau - s)}}{a\left(\tau - s\right)^{\frac{1}{2}}} \hspace{2pt}\mathrm{d}s \notag\\[2mm]
&\lesssim_{\hspace{1pt}M, \hspace{1pt}[\hspace{0.5pt}u\hspace{0.5pt}]_\infty}\hspace{2pt}  e^{- \frac{\tau}{2}} \left(1-e^{- \frac{\tau}{4}}\right)^{- \frac{n+1}{2}}\int_0^{\frac{3\tau}{4}} e^{ \frac{3-n}{2} \hspace{0.5pt}s} \hspace{2pt}\mathrm{d}s  +  e^{ \frac{5 - 2n}{8} \hspace{0.5pt}\tau} \int_{\frac{3\tau}{4}}^\tau \dfrac{ \mathrm{d}s}{a\left(\tau - s\right)^{\frac{1}{2}}}.
\end{align}
Therefore, when $n \geq 3$, the above estimate induces
\begin{align}\label{id_7_26}
\big\| U_2(\cdot,\tau) \big\|_{\infty} \longrightarrow 0 \qquad \mbox{as } \tau \rightarrow \infty.
\end{align}Utilizing (\ref{id_7_01}), (\ref{loc unif conv of U}), (\ref{id_7_24}), (\ref{id_7_26}) and the fact that $\overline{U} = M \mathcal{G}_n$, we obtain
\begin{align}\label{prop infty case}
\lim_{\tau \rightarrow \infty} \big\| U(\cdot,\tau) - M \mathcal{G}_n \big\|_\infty = 0.
\end{align}
It is equivalent to (\ref{general p_L1_n=3}) with $p=\infty$. In light of (\ref{prop L1 case}) and (\ref{prop infty case}), the proof of (\ref{general p_L1_n=3}) for general $p \in (1,\infty)$ then follows by an interpolation argument.
\end{proof}

\subsection{\normalsize Proof of Theorem \ref{thm_2}}\label{sec_4_2}\vspace{0.5pc}

When the initial density has a finite second moment, many literatures have been devoted to studying the long-time asymptotics of global solutions to the classical Patlak-Keller-Segel equation (\ref{eqn_0_0_2}) on $\mathbb{R}^2$ (see e.g. \cite{BDP06, CD14, FM16}). Now, we focus on the higher dimensional case. Our main result in this section is
 \begin{prop}\label{weight asymptotics of u} Let $n \geq 3$ and suppose that $u$ is a classical global mild solution to (\ref{eqn_0_0_1}). If (\ref{unif_bound_u_near t infity}) holds for $u$ and meanwhile the non-negative initial data  $u_0(\cdot) = u(\cdot, 0)$ satisfies   \begin{align}\label{finite_mass_moment} \int_{\mathbb{R}^n} u_0(x) \hspace{1pt}\big(1 + | x |^2 \big) \hspace{1pt}\mathrm{d} x < \infty, \end{align} then  the followings hold: \begin{enumerate}
\item[$\mathrm{(1).}$] If $n \geq 5$, then \begin{eqnarray*}\label{p = 1} t^{\frac{n}{2}\left( 1 - \frac{1}{p}\right)}\big\|\hspace{1pt}u\left(\cdot, t\right) - M \Gamma_t + B_0 \cdot \nabla \Gamma_t\hspace{1pt}\big\|_p = \mathrm{O}\big(t^{-1}\big),\hspace{15pt}\text{for any $p \in [\hspace{0.5pt}1, \infty\hspace{0.5pt}]$ and $t \geq 1$.}
\end{eqnarray*}Here $B_0$ is the center of mass given in (\ref{center of mass}). $\Gamma_t$ is the fundamental solution of $\partial_t w = \Delta w$ on $\mathbb{R}^n$; 
\item[$\mathrm{(2).}$] If $n = 3$, then for any $p \in [\hspace{1pt}1, \infty\hspace{0.5pt}]$ and $t \geq 1$, it holds 
\begin{eqnarray*}t^{\frac{3}{2}\left(1 - \frac{1}{p}\right)} \hspace{2pt}\left\| u(\cdot, t) - M\hspace{1pt}\Gamma_t + B_0 \cdot \nabla \Gamma_t + M^2 \hspace{0.5pt}\mathcal{W}\left(\cdot, t\right)  + c_1 t^{-\frac{5}{2}}\log t \left( \frac{1}{2} - \frac{|\hspace{0.5pt}x\hspace{0.5pt}|^2}{12t} \right) e^{-\frac{|\hspace{0.5pt}x\hspace{0.5pt}|^2}{4t}}\right\|_{p} = \mathrm{O}\big(t^{-1}\big).\end{eqnarray*}The function $\mathcal{W}$ is defined in (\ref{representation of function cal W}). The constant $c_1$ is given in Theorem \ref{thm_2};
\item[$\mathrm{(3)}.$]  If $n = 4$, then for any $p \in [\hspace{1pt}1, \infty\hspace{0.5pt}]$ and $t \geq 1$, it holds 
\begin{eqnarray}t^{2\left(1 - \frac{1}{p}\right)} \hspace{2pt}\left\| u(\cdot, t) - M\hspace{1pt}\Gamma_t + B_0 \cdot \nabla \Gamma_t   - c_2 \hspace{0.5pt} t^{- 3}\log t \left( \frac{1}{2} - \frac{|\hspace{0.5pt}x\hspace{0.5pt}|^2}{16t} \right) e^{-\frac{|\hspace{0.5pt}x\hspace{0.5pt}|^2}{4t}}\right\|_{p} = \mathrm{O}\big(t^{-1}\big).\end{eqnarray}The constant $c_2$ is given in Theorem \ref{thm_2}. \end{enumerate}
\end{prop} 
\noindent Theorem \ref{thm_2} follows easily from Proposition \ref{weight asymptotics of u}. \vspace{0.3pc}

The remainder of this section is devoted to proving Proposition \ref{weight asymptotics of u}. Our approach is based on a bootstrap argument.  Recall the decomposition of $U$ in (\ref{id_7_01}). By the standard estimate of heat equations (see \cite{V17}\hspace{0.5pt}), the first component $U_1(\cdot, \tau) = S_n(\tau) \hspace{1pt}U_0$ satisfies \begin{prop}\label{estimate of U1 up to second order}Suppose that $U_0$ has a finite second moment. Then there is a  constant $C > 0$ so that \begin{align}\label{est of L1 and Linfty of heat equation} \left\| \hspace{1pt}S_n(\tau)\hspace{1pt}U_0  - M \mathcal{G}_n + e^{-\frac{\tau}{2}}  \hspace{1pt}B_0\cdot \nabla \mathcal{G}_n \hspace{1pt}\right\|_p \hspace{2pt}\leq\hspace{2pt} C \hspace{0.5pt}e^{- \tau}, \hspace{15pt}\text{for any $\tau \geq 1$ and $p \in [1, \infty]$.}
\end{align}Here the vector $B_0$ denotes the following center of mass of $U_0$:\begin{align*} B_0 := \int_{\mathbb{R}^n} z \hspace{1pt}U_0(z)\hspace{2pt}\mathrm{d} z.
\end{align*}The constant $C$ depends on $n$, $M$ and the second moment of $U_0$. 
\end{prop}\noindent 
Notice that a solution $u$ to (\ref{eqn_0_0_1}) with a finite second moment preserves the center of mass. In light of (\ref{def_U}), we have \begin{align*} B_0 =  \int_{\mathbb{R}^n}z \hspace{1pt}U_0\left(z\right)\mathrm{d} z = \int_{\mathbb{R}^n}x \hspace{1pt}u_0\left(x\right)\mathrm{d} x.
\end{align*} The  estimate of $U_1$ in Proposition \ref{estimate of U1 up to second order}  is optimal if $U_0$ is only assumed to have a finite second moment. However, as for the supremum norm estimate of $U_2$, the temporal decay rate  in (\ref{U2 uniform estimate}) is still not satisfactory. 

\subsubsection{\normalsize Approximation of $U_2$ up to order $e^{- \frac{\tau}{2}}$} \label{sec_4_2_1}\vspace{0.5pc}
 
 The main result in this section is\begin{lem}\label{temporal optima of U2}Suppose that $u_0$ has a finite second moment. Then there exist a positive constant $C$ depending on $M$, $[u]_\infty$ and the second moment of $u_0$  and a positive constant $\beta_n$ depending on $n$ only so that   \begin{align*}  \left\| \hspace{1pt}U_2 + M^2  e^{- \frac{\tau}{2}}\hspace{1pt}\delta_{n, 3}\int_0^\infty e^{\frac{s}{2}} S_n(s) \hspace{1pt}\Big[ \mathrm{div}\Big(\hspace{1pt} \mathcal{G}_n \nabla \mathcal{V}_n\Big) \hspace{1pt}\Big] \hspace{1pt}\mathrm{d} s \hspace{1pt}\right\|_p \hspace{2pt}\leq\hspace{2pt} C e^{ - \left(\frac{1}{2} + \beta_n\right) \tau}, \hspace{10pt}\text{for any $p \in [1, \infty]$ and $\tau \geq 0$.}
\end{align*}Here $\mathcal{V}_n := E_n * \mathcal{G}_n$. $u$ is the global mild solution to (\ref{eqn_0_0_1}) with the initial data $u_0$ at $t = 0$. $U$ is the representation of $u$ under the similarity variables. The constant $\delta_{n, 3} = 1$ if $n = 3$. It equals to $0$ if $n \neq 3$.
\end{lem}
\begin{proof}[\bf Proof] We use $U^{(0)}$ and $V^{(0)}$ to denote $M \mathcal{G}_n$ and $M\mathcal{V}_n$, respectively. By the representation of $U_2$ in (\ref{def_U_2}),  \begin{align}\label{decomp of e to tau over 2 U2}e^{\frac{\tau}{2}} U_2  =\left(4 \pi\right)^{-\frac{n}{2}} \hspace{1pt}\big[\hspace{1pt}\mathrm{I} + \mathrm{II} + \mathrm{III}\hspace{1pt}\big],
\end{align}where the terms $\mathrm{I}$, $\mathrm{II}$ and $\mathrm{III}$ are given as follows: \begin{align*} &\mathrm{I} :=  \int_{\frac{\tau}{2}}^{\tau} \dfrac{f_n\left(s \right)e^{\frac{\tau}{2}}}{ a\left(\tau - s\right)^{ \frac{n}{2}}} \hspace{2pt} \mathrm{d} s\int_{\mathbb{R}^n} U^{(0)}(z) \nabla V^{(0)}(z) \cdot \nabla_z e^{ -\frac{\big| \xi - e^{-\frac{\tau - s}{2}} z \hspace{1pt} \big|^2}{4a\left(\tau-s\right)}}\mathrm{d} z;\\[2mm]
&\mathrm{II} := \int_{\frac{\tau}{2}}^{\tau} \dfrac{f_n\left(s\right) e^{\frac{\tau}{2}}}{ a\left(\tau - s\right)^{\frac{n}{2}}} \hspace{2pt} \mathrm{d} s\int_{\mathbb{R}^n} \Big[ U\left(z,s\right) \nabla V\left(z,s\right) - U^{(0)}(z) \nabla V^{(0)}(z)\Big] \cdot \nabla_ze^{ -\frac{\big| \xi - e^{-\frac{\tau - s}{2}} z \hspace{1pt} \big|^2}{4a\left(\tau-s\right)}} \mathrm{d} z;\\[2mm]
&\mathrm{III} := \int_0^{\frac{\tau}{2}} \dfrac{f_n\left(s\right) e^{\frac{\tau}{2}}}{ a\left(\tau - s\right)^{\frac{n}{2}}} \hspace{2pt} \mathrm{d} s\int_{\mathbb{R}^n}  U\left(z,s\right) \nabla V\left(z,s\right) \cdot \nabla_ze^{ -\frac{\big| \xi - e^{-\frac{\tau - s}{2}} z \hspace{1pt} \big|^2}{4a\left(\tau-s\right)}} \mathrm{d} z.
\end{align*}
In the next, we estimate the terms $\mathrm{I}$, $\mathrm{II}$ and $\mathrm{III}$. \vspace{0.5pc}\\
\textbf{Estimate of I.}   Replacing $U_0$ in (\ref{est of L1 and Linfty of heat equation}) by $\mathrm{div}\big(\mathcal{G}_n \nabla \mathcal{V}_n\big)$ and using the following null conditions: \begin{align*} \int_{\mathbb{R}^n} \mathrm{div}\Big(\hspace{1pt} \mathcal{G}_n \nabla \mathcal{V}_n\Big) = \int_{\mathbb{R}^n} \xi_j \hspace{1pt} \mathrm{div}\Big(\hspace{1pt} \mathcal{G}_n \nabla \mathcal{V}_n\Big) \hspace{2pt}\mathrm{d} \xi = 0, \hspace{15pt}\text{ $j = 1, ..., n$,}
\end{align*} we have \begin{align}\label{estimate of div of gn nabla vn pre}\left\| \hspace{1pt}S_n(s) \hspace{1pt}\Big[ \mathrm{div}\Big(\hspace{1pt} \mathcal{G}_n \nabla \mathcal{V}_n\Big) \hspace{1pt}\Big] \hspace{1pt}\right\|_p \hspace{2pt}\lesssim\hspace{2pt} e^{- s}, \hspace{15pt}\text{for any $s \geq 1$ and $p \in [1, \infty]$.}
\end{align}For $s \in [0, 1]$, the $L^p$ norm of $S_n(s)\hspace{0.5pt}\left[ \mathrm{div}\big( \mathcal{G}_n \nabla \mathcal{V}_n\big) \right]$ is uniformly bounded from above with the upper bound independent of $p$. This uniform upper bound, in conjunction with (\ref{estimate of div of gn nabla vn pre}), yields \begin{align}\label{estimate of div of gn nabla vn}\left\| \hspace{1pt}S_n(s) \hspace{1pt}\Big[ \mathrm{div}\Big(\hspace{1pt} \mathcal{G}_n \nabla \mathcal{V}_n\Big) \hspace{1pt}\Big] \hspace{1pt}\right\|_p \hspace{2pt}\lesssim\hspace{2pt} e^{- s}, \hspace{15pt}\text{for any $s \geq 0$ and $p \in [1, \infty]$.}
\end{align}

Change the variable  $s$ in term $\mathrm{I}$ to $\tau - s$ and integrate by parts with respect to the variable $z$. It holds \begin{align*} \mathrm{I} = - \left(4 \pi\right)^{\frac{n}{2}}  \int_0^{\frac{\tau}{2}} e^{\frac{s}{2}} \hspace{1pt}e^{\frac{3 - n}{2}(\tau - s)} S_n\left(s\right) \Big[ \mathrm{div} \left( U^{(0)} \nabla V^{(0)}\right)\Big]\hspace{1pt}  \mathrm{d} s.
\end{align*} By (\ref{estimate of div of gn nabla vn}), it follows \begin{align*} \big\| \hspace{1pt}\mathrm{I}\hspace{1pt} \big\|_p &\hspace{2pt}\lesssim\hspace{2pt} \int_0^{\frac{\tau}{2}} e^{\frac{s}{2}} \hspace{1pt}e^{\frac{3 - n}{2}(\tau - s)} \hspace{1pt}\left\| S_n\left(s\right) \Big[ \mathrm{div} \left( U^{(0)} \nabla V^{(0)} \right)\Big] \right\|_p \mathrm{d} s  \hspace{2pt}\lesssim\hspace{2pt}  M^2\int_0^{\frac{\tau}{2}} e^{- \frac{s}{2}}\hspace{1pt}e^{\frac{3 - n}{2}(\tau - s)}  \hspace{2pt} \mathrm{d} s.
\end{align*}Therefore, \begin{align*}\big\| \hspace{1pt}\mathrm{I}\hspace{1pt} \big\|_p &\hspace{2pt}\lesssim\hspace{2pt}M^2 \hspace{1pt}e^{ - \frac{\tau}{4}},  \hspace{10pt}\text{for all $p \in [1, \infty]$ and  $n \geq 4$.}
\end{align*}If $n = 3$, then similar derivations for the above estimate yield \begin{align*} \left\|\hspace{1pt} \hspace{1pt}\mathrm{I} + \left(4\pi\right)^{\frac{3}{2}}  \int_0^{\infty} e^{\frac{s}{2}} \hspace{1pt} S_3\left(s\right) \Big[ \mathrm{div}\left(U^{(0)} \nabla V^{(0)}\right)\Big] \mathrm{d} s  \hspace{1pt}\right\|_p   \hspace{2pt}\lesssim \hspace{2pt} M^2 e^{ - \frac{\tau}{4}}, \hspace{15pt}\text{for all $p \in [1, \infty]$.}
\end{align*} In light of the last two estimates, it follows \begin{align}\label{estimate of term I}\left\|\hspace{1pt} \hspace{1pt}\mathrm{I} + \left(4\pi\right)^{\frac{n}{2}} M^2 \delta_{n, 3} \int_0^{\infty} e^{\frac{s}{2}} \hspace{1pt} S_n\left(s\right) \Big[ \mathrm{div}\big(\mathcal{G}_n \nabla \mathcal{V}_n\big)\Big] \mathrm{d} s  \hspace{1pt}\right\|_p   \hspace{2pt}\lesssim\hspace{2pt} M^2 e^{-\frac{\tau}{4}}, \hspace{10pt}\text{for all $p \in [1, \infty]$ and $n \geq 3$.}\end{align}
 \vspace{0.3pc}\\
\noindent\textbf{Estimate of II.} Due to (\ref{conv of U2 in L1}), (\ref{U2 uniform estimate}) and (\ref{est of L1 and Linfty of heat equation}), there is a positive constant $C$  so that \begin{align}\label{estimate of U1 - MGn with exponential bound} \big\|\hspace{1pt}U - M \mathcal{G}_n\hspace{1pt}\big\|_p \hspace{2pt}\leq\hspace{2pt}C e^{- \frac{\tau}{16}}, \hspace{15pt}\text{for all $\tau \geq 0$ and $p \in [1, \infty]$.}
\end{align}Now we fix an $\alpha \in (1/2, 1)$. By (\ref{estimate of U1 - MGn with exponential bound}) and (\ref{estimate_grad_v}), it turns out  \begin{align}\label{estimate 000000}& \left| \int_{\frac{\tau}{2}}^{\alpha\tau} \dfrac{e^{\frac{3 - n}{2}s}}{ a\left(\tau - s\right)^{1 + \frac{n}{2}}} \hspace{2pt} \mathrm{d} s\int_{\mathbb{R}^n} \Big[ U\left(z,s\right) \nabla V\left(z,s\right) - U^{(0)}(z) \nabla V^{(0)}(z)\Big] \cdot \left(\xi -  e^{- \frac{\tau - s }{2}} z  \right) e^{ -\frac{\big| \xi - e^{-\frac{\tau - s}{2}} z \hspace{1pt} \big|^2}{4a\left(\tau-s\right)}} \mathrm{d} z \right| \notag\\[2mm]
&\hspace{60pt}\lesssim\hspace{2pt} \int_{\frac{\tau}{2}}^{\alpha\tau} \dfrac{e^{\frac{3 - n}{2}s}}{ a\big( \big(1 - \alpha \big) \tau\big)^{\frac{n+1}{2}}} \hspace{2pt} \mathrm{d} s\int_{\mathbb{R}^n} \Big|\hspace{1pt} U\left(z,s\right) \nabla V\left(z,s\right) - U^{(0)}(z) \nabla V^{(0)}(z) \hspace{1pt}\Big| \hspace{2pt}\mathrm{d} z   \notag\\[2mm]
&\hspace{60pt}\leq \hspace{2pt}C\hspace{1pt}\left( 1 - e^{- (1 - \alpha)\tau}\right)^{- \frac{n+1}{2}} \int_{\frac{\tau}{2}}^{\infty} e^{\frac{23-8n}{16}s} \hspace{2pt}\mathrm{d} s.\end{align}Still by (\ref{estimate of U1 - MGn with exponential bound}) and (\ref{estimate_grad_v}), it satisfies \begin{align*}& \left| \int_{\alpha \tau}^{\tau} \dfrac{e^{\frac{3 - n}{2}s}}{ a\left(\tau - s\right)^{1 + \frac{n}{2}}} \hspace{2pt} \mathrm{d} s\int_{\mathbb{R}^n} \Big[ U\left(z,s\right) \nabla V\left(z,s\right) - U^{(0)}(z) \nabla V^{(0)}(z)\Big] \cdot \left(\xi -  e^{- \frac{\tau - s }{2}} z  \right) e^{ -\frac{\big| \xi - e^{-\frac{\tau - s}{2}} z \hspace{1pt} \big|^2}{4a\left(\tau-s\right)}} \mathrm{d} z \right| \notag\\[2mm]
&\hspace{80pt}\lesssim\hspace{2pt}  \int_{\alpha \tau}^{\tau} \dfrac{e^{\frac{3 - n}{2}s}e^{\frac{n}{2}\left(\tau - s \right)}}{ a\left(\tau - s\right)^{\frac{1}{2}}} \hspace{2pt}\Big\| U\left(\cdot, s\right) \nabla V\left(\cdot, s\right) - U^{(0)} \nabla  V^{(0)}\Big\|_\infty \hspace{2pt} \mathrm{d} s  \notag\\[2mm]
&\hspace{80pt}\leq \hspace{2pt}C\hspace{1pt} \int_{\alpha \tau}^\tau \dfrac{e^{\frac{23 - 8n}{16}s} e^{\frac{n}{2}(\tau - s )} }{ a\left(\tau - s\right)^{\frac{1}{2}}} \hspace{2pt} \mathrm{d} s \hspace{2pt}\leq\hspace{2pt}C \hspace{1pt}e^{\frac{23 - 8n}{16} \alpha \tau} e^{\frac{n}{2}(1 - \alpha ) \tau} \int_0^\tau \dfrac{\mathrm{d} s}{\sqrt{a(s)}}. \end{align*}Therefore, we can fix a $\alpha$ close to $1$ so that  \begin{align*} & \left| \int_{\alpha \tau}^{\tau} \dfrac{e^{\frac{3 - n}{2}s}}{ a\left(\tau - s\right)^{1 + \frac{n}{2}}} \hspace{2pt} \mathrm{d} s\int_{\mathbb{R}^n} \Big[ U\left(z,s\right) \nabla V\left(z,s\right) - U^{(0)}(z) \nabla  V^{(0)}(z)\Big] \cdot \left(\xi -  e^{- \frac{\tau - s }{2}} z  \right) e^{ -\frac{\big| \xi - e^{-\frac{\tau - s}{2}} z \hspace{1pt} \big|^2}{4a\left(\tau-s\right)}} \mathrm{d} z \right| \notag\\[2mm]
&\hspace{80pt}\leq \hspace{2pt} C \hspace{1pt}e^{- \mu_n\tau} \int_0^\tau \dfrac{\mathrm{d} s}{\sqrt{a(s)}}, \hspace{15pt}\text{for some constant $\mu_n > 0$ depending on $n$.} \end{align*}Note that  $\alpha$ can also be chosen to depend only on $n$. With the fixed $\alpha$, the last estimate and (\ref{estimate 000000}) induce 	 \begin{align}\label{estimate of times e tau 2 of U2}& \big\|\hspace{1pt}\mathrm{II}\hspace{1pt} \big\|_\infty \hspace{2pt}\leq\hspace{2pt} C \hspace{1pt}e^{- \beta_n\tau}, \hspace{15pt}\text{for some constant $\beta_n > 0$ depending only on $n$.}\end{align}As for $L^1$-estimate of $\mathrm{II}$, by Fubini's theorem, it satisfies\begin{align}\label{est of quantity in times in L1}&\big\|\hspace{1pt}\mathrm{II}\hspace{1pt} \big\|_1 \hspace{2pt}\leq \hspace{2pt}C \int_{\frac{\tau}{2}}^\tau \dfrac{e^{\frac{23-8n}{16}s}}{a\left(\tau - s\right)^{\frac{1}{2}}} \hspace{2pt}\mathrm{d}s \hspace{2pt}\leq\hspace{2pt}C e^{- \beta_n \tau},\hspace{15pt}\text{for some constant $\beta_n > 0$ depending only on $n$.}
\end{align}\\
\textbf{Estimate of III.} Before we estimate the term $\mathrm{III}$, we need to consider the uniform boundedness of the second moment of $U$.  Let $\eta$ be a smooth function compactly supported on $B_2$. Moreover, $\eta \equiv 1$ on $B_1$ and satisfies $0 \leq \eta \leq 1$ on $\mathbb{R}^n$. Fixing an arbitrary $R > 0$, we define \begin{align}  \eta_R(x) :=   | x |^2 \hspace{1pt}\eta\left(\frac{x}{R}\right), \hspace{15pt}\text{ for any $x \in \mathbb{R}^n$.}\end{align}Then we multiply $\eta_R$ on both sides of (\ref{eqn_0_0_1}) and integrate over $\mathbb{R}^n$. It turns out \begin{align}\label{approx_2nd_moment}\dfrac{\mathrm{d}}{\mathrm{d} t} \int_{\mathbb{R}^n} u \hspace{1pt}\eta_R = &\hspace{2pt} \int_{\mathbb{R}^n} u \hspace{1pt}\Delta \eta_R   -  \dfrac{1}{2\hspace{0.5pt} n \hspace{0.5pt} \omega_n} \int_{\mathbb{R}^n}\int_{\mathbb{R}^n} \dfrac{\big( \nabla \eta_R(x) - \nabla \eta_R(y)\big) \cdot (x - y)}{| x - y |^n} u(x, t) \hspace{1pt}u(y, t) \hspace{2pt}\mathrm{d} x \hspace{1pt}\mathrm{d} y.
\end{align}Here $\omega_n$ denotes the volume of the unit ball in $\mathbb{R}^n$. Since for any $R > 0$, it satisfies  \begin{align*} \sum_{i, j} \big| \partial_{ij} \eta_R \big| \leq c_\eta \hspace{10pt}\text{ on $\mathbb{R}^n$}, \end{align*} where $c_\eta$ is a positive constant  depending only on $\eta$, therefore (\ref{approx_2nd_moment}) can be reduced to 
\begin{align}\label{approx_2nd_moment_time}\dfrac{\mathrm{d}}{\mathrm{d} t} \int_{\mathbb{R}^n} u \hspace{1pt}\eta_R \hspace{3pt}\lesssim\hspace{3pt}& \int_{\mathbb{R}^n} u  +  \int_{\mathbb{R}^n}\int_{\mathbb{R}^n} \dfrac{u(x, t) \hspace{1pt}u(y, t) }{| x - y |^{n - 2}}  \hspace{3pt}\lesssim\hspace{3pt}M + \big\| u(\cdot, t ) \big\|_{\frac{2n}{n+2}}^2.
\end{align}The second estimate above used Hardy-Littlewood-Sobolev inequality. Note that  we can apply the conservation of the total mass of $u$ and (\ref{unif_bound_u_near t infity}) to bound $\| u \|_{\frac{2n}{n+2}}$. Then we integrate the variable $t$ in (\ref{approx_2nd_moment_time}) and obtain \begin{align*} \int_{B_R \times \{ t \}} u \hspace{1pt} |x|^2 \hspace{2pt}\mathrm{d} x \hspace{2pt}\leq\hspace{2pt} \int_{\mathbb{R}^n \times \{t\} } u \hspace{1pt}\eta_R \hspace{2pt}\leq\hspace{2pt} \int_{\mathbb{R}^n} u_0 \hspace{1pt}| x |^2 \hspace{2pt}\mathrm{d} x + c \hspace{1pt} t\left( M + M^{1 + \frac{2}{n}} [\hspace{1pt}u\hspace{1pt}]_\infty^{1 - \frac{2}{n}}\right).
\end{align*}Here $c > 0$ is an universal constant. This  estimate holds for any $R > 0$. By taking $R \rightarrow \infty$ in the above, it turns out  \begin{align}\label{estimate_weighted_u} \int_{\mathbb{R}^n \times \{ t \}} u \hspace{1pt} |x|^2 \hspace{2pt}\mathrm{d} x  \hspace{2pt}\leq\hspace{2pt} \int_{\mathbb{R}^n} u_0 \hspace{1pt}| x |^2 \hspace{2pt}\mathrm{d} x + c \hspace{1pt} t\left( M + M^{1 + \frac{2}{n}} [\hspace{1pt}u\hspace{1pt}]_\infty^{1 - \frac{2}{n}}\right), \hspace{10pt}\text{for all $t \geq 0$.}
\end{align}This estimate and the change of variables in (\ref{def_U}) infer \begin{align}\label{estimate of weighted of U}\int_{\mathbb{R}^n  } U (\xi, \tau ) \hspace{1.5pt} |\hspace{1pt}\xi \hspace{1pt}|^2 \hspace{2pt}\mathrm{d} \xi \hspace{3pt}\lesssim\hspace{3pt} \int_{\mathbb{R}^n} u_0 \hspace{1pt}| x |^2 \hspace{2pt}\mathrm{d} x +  M + M^{1 + \frac{2}{n}} [\hspace{1pt}u\hspace{1pt}]_\infty^{1 - \frac{2}{n}}, \hspace{15pt}\text{for all $\tau \geq -1$.}
\end{align}

We are ready to estimate the term $\mathrm{III}$. Notice that \begin{align*} \int_{\mathbb{R}^n} U\left(\cdot, s\right) \nabla V \left(\cdot, s \right) = 0, \hspace{15pt} \text{for all $s \in \mathbb{R}$.}
\end{align*}Therefore, \begin{align*} \mathrm{III} &=  \dfrac{1}{2}  \int_0^{\frac{\tau}{2}} \dfrac{ e^{ \frac{3 - n}{2}s}}{ a\left(\tau - s\right)^{1 + \frac{n}{2}}} \hspace{2pt} \mathrm{d} s\int_{\mathbb{R}^n}  U\left(z,s\right) \nabla V\left(z,s\right) \cdot \int_0^{e^{- \frac{\tau - s}{2}}} \dfrac{\mathrm{d}}{\mathrm{d} \nu} \left[ \big( \xi - \nu z \big)e^{ - \frac{\left|\hspace{0.5pt} \xi - \nu z \hspace{0.5pt}\right|^2}{4\hspace{0.5pt}\left(1 - \nu^2\right)}}\right] \mathrm{d} \nu \hspace{2pt}\mathrm{d} z.
\end{align*}Direct computations yield \begin{align}\label{estimate of differentiation} \left|\hspace{1pt} \dfrac{\mathrm{d}}{\mathrm{d} \nu} \left[ \big( \xi - \nu z \big)e^{ - \frac{\left|\hspace{0.5pt} \xi - \nu z \hspace{0.5pt}\right|^2}{4\hspace{0.5pt}\left(1 - \nu^2\right)}}\right]  \hspace{1pt}\right| \hspace{2pt}\lesssim\hspace{2pt} \left[ |\hspace{0.5pt} z \hspace{0.5pt}| + |\hspace{0.5pt} z \hspace{0.5pt}|\hspace{1pt} \dfrac{| \xi - \nu z |^2}{1 - \nu^2} + \dfrac{\nu}{\left(1 -\nu^2 \right)^2 }\hspace{1pt} | \xi - \nu z |^3 \right] e^{ - \frac{\left|\hspace{0.5pt} \xi - \nu z \hspace{0.5pt}\right|^2}{4\hspace{1pt}\left(1 - \nu^2\right)}}.
\end{align}Utilizing (\ref{estimate of weighted of U}) and (\ref{estimate of differentiation}), we obtain \begin{align}\label{L infinty estimate of III} \big\|\hspace{1pt}\mathrm{III}\hspace{1pt}\big\|_\infty\hspace{2pt}\leq\hspace{2pt}C e^{- \frac{\tau}{8}}, \hspace{15pt}\text{for any $\tau \geq 1$.}
\end{align}Here $C$ is a positive constant depending on $M$, $[u]_\infty$ and the second moment of $u_0$. In addition, by Fubini's theorem, the $L^1$-norm of $\mathrm{III}$ can be estimated as follows: \begin{align*} \big\|\hspace{1pt}\mathrm{III}\hspace{1pt}\big\|_1\hspace{2pt}\leq\hspace{2pt}C e^{- \frac{\tau}{8}}, \hspace{15pt}\text{for any $\tau \geq 1$.}
\end{align*}
The proof is then completed by (\ref{decomp of e to tau over 2 U2}), (\ref{estimate of term I}), (\ref{estimate of times e tau 2 of U2}), (\ref{est of quantity in times in L1}), (\ref{L infinty estimate of III}) and the last estimate.
\end{proof}

\subsubsection{\normalsize Optimal approximation of $U_2$} \label{sec_4_2_2}\vspace{0.5pc}

With the assumption that the initial density has a finite second moment, we give an optimal approximation of the component $U_2$ in this section. More precisely, we have

\begin{prop}\label{optimal decay rate U2}
Under the same assumptions as in Lemma \ref{temporal optima of U2}, there exists a constant $C > 0$ depending on $n$, $[u]_\infty$, $M$ and the second moment of $u_0$ so that the followings hold for any $p \in [1, \infty]$ and $\tau \geq 1$. 
\begin{enumerate}
\item[$\mathrm{1}.$] If $n = 3$, then \begin{align*}\left\| \hspace{1pt}U_2 + M^2 e^{- \frac{\tau}{2}} \int_0^\infty e^{\frac{s}{2}} S_3(s) \hspace{1pt}\Big[ \mathrm{div}\Big(\hspace{1pt} \mathcal{G}_3 \nabla \mathcal{V}_3\Big) \hspace{1pt}\Big] \hspace{1pt}\mathrm{d} s + c_1 \hspace{0.5pt}\tau \hspace{0.5pt} e^{-\tau} \left( \frac{1}{2} - \frac{|\hspace{0.5pt}\xi\hspace{0.5pt}|^2}{12} \right) e^{-\frac{|\hspace{0.5pt}\xi\hspace{0.5pt}|^2}{4}} \hspace{1pt}\right\|_p \hspace{2pt}\leq\hspace{2pt} C\hspace{0.5pt} e^{-\tau}.\end{align*}Here $c_1$ is defined in Theorem \ref{thm_2};
\item[$\mathrm{2}.$] If $n = 4$, then  \begin{align*} \left\| \hspace{1pt}U_2 - c_2\hspace{0.5pt} \tau \hspace{0.5pt} e^{-\tau} \left( \frac{1}{2} - \frac{|\hspace{0.5pt}\xi\hspace{0.5pt}|^2}{16} \right) e^{-\frac{|\hspace{0.5pt}\xi\hspace{0.5pt}|^2}{4}} \hspace{1pt}\right\|_p \hspace{2pt}\leq\hspace{2pt} C \hspace{0.5pt}e^{-\tau}.
\end{align*}Here $c_2$ is defined in Theorem \ref{thm_2};
\item[$\mathrm{3}.$] If $n \geq 5$, then \begin{align*}
\left\|\hspace{1pt}U_2\hspace{1pt}\right\|_p \hspace{2pt}\leq\hspace{2pt} C\hspace{0.5pt}e^{-\tau}.
\end{align*}\end{enumerate}
\end{prop}
Proposition \ref{weight asymptotics of u} follows easily from this proposition and the change of variables in (\ref{def_U}). Before we prove Proposition \ref{optimal decay rate U2}, let us introduce two lemmas.
\begin{lem}\label{taylor expansion estimate of gaussian} For any $\xi, z \in \mathbb{R}^n$ and $s \geq 1$, it satisfies
\begin{align}\label{taylor_expansion}
\dfrac{1}{(1-e^{-s})^{\frac{n}{2}}} e^{-\frac{\left| \xi - e^{-\frac{s}{2}} z \right|^2}{4(1-e^{-s})}} = e^{-\frac{|\xi|^2}{4}} + \frac{\xi \cdot z}{2} e^{-\frac{|\xi|^2}{4}} e^{-\frac{s}{2}} + \left( n -\frac{|\hspace{0.5pt}\xi\hspace{0.5pt}|^2 + |\hspace{0.5pt}z\hspace{0.5pt}|^2}{2}  + \frac{(\xi \cdot z)^2}{4} \right) e^{-\frac{|\xi|^2}{4}} e^{-s} + \mathrm{Rem}.
\end{align}
The remainder term $\mathrm{Rem} = \mathrm{Rem}(n, s, \xi,z)$ is bounded from above by
\begin{align}\label{remainder_estimate}
\left|\hspace{0.5pt} \mathrm{Rem} \hspace{0.5pt}\right| \hspace{2pt}\lesssim\hspace{2pt} \left( 1 + |\hspace{0.5pt}\xi - r_0 \hspace{0.5pt}z\hspace{0.5pt}|^6 + |\hspace{0.5pt}z\hspace{0.5pt}|^6 \right)e^{-\frac{\left|\hspace{0.5pt} \xi - r_0 z \hspace{0.5pt}\right|^2}{4\left(1- r_0^2\right)}} e^{-\frac{3s}{2}}, \hspace{15pt}\text{for all $\xi, z \in \mathbb{R}^n$ and $s \geq 1$.}
\end{align}
Here $r_0 = r_0\left(n, s, \xi,z\right) \in \left(0,e^{- \frac{1}{2}}\right)$. Moreover, the $L^1$-norm of $\mathrm{Rem}$ can be estimated by
\begin{align}\label{remainder_estimate_1}
\int_{\mathbb{R}^n} \left| \mathrm{Rem} \right| \mathrm{d}\xi \hspace{2pt}\lesssim\hspace{2pt} e^{-\frac{3s}{2}} \left( 1 + |z|^{n+6} \right), \hspace{15pt}\text{for all $z \in \mathbb{R}^n$ and $s \geq 1$.}
\end{align}
\end{lem}\begin{proof}[\bf Proof] (\ref{taylor_expansion}) and (\ref{remainder_estimate}) are direct results of the Taylor expansion of the function on the left-hand side of (\ref{taylor_expansion}) with respect to $r = e^{ - \frac{s}{2}}$. In addition, by  (\ref{remainder_estimate}), we have, for any $z \in \mathbb{R}^n$, $s \geq 1$ and $n \geq 3$, that
\begin{align}\label{id_8_13}
\int_{\mathbb{R}^n} \left|\hspace{0.5pt} \mathrm{Rem} \hspace{0.5pt}\right| \mathrm{d}\xi \hspace{2pt}\lesssim\hspace{2pt} e^{-\frac{3s}{2}} \left( 1 + |z|^6 \right) \int_{\mathbb{R}^n} e^{-\frac{\left| \xi - r_0z \right|^2}{8}} \mathrm{d}\xi.
\end{align}Since  $r_0 \in \left(0,e^{- \frac{1}{2}}\right)$, it turns out
\begin{align}\label{last two estimates}
\int_{B_{2|z|}} e^{-\frac{\left| \xi - r_0 z \right|^2}{8}} \hspace{1pt}\mathrm{d}\xi \hspace{2pt}\lesssim\hspace{2pt} |z|^n \quad \mbox{and} \quad \int_{B\mystrut^c_{2|z|}} e^{-\frac{\left| \xi - r_0 z \right|^2}{8}} \mathrm{d}\xi \hspace{2pt}\lesssim\hspace{2pt} \int_{B\mystrut^c_{2|z|}} e^{-\frac{|\xi|^2}{32}} \mathrm{d}\xi \hspace{2pt}\lesssim\hspace{2pt} 1.
\end{align}
  Therefore, applying (\ref{last two estimates}) to (\ref{id_8_13}) yields (\ref{remainder_estimate_1}).
\end{proof}

\begin{lem}\label{wighted estimate of w star}
Let $n = 3$ and define
\begin{align*}
\mathcal{W}_\star :=  \int_0^\infty e^{\frac{s}{2}} S_3(s) \hspace{1pt}\Big[ \mathrm{div}\Big(\hspace{1pt} \mathcal{G}_3 \nabla \mathcal{V}_3\Big) \hspace{1pt}\Big] \hspace{1pt}\mathrm{d} s.
\end{align*}
It holds
\begin{align*}
\int_{\mathbb{R}^3} \left|\hspace{0.5pt}\mathcal{W}_\star\hspace{0.5pt}\right| \hspace{1pt}|\hspace{0.5pt}\xi\hspace{0.5pt}|^k  \hspace{2pt} \mathrm{d}\xi \hspace{2pt}<\hspace{2pt} \infty \quad \mbox{and} \quad \int_{\mathbb{R}^3}  \left|\hspace{0.5pt}\nabla_\xi \mathcal{W}_\star\hspace{0.5pt}\right| \hspace{1pt}|\hspace{0.5pt}\xi\hspace{0.5pt}|^k \hspace{1pt} \mathrm{d}\xi \hspace{2pt}<\hspace{2pt} \infty, \hspace{15pt} \mbox{for all $k \geq 0$.}
\end{align*}
\end{lem}

\begin{proof}[\bf Proof]
Using (\ref{semi group of U variable}), we have
\begin{align*}
\mathcal{W}_\star\left(\xi\right) = (4\pi)^{-\frac{3}{2}} \int_0^\infty \int_{\mathbb{R}^3}\frac{e^{\frac{s}{2}}}{(1-e^{-s})^{\frac{3}{2}}} e^{-\frac{\left|\xi - e^{-\frac{s}{2}}z\right|^2}{4(1-e^{-s})}}\mathrm{div}\Big(\hspace{1pt} \mathcal{G}_3 \nabla \mathcal{V}_3\Big) \hspace{2pt} \mathrm{d}z \hspace{1pt}\mathrm{d}s.
\end{align*}
By Taylor's theorem, it holds
\begin{align}\label{id_8_14}
\dfrac{1}{(1-e^{-s})^{\frac{3}{2}}} e^{-\frac{\left| \xi - e^{-\frac{s}{2}} z \right|^2}{4(1-e^{-s})}} = e^{-\frac{|\xi|^2}{4}} + \frac{\xi \cdot z}{2} e^{-\frac{|\xi|^2}{4}} e^{-\frac{s}{2}} + \mathrm{Rem}_1, \hspace{15pt} \text{for all $\xi, z \in \mathbb{R}^3$ and $s \geq 1$.}
\end{align}
Here with some $r_1 = r_1\left(s, \xi,z\right) \in \left(0, e^{- \frac{1}{2}}\right)$, the remainder term $\mathrm{Rem}_1 = \mathrm{Rem}_1\left(s, \xi,z \right)$ satisfies
\begin{align}\label{id_8_15}
\left|\hspace{0.5pt} \mathrm{Rem}_1 \hspace{0.5pt} \right| \hspace{2pt}\lesssim\hspace{2pt} \left( 1 + |\xi - r_1 z|^4 + |z|^4 \right)e^{-\frac{\left| \xi - r_1 z \right|^2}{4(1-r_1^2)}} e^{-s}, \hspace{15pt}\text{ for all $\xi, z \in \mathbb{R}^3$ and $s \geq 1$.}
\end{align}
Since
\begin{align}\label{id_8_15_1}
\int_{\mathbb{R}^3} \mathrm{div} \big(\hspace{1pt} \mathcal{G}_3 \nabla \mathcal{V}_3 \hspace{1pt}\big) \mathrm{d}z = 0 \quad \mbox{and} \quad \int_{\mathbb{R}^3} z \hspace{1pt}\mathrm{div} \big(\hspace{1pt} \mathcal{G}_3 \nabla \mathcal{V}_3 \hspace{1pt}\big) \hspace{1pt} \mathrm{d}z = 0,
\end{align}
by using (\ref{id_8_14}), it turns out
\begin{align}\label{id_8_16}
(4\pi)^{\frac{3}{2}} \mathcal{W}_\star = \mathcal{W}_1 + \mathcal{W}_2,
\end{align}
where
\begin{align*}
&\mathcal{W}_1 := \int_1^\infty e^{\frac{s}{2}} \mathrm{d}s \int_{\mathbb{R}^3} \mathrm{Rem}_1 \cdot \mathrm{div}\Big(\hspace{1pt} \mathcal{G}_3 \nabla \mathcal{V}_3\Big) \mathrm{d}z; \hspace{5pt}\mathcal{W}_2 := \int_0^1 \int_{\mathbb{R}^3}\frac{e^{\frac{s}{2}}}{(1-e^{-s})^{\frac{3}{2}}} e^{-\frac{\left|\xi - e^{-\frac{s}{2}}z\right|^2}{4(1-e^{-s})}}\mathrm{div}\Big(\hspace{1pt} \mathcal{G}_3 \nabla \mathcal{V}_3\Big) \hspace{1pt}\mathrm{d}z \hspace{1pt}\mathrm{d}s.
\end{align*}
In view of (\ref{id_8_15}) and $r_1 = r_1(s, \xi,z) \in \left(0,e^{- \frac{1}{2}}\right)$, 
\begin{align}\label{id_8_17}
\int_{\mathbb{R}^3} |\hspace{0.5pt}\xi\hspace{0.5pt}|^k \left|\hspace{0.5pt}\mathcal{W}_1\hspace{0.5pt}\right| \hspace{2pt}\mathrm{d}\xi &\hspace{2pt}\lesssim\hspace{2pt} \int_1^\infty e^{-\frac{s}{2}} \mathrm{d}s \int_{\mathbb{R}^3} \int_{\mathbb{R}^3} |\hspace{0.5pt}\xi\hspace{0.5pt}|^k \left( 1 + |\xi - r_1 z|^4 + |z|^4 \right)e^{-\frac{\left| \xi - r_1 z \right|^2}{4(1-r_1^2)}} \left| \mathrm{div}\Big(\hspace{1pt} \mathcal{G}_3 \nabla \mathcal{V}_3\Big) \right| \mathrm{d}z \hspace{2pt}\mathrm{d}\xi \notag\\[2mm]
&\hspace{2pt}\lesssim\hspace{2pt} \int_1^\infty e^{-\frac{s}{2}} \mathrm{d}s \int_{\mathbb{R}^3} \int_{\mathbb{R}^3} |\hspace{0.5pt}\xi\hspace{0.5pt}|^k \left( 1 + |\hspace{0.5pt}z\hspace{0.5pt}|^4 \right)e^{-\frac{\left| \xi - r_1 z \right|^2}{8}} \left| \mathrm{div}\Big(\hspace{1pt} \mathcal{G}_3 \nabla \mathcal{V}_3\Big) \right| \mathrm{d}z \hspace{2pt} \mathrm{d}\xi.
\end{align}
As in the derivation of (\ref{remainder_estimate_1}), we have
\begin{align*}
\int_{|\hspace{0.5pt}\xi\hspace{0.5pt}| \hspace{0.5pt}\leq\hspace{0.5pt} 2|\hspace{0.5pt}z\hspace{0.5pt}|} |\hspace{0.5pt}\xi\hspace{0.5pt}|^k e^{-\frac{\left| \xi - r_1 z \right|^2}{8}}\mathrm{d}\xi \hspace{2pt}\leq\hspace{2pt} \int_{|\xi| \hspace{0.5pt}\leq\hspace{0.5pt} 2|z|} |\hspace{0.5pt}\xi\hspace{0.5pt}|^k \mathrm{d} \xi \hspace{2pt}\lesssim_k \hspace{2pt} |\hspace{0.5pt}z\hspace{0.5pt}|^{k+3}
\end{align*}
and
\begin{align*}
\int_{|\xi| \geq 2|z|} |\hspace{0.5pt}\xi\hspace{0.5pt}|^k e^{-\frac{\left| \xi - r_1 z \right|^2}{8}}\mathrm{d}\xi \hspace{2pt}\lesssim\hspace{2pt} \int_{|\hspace{0.5pt}\xi\hspace{0.5pt}| \hspace{0.5pt}\geq\hspace{0.5pt} 2|\hspace{0.5pt}z\hspace{0.5pt}|} |\hspace{0.5pt}\xi\hspace{0.5pt}|^k e^{-\frac{|\xi|^2}{32}} \mathrm{d}\xi \hspace{2pt}\lesssim_k\hspace{2pt} 1.
\end{align*}
Therefore, applying the last two estimates to  (\ref{id_8_17}) induces
\begin{align}\label{id_8_18}
\int_{\mathbb{R}^3} |\hspace{0.5pt}\xi\hspace{0.5pt}|^k \left|\hspace{0.5pt}\mathcal{W}_1\hspace{0.5pt}\right| \mathrm{d}\xi &\hspace{2pt}\lesssim_k\hspace{2pt} \int_1^\infty e^{-\frac{s}{2}} \mathrm{d}s \int_{\mathbb{R}^3} \left( 1 + |\hspace{0.5pt}z\hspace{0.5pt}|^{k+7} \right) \left| \mathrm{div}\Big(\hspace{1pt} \mathcal{G}_3 \nabla \mathcal{V}_3\Big) \right| \mathrm{d}z \hspace{2pt}<\hspace{2pt} \infty, \hspace{15pt} \mbox{for all $k \geq 0$.}
\end{align}
Furthermore, it holds
\begin{align}\label{id_8_19}
\int_{\mathbb{R}^3} |\hspace{0.5pt}\xi\hspace{0.5pt}|^k \left|\hspace{0.5pt}\mathcal{W}_2\hspace{0.5pt}\right| \mathrm{d}\xi &\hspace{2pt}\lesssim\hspace{2pt} \int_{\mathbb{R}^3} |\hspace{0.5pt}\xi\hspace{0.5pt}|^k \int_0^1 \int_{\mathbb{R}^3}\frac{e^{\frac{s}{2}}}{(1-e^{-s})^{\frac{3}{2}}} e^{-\frac{\left|\xi - e^{-\frac{s}{2}}z\right|^2}{4(1-e^{-s})}}\left|\mathrm{div}\Big(\hspace{1pt} \mathcal{G}_3 \nabla \mathcal{V}_3\Big)\right| \mathrm{d}z \hspace{2pt}\mathrm{d}s \hspace{2pt}\mathrm{d}\xi \notag\\[2mm]
&\hspace{2pt}\lesssim_k\hspace{2pt} \int_0^1 \frac{\mathrm{d} s}{(1-e^{-s})^{\frac{3}{2}}}   \int_{\mathbb{R}^3} \int_{\mathbb{R}^3} \left( \left| \xi - e^{-\frac{s}{2}} z \right|^k + e^{-\frac{ks}{2}} |\hspace{0.5pt}z\hspace{0.5pt}|^k \right) e^{-\frac{\left|\xi - e^{-\frac{s}{2}}z\right|^2}{4(1-e^{-s})}}\left|\mathrm{div}\Big(\hspace{1pt} \mathcal{G}_3 \nabla \mathcal{V}_3\Big)\right| \mathrm{d}z \hspace{2pt}\mathrm{d}\xi \notag\\[2mm]
&\hspace{2pt}\lesssim_k\hspace{2pt} \int_0^1 \mathrm{d}s \int_{\mathbb{R}^3} \left( \left(1 - e^{-s}\right)^{\frac{k}{2}} + e^{-\frac{ks}{2}} |\hspace{0.5pt}z\hspace{0.5pt}|^k \right) \left|\mathrm{div}\Big(\hspace{1pt} \mathcal{G}_3 \nabla \mathcal{V}_3\Big)\right| \mathrm{d}z < \infty, \hspace{15pt} \mbox{for all $k \geq 0$.}
\end{align}
Combining (\ref{id_8_16}), (\ref{id_8_18}) and (\ref{id_8_19}), we obtain
\begin{align*}
\int_{\mathbb{R}^3} |\hspace{0.5pt}\xi\hspace{0.5pt}|^k \left|\hspace{0.5pt}\mathcal{W}_\star\hspace{0.5pt}\right| \mathrm{d}\xi \hspace{2pt}<\hspace{2pt} \infty, \hspace{15pt} \mbox{for all $k \geq 0$.}
\end{align*}

In the sense of distribution,
\begin{align*}
\nabla_\xi \mathcal{W}_\star\left(\xi\right) = -\frac{1}{2} (4\pi)^{-\frac{3}{2}} \int_0^\infty \int_{\mathbb{R}^3}\frac{e^{\frac{s}{2}}}{(1-e^{-s})^{\frac{5}{2}}} e^{-\frac{\left|\xi - e^{-\frac{s}{2}}z\right|^2}{4(1-e^{-s})}}\mathrm{div}\Big(\hspace{1pt} \mathcal{G}_3 \nabla \mathcal{V}_3\Big) \big( \xi - e^{-\frac{s}{2}} z \big) \hspace{2pt} \mathrm{d}z \hspace{1pt}\mathrm{d}s.
\end{align*}
Still by Taylor's theorem, we have
\begin{align}\label{id_8_20}
\dfrac{1}{(1-e^{-s})^{\frac{5}{2}}} e^{-\frac{\left| \xi - e^{-\frac{s}{2}} z \right|^2}{4(1-e^{-s})}} \left( \xi - e^{-\frac{s}{2}} z \right) = e^{-\frac{|\xi|^2}{4}} \xi + e^{-\frac{|\xi|^2}{4}} \left( \frac{\xi \cdot z}{2} \xi - z \right) e^{-\frac{s}{2}} + \mathrm{Rem}_2,
\end{align}
for all $\xi, z \in \mathbb{R}^n$ and $s \geq 1$. Here for some $r_2 = r_2(s, \xi,z) \in \left(0,e^{- \frac{1}{2}}\right)$,  $\mathrm{Rem}_2 = \mathrm{Rem}_2(s, \xi,z )$ satisfies
\begin{align}\label{id_8_21}
\left|\hspace{0.5pt} \mathrm{Rem}_2 \hspace{0.5pt}\right| \hspace{2pt}\lesssim\hspace{2pt} \left( 1 + |\xi - r_2 z|^5 + |\hspace{0.5pt}z\hspace{0.5pt}|^5 \right)e^{-\frac{\left| \xi - r_2 z \right|^2}{4\left(1-r_2^2\right)}} e^{-s}, \hspace{15pt}\text{for all $\xi, z \in \mathbb{R}^3$ and $s \geq 1$.}
\end{align}
In light of (\ref{id_8_20}) and (\ref{id_8_15_1}), we have
\begin{align}\label{id_W3_W4}
-2(4\pi)^{\frac{3}{2}} \hspace{1pt} \nabla_\xi \mathcal{W}_\star\left(\xi\right) = \mathcal{W}_3 + \mathcal{W}_4,
\end{align}
where
\begin{align*} &\mathcal{W}_3 := \int_1^\infty e^{\frac{s}{2}} \mathrm{d}s \int_{\mathbb{R}^3} \mathrm{div}\Big(\hspace{1pt} \mathcal{G}_3 \nabla \mathcal{V}_3\Big) \mathrm{Rem}_2 \, \mathrm{d}z;  \\[2mm]
&\mathcal{W}_4 := \int_0^1 \int_{\mathbb{R}^3}\frac{e^{\frac{s}{2}}}{(1-e^{-s})^{\frac{5}{2}}} e^{-\frac{\left|\xi - e^{-\frac{s}{2}}z\right|^2}{4(1-e^{-s})}}\mathrm{div}\Big(\hspace{1pt} \mathcal{G}_3 \nabla \mathcal{V}_3\Big) \left( \xi - e^{-\frac{s}{2}} z \right) \mathrm{d}z \hspace{2pt}\mathrm{d}s.
\end{align*}
By (\ref{id_8_21}) and the fact that $r_2 = r_2(s, \xi,z) \in \left(0,e^{- \frac{1}{2}}\right)$,
\begin{align*}
\int_{\mathbb{R}^3} |\hspace{0.5pt}\xi\hspace{0.5pt}|^k \left|\hspace{0.5pt}\mathcal{W}_3\hspace{0.5pt}\right| \mathrm{d}\xi &\hspace{2pt}\lesssim\hspace{2pt} \int_1^\infty e^{-\frac{s}{2}} \mathrm{d}s \int_{\mathbb{R}^3} \int_{\mathbb{R}^3} |\hspace{0.5pt}\xi\hspace{0.5pt}|^k \left( 1 + |\xi - r_2 z|^5 + |\hspace{0.5pt}z\hspace{0.5pt}|^5 \right)e^{-\frac{\left| \xi - r_2 z \right|^2}{4\left(1-r_2^2\right)}} \left| \mathrm{div}\Big(\hspace{1pt} \mathcal{G}_3 \nabla \mathcal{V}_3\Big) \right| \mathrm{d}z \hspace{2pt}\mathrm{d}\xi \notag\\[2mm]
&\hspace{2pt}\lesssim\hspace{2pt} \int_1^\infty e^{-\frac{s}{2}} \mathrm{d}s \int_{\mathbb{R}^3} \int_{\mathbb{R}^3} |\hspace{0.5pt}\xi\hspace{0.5pt}|^k \left( 1 + |\hspace{0.5pt}z\hspace{0.5pt}|^5 \right)e^{-\frac{\left| \xi - r_2 z \right|^2}{8}} \left| \mathrm{div}\Big(\hspace{1pt} \mathcal{G}_3 \nabla \mathcal{V}_3\Big) \right| \mathrm{d}z \hspace{2pt}\mathrm{d}\xi.
\end{align*}
As in (\ref{id_8_18}), the last inequality infers
\begin{align}\label{id_8_22}
\int_{\mathbb{R}^3} |\hspace{0.5pt}\xi\hspace{0.5pt}|^k \left|\hspace{0.5pt}\mathcal{W}_3\hspace{0.5pt}\right| \mathrm{d}\xi \hspace{2pt}\lesssim_k\hspace{2pt} \int_1^\infty e^{-\frac{s}{2}} \mathrm{d}s \int_{\mathbb{R}^3} \left( 1 + |\hspace{0.5pt}z\hspace{0.5pt}|^{k+8} \right) \left| \mathrm{div}\Big(\hspace{1pt} \mathcal{G}_3 \nabla \mathcal{V}_3\Big) \right| \mathrm{d}z \hspace{2pt}<\hspace{2pt} \infty, \hspace{15pt} \mbox{for all $k \geq 0$.}
\end{align}
For $\mathcal{W}_4$, it can be estimated by
\begin{align}\label{id_8_23}
&\int_{\mathbb{R}^3} |\hspace{0.5pt}\xi\hspace{0.5pt}|^k \left|\hspace{0.5pt}\mathcal{W}_4\hspace{0.5pt}\right| \mathrm{d}\xi  \hspace{2pt}\lesssim \hspace{2pt}\int_{\mathbb{R}^3} |\hspace{0.5pt}\xi\hspace{0.5pt}|^k \int_0^1 \int_{\mathbb{R}^3}\frac{e^{\frac{s}{2}}}{(1-e^{-s})^{\frac{5}{2}}} \left|\xi - e^{-\frac{s}{2}}z\right| e^{-\frac{\left|\xi - e^{-\frac{s}{2}}z\right|^2}{4(1-e^{-s})}}\left|\mathrm{div}\Big(\hspace{1pt} \mathcal{G}_3 \nabla \mathcal{V}_3\Big)\right| \mathrm{d}z \hspace{2pt}\mathrm{d}s \hspace{2pt}\mathrm{d}\xi\hspace{2pt}\notag\\[2mm]
&\qquad\lesssim_k \hspace{2pt}\int_0^1 \frac{\mathrm{d}s}{(1-e^{-s})^{\frac{5}{2}}}  \int_{\mathbb{R}^3} \int_{\mathbb{R}^3} \left( \left| \xi - e^{-\frac{s}{2}} z \right|^k + e^{-\frac{ks}{2}} |\hspace{0.5pt}z\hspace{0.5pt}|^k \right) \left|\xi - e^{-\frac{s}{2}}z\right| e^{-\frac{\left|\xi - e^{-\frac{s}{2}}z\right|^2}{4(1-e^{-s})}}\left|\mathrm{div}\Big(\hspace{1pt} \mathcal{G}_3 \nabla \mathcal{V}_3\Big)\right| \mathrm{d}z \hspace{2pt} \mathrm{d}\xi \notag\\[2mm]
&\qquad\lesssim_k\hspace{2pt} \int_0^1 \frac{\mathrm{d} s}{(1-e^{-s})^{\frac{1}{2}}} \int_{\mathbb{R}^3} \left( \left(1 - e^{-s}\right)^{\frac{k}{2}} + e^{-\frac{ks}{2}} |\hspace{0.5pt}z\hspace{0.5pt}|^k \right) \left|\mathrm{div}\Big(\hspace{1pt} \mathcal{G}_3 \nabla \mathcal{V}_3\Big)\right| \mathrm{d}z < \infty, \hspace{15pt} \mbox{for all $k \geq 0$.}
\end{align}
Therefore, by (\ref{id_W3_W4})--(\ref{id_8_23}), we obtain
\begin{align*}
\int_{\mathbb{R}^3} |\hspace{0.5pt}\xi\hspace{0.5pt}|^k \left|\hspace{0.5pt}\nabla_\xi \mathcal{W}_\star\right| \mathrm{d}\xi \hspace{2pt}<\hspace{2pt} \infty, \hspace{15pt} \mbox{for all $k \geq 0$.}
\end{align*}The proof is completed.
\end{proof}Now we complete this section by proving Proposition \ref{optimal decay rate U2}.
\begin{proof}[\bf Proof of Proposition \ref{optimal decay rate U2}]
Throughout the proof,  $C$ is a positive constant depending on $n$, $[u]_\infty$, $M$ and the second moment of $u_0$. Using the representation of $U_2$ in (\ref{def_U_2}), we decompose $U_2$ into
\begin{align}\label{decomp of e to tau U2}
U_2 =\left(4 \pi\right)^{-\frac{n}{2}} \hspace{1pt}\big[\hspace{1pt}U_{2, 1} + U_{2, 2} + U_{2, 3} + U_{2, 4} \hspace{1pt}\big],
\end{align}
where
\begin{align*}
&U_{2, 1} := \int_0^{\tau} \dfrac{f_n\left(s \right)}{ a\left(\tau - s\right)^{ \frac{n}{2}}} \hspace{2pt} \mathrm{d} s\int_{\mathbb{R}^n} U^{(0)}(z) \nabla V^{(0)}(z) \cdot \nabla_z e^{ -\frac{\big| \xi - e^{-\frac{\tau - s}{2}} z \hspace{1pt} \big|^2}{4a\left(\tau-s\right)}}\mathrm{d} z;\\[2mm]
&U_{2, 2} := - \int_0^{\tau} \dfrac{f_n\left(s \right) e^{-\frac{s}{2}}}{ a\left(\tau - s\right)^{ \frac{n}{2}}} \hspace{2pt} \mathrm{d} s\int_{\mathbb{R}^n} \Big[ U^{(0)}(z) \nabla \mathcal{V}_n^{(1)}(z) + \mathcal{G}_n^{(1)}(z) \nabla V^{(0)}(z) \Big] \cdot \nabla_z e^{ -\frac{\big| \xi - e^{-\frac{\tau - s}{2}} z \hspace{1pt} \big|^2}{4a\left(\tau-s\right)}}\mathrm{d} z;\\[2mm]
&U_{2, 3} := \int_0^{\tau} \dfrac{f_n\left(s \right) e^{-s}}{ a\left(\tau - s\right)^{ \frac{n}{2}}} \hspace{2pt} \mathrm{d} s\int_{\mathbb{R}^n} \mathcal{G}_n^{(1)}(z) \nabla \mathcal{V}_n^{(1)}(z) \cdot \nabla_z e^{ -\frac{\big| \xi - e^{-\frac{\tau - s}{2}} z \hspace{1pt} \big|^2}{4a\left(\tau-s\right)}}\mathrm{d} z;\\[2mm]
&U_{2, 4} := \int_0^{\tau} \dfrac{f_n\left(s\right)}{ a\left(\tau - s\right)^{\frac{n}{2}}} \hspace{2pt} \mathrm{d} s\int_{\mathbb{R}^n} \Big[ U\left(z,s\right) \nabla V\left(z,s\right) - U^{(1)}(z,s) \nabla V^{(1)}(z,s)\Big] \cdot \nabla_ze^{ -\frac{\big| \xi - e^{-\frac{\tau - s}{2}} z \hspace{1pt} \big|^2}{4a\left(\tau-s\right)}} \mathrm{d} z.
\end{align*}
In the above definitions,  $U^{(0)}$ and $V^{(0)}$ are the same as in the proof of Lemma \ref{temporal optima of U2}. The functions $\mathcal{G}_n^{(1)}$ and $\mathcal{V}_n^{(1)}$ are given by 
\begin{align*}
&\mathcal{G}_n^{(1)} := B_0\cdot \nabla \mathcal{G}_n + M^2 \delta_{n,3} \int_0^\infty e^{\frac{s}{2}} S_n(s) \hspace{1pt}\Big[ \mathrm{div}\Big(\hspace{1pt} \mathcal{G}_n \nabla \mathcal{V}_n\Big) \hspace{1pt}\Big] \hspace{1pt}\mathrm{d} s  \qquad \text{and}\qquad \mathcal{V}_n^{(1)} = E_n * \mathcal{G}_n^{(1)}.\end{align*}
Moreover, we use $U^{(1)} $ and $ V^{(1)}$ to denote $U^{(0)} - e^{-\frac{\tau}{2}} \hspace{1pt} \mathcal{G}_n^{(1)}$ and $ V^{(0)} - e^{-\frac{\tau}{2}} \hspace{1pt}\mathcal{V}_n^{(1)}$, respectively. \vspace{0.5pc}

The remainder of the proof are devoted to estimating the terms $U_{2, 1}$, $U_{2, 2}$, $U_{2, 3}$ and $U_{2, 4}$.\vspace{0.5pc}\\
\textbf{Estimate of $U_{2, 1}$.} Integrating by parts with respect to $z$ and changing variable from $s$ to $\tau - s$, we have
\begin{align*}
U_{2, 1} = -(4\pi)^{\frac{n}{2}} \int_0^\tau e^{\frac{2-n}{2}(\tau-s)} S_n(s) \Big[ \mathrm{div} \left( U^{(0)} \nabla V^{(0)} \right)\Big] \hspace{1pt} \mathrm{d}s.
\end{align*}
Using (\ref{estimate of div of gn nabla vn}) then yields
\begin{align}\label{id_8_07}
\big\|\hspace{1pt}U_{2, 1}\hspace{1pt}\big\|_p & \hspace{2pt}\lesssim\hspace{2pt} e^{-\tau} \int_0^\tau e^{\frac{4-n}{2}(\tau-s)} e^s \left\| S_n\left(s\right) \Big[ \mathrm{div} \left( U^{(0)} \nabla V^{(0)} \right)\Big] \hspace{1pt}\right\|_p \mathrm{d} s  \hspace{2pt}\lesssim\hspace{2pt} M^2 e^{-\tau}, \hspace{15pt} \text{if  $n \geq 5$.}
\end{align}
The above estimate holds for all $p \in [1,\infty]$. Similarly, it can also be shown for all $p \in [1, \infty]$ that
\begin{align}\label{id_8_08}
\left\|\hspace{1pt} U_{2,1} + \left(4\pi\right)^{\frac{3}{2}} e^{-\frac{\tau}{2}} \int_0^{\infty} e^{\frac{s}{2}} \hspace{1pt} S_3\left(s\right) \Big[ \mathrm{div}\left(U^{(0)} \nabla V^{(0)}\right)\Big] \mathrm{d} s \hspace{1pt}\right\|_p \hspace{2pt} \lesssim\hspace{2pt} M^2 e^{-\tau}, \hspace{15pt}\text{if $n = 3$.}
\end{align}As for the case when $n = 4$, by (\ref{taylor_expansion}) in Lemma \ref{taylor expansion estimate of gaussian} and the null conditions
\begin{align*}
\int_{\mathbb{R}^4} \mathrm{div} \left( U^{(0)} \nabla V^{(0)} \right) \mathrm{d}z =  \int_{\mathbb{R}^4} z_j \hspace{1pt} \mathrm{div} \left( U^{(0)} \nabla V^{(0)} \right)  \mathrm{d}z = 0,\hspace{15pt}\text{ $j = 1, ..., 4$,}
\end{align*}
it turns out, for all $\tau \geq 1$, that
\begin{align*}
U_{2, 1} &= (\tau-1) \hspace{1pt}e^{-\tau} \hspace{1pt}e^{-\frac{|\hspace{0.5pt}\xi\hspace{0.5pt}|^2}{4}}\int_{\mathbb{R}^4} \left( \frac{|z|^2}{2} - \frac{(\xi \cdot z)^2}{4} \right) \mathrm{div} \left( U^{(0)} \nabla V^{(0)} \right)(z)\hspace{1.5pt} \mathrm{d}z   \\[2mm]
& - \int_1^\tau e^{s - \tau} \mathrm{d}s \int_{\mathbb{R}^4} \mathrm{Rem} \cdot \mathrm{div} \left( U^{(0)} \nabla V^{(0)} \right)(z)\hspace{1.5pt}\mathrm{d}z -(4\pi)^{2} \int_0^1 e^{s - \tau} S_4(s) \Big[ \mathrm{div} \left( U^{(0)} \nabla V^{(0)} \right)\Big] \mathrm{d}s.
\end{align*}
Due to the radial symmetry of $\mathrm{div} \left( U^{(0)} \nabla V^{(0)} \right)$, we have
\begin{align*}
\int_{\mathbb{R}^4} \left( \frac{|z|^2}{2} - \frac{(\xi \cdot z)^2}{4} \right)   \mathrm{div} \left( U^{(0)} \nabla V^{(0)} \right)(z)\hspace{1.5pt} \mathrm{d}z = \left( \frac{1}{2} - \frac{|\hspace{0.5pt}\xi\hspace{0.5pt}|^2}{16} \right)  \int_{\mathbb{R}^4} |z|^2 \mathrm{div} \left( U^{(0)} \nabla V^{(0)} \right)(z)\hspace{1.5pt} \mathrm{d}z.
\end{align*}
Therefore, in view of (\ref{remainder_estimate}), (\ref{remainder_estimate_1}) and the last two equalities, it follows
\begin{align}\label{id_8_09}
\left\|\hspace{1pt} U_{2,1} - \tau e^{-\tau} \left( \frac{1}{2} - \frac{|\hspace{0.5pt}\xi\hspace{0.5pt}|^2}{16} \right) e^{-\frac{|\xi|^2}{4}} \int_{\mathbb{R}^4} |z|^2 \mathrm{div} \left( U^{(0)} \nabla V^{(0)} \right) \mathrm{d}z \hspace{1pt}\right\|_p \hspace{2pt}\lesssim\hspace{2pt} M^2 e^{-\tau}, \hspace{15pt}\text{if $n = 4$ and $\tau \geq 1$.}
\end{align}
Here $p$ is also an arbitrary number in $[1, \infty]$.\vspace{0.5pc}\\
\noindent\textbf{Estimate of $U_{2,2}$.}
Still integrating by parts with respect to $z$ and changing the variable from $s$ to $\tau - s$, we obtain
\begin{align*}
U_{2, 2} = (4\pi)^{\frac{n}{2}} \int_0^\tau e^{\frac{1-n}{2}(\tau-s)} S_n(s) \Big[ \mathrm{div} \left( U^{(0)} \nabla \mathcal{V}_n^{(1)} + \mathcal{G}_n^{(1)} \nabla V^{(0)} \right)\Big] \mathrm{d}s.
\end{align*}
On the other hand, by $- \Delta V^{(0)} = U^{(0)}$ and $ - \Delta \mathcal{V}_n^{(1)} = \mathcal{G}_n^{(1)}$, we can also integrate by parts and get
\begin{align*}
  \int_{\mathbb{R}^n} z_j\hspace{0.5pt}\mathrm{div} \left( U^{(0)} \nabla \mathcal{V}_n^{(1)} + \mathcal{G}_n^{(1)} \nabla V^{(0)} \right)  \mathrm{d}z = 0, \hspace{15pt}j = 1, ..., n.
\end{align*}
In light of (\ref{taylor_expansion}) and the last two equalities, it follows
\begin{align*}
U_{2,2} &= - e^{\frac{1-n}{2}\tau} e^{-\frac{|\xi|^2}{4}} \int_1^\tau e^{\frac{n-3}{2}s} \mathrm{d}s \int_{\mathbb{R}^n} \left(\frac{|z|^2}{2} - \frac{(\xi \cdot z)^2}{4} \right) \mathrm{div} \left( U^{(0)} \nabla \mathcal{V}_n^{(1)} + \mathcal{G}_n^{(1)} \nabla V^{(0)} \right) \mathrm{d}z \\[2mm]
& + \int_1^\tau e^{\frac{1-n}{2}(\tau-s)} \mathrm{d}s \int_{\mathbb{R}^n} \mathrm{Rem} \cdot \mathrm{div} \left( U^{(0)} \nabla \mathcal{V}_n^{(1)} + \mathcal{G}_n^{(1)} \nabla V^{(0)} \right) \mathrm{d}z \\[2mm]
& + (4\pi)^{\frac{n}{2}} \int_0^1 e^{\frac{1-n}{2}(\tau-s)} S_n(s) \Big[ \mathrm{div} \left( U^{(0)} \nabla \mathcal{V}_n^{(1)} + \mathcal{G}_n^{(1)} \nabla V^{(0)} \right)\Big] \mathrm{d}s, \hspace{20pt}\text{for all $\tau \geq 1$.}
\end{align*}
Note that (\ref{remainder_estimate}) infers 
\begin{align*}
\left\| \int_1^\tau e^{\frac{1-n}{2}(\tau-s)} \mathrm{d}s \int_{\mathbb{R}^n} \mathrm{Rem} \cdot \mathrm{div} \left( U^{(0)} \nabla \mathcal{V}_n^{(1)} + \mathcal{G}_n^{(1)} \nabla V^{(0)} \right) \mathrm{d}z \right\|_p \lesssim_{M, B_0} e^{-\tau}, \hspace{5pt}\text{for all $p \in [1, \infty]$ and $n \geq 3$.}
\end{align*}
Therefore, we can obtain
\begin{align}\label{id_8_10}
\big\|\hspace{1pt}U_{2, 2}\hspace{1pt}\big\|_p \hspace{2pt}\lesssim_{M, B_0}\hspace{1pt} e^{-\tau}, \hspace{15pt}\text{for all $\tau \geq 1$, $p \in [1, \infty]$ and $n \geq 4$.}
\end{align}
When $n = 3$, as in the estimate of $U_{2,1}$ for the $n=4$ case, it turns out with an aid of Lemma \ref{wighted estimate of w star} that
\begin{align}\label{id_8_11}
\left\| U_{2, 2} + \tau e^{-\tau} \left( \frac{1}{2} - \frac{|\xi|^2}{12} \right) e^{-\frac{|\xi|^2}{4}} \int_{\mathbb{R}^3} |z|^2 \mathrm{div} \left( U^{(0)} \nabla \mathcal{V}_3^{(1)} + \mathcal{G}_3^{(1)} \nabla V^{(0)} \right) \mathrm{d}z \right\|_p \lesssim_{M, B_0} e^{-\tau},
\end{align}
for all $\tau \geq 1$ and $p \in [1, \infty]$.\vspace{0.5pc}\\
\noindent\textbf{Estimate of $U_{2, 3}$.} Similarly, we have
\begin{align*}
U_{2, 3} = -(4\pi)^{\frac{n}{2}} \int_0^\tau e^{-\frac{n}{2}(\tau-s)} S_n(s) \Big[ \mathrm{div} \left( \mathcal{G}_n^{(1)} \nabla \mathcal{V}_n^{(1)} \right)\Big] \mathrm{d}s.
\end{align*}Therefore, 
\begin{align}\label{id_8_12}
\big\|\hspace{1pt}U_{2, 3}\hspace{1pt}\big\|_p \hspace{2pt}\lesssim_{M, B_0}\hspace{2pt} e^{-\tau},\hspace{15pt} \text{for all $\tau \geq 1$, $p \in [1, \infty]$ and $n \geq 3$.}
\end{align}
\noindent\textbf{Estimate of $U_{2,4}$.} 
Combining Proposition \ref{estimate of U1 up to second order} and Lemma \ref{temporal optima of U2}, we can find a constant $\beta_n > 0$ depending only on $n$ so that 
\begin{align}\label{id_8_01}
\left\| \hspace{1pt} U  - U^{(1)} \hspace{1pt}\right\|_p \hspace{2pt}\leq\hspace{2pt} C \hspace{0.5pt}e^{-\left(\frac{1}{2} + \beta_n\right) \tau}, \hspace{15pt}\text{for any $\tau \geq 0$ and $p \in [1, \infty]$.}
\end{align}
By (\ref{id_8_01}) and (\ref{estimate_grad_v}), there is a  constant $\alpha_n \in (0, 1)$ depending on $n$ so that 
\begin{align}\label{id_8_03}
&\left| \int_{\alpha_n \tau}^{\tau} \dfrac{e^{\frac{2-n}{2}s} e^{-\frac{\tau-s}{2}}}{ a\left(\tau - s\right)^{1+\frac{n}{2}}} \hspace{2pt} \mathrm{d} s\int_{\mathbb{R}^n} \Big[ U \nabla V  - U^{(1)}  \nabla V^{(1)} \Big]\left(z,s\right) \cdot \left(\xi -  e^{- \frac{\tau - s }{2}} z  \right) e^{ -\frac{\big| \xi - e^{-\frac{\tau - s}{2}} z \hspace{1pt} \big|^2}{4a\left(\tau-s\right)}} \mathrm{d} z \right| \notag\\[2mm]
&\hspace{60pt}\lesssim\hspace{2pt} e^{-\frac{\tau}{2}}\int_{\alpha_n \tau}^{\tau} \dfrac{e^{\frac{3-n}{2}s} e^{\frac{n}{2}(\tau-s)}}{ a\left(\tau-s\right)^{\frac{1}{2}}} \left\| U\left(z,s\right) \nabla V\left(z,s\right) - U^{(1)}(z,s) \nabla V^{(1)}(z,s)\right\|_\infty \mathrm{d} s \notag\\[2mm]
&\hspace{60pt}\leq\hspace{2pt} C\hspace{1pt} e^{-\tau} \int_{\alpha_n \tau}^\tau \frac{e^{\left(\frac{3-n}{2} -\beta_n\right)s} e^{\frac{n+1}{2}(\tau-s)}}{a(\tau-s)^{\frac{1}{2}}} \mathrm{d}s \notag\\[2mm]
&\hspace{60pt}\leq\hspace{2pt} C\hspace{1pt} e^{-\tau} e^{\left(\frac{3-n}{2} -\beta_n\right)\alpha_n\tau} e^{\frac{n+1}{2}(1-\alpha_n)\tau} \int_0^\tau \frac{\mathrm{d}s}{\sqrt{a(s)}} \hspace{3pt}\leq\hspace{2pt} C \hspace{1pt}e^{-\tau}, \hspace{15pt} \text{for all $\tau \geq 0$.}
\end{align}Since
\begin{align*}
\int_{\mathbb{R}^n} U^{(1)}\left(\cdot, s\right) \nabla V^{(1)}\left(\cdot, s\right) = 0 \quad \mbox{and} \quad \int_{\mathbb{R}^n} U\left(\cdot, s\right) \nabla V \left(\cdot, s \right) = 0, \hspace{15pt} \text{for all $s \geq 0$},
\end{align*}
it turns out
\begin{align*}
&\int_0^{\alpha_n \tau} \dfrac{e^{\frac{2-n}{2}s} e^{-\frac{\tau-s}{2}}}{ a\left(\tau - s\right)^{1+\frac{n}{2}}} \hspace{2pt} \mathrm{d} s\int_{\mathbb{R}^n} \Big[ U  \nabla V  - U^{(1)}  \nabla V^{(1)} \Big]\left(z,s\right) \cdot \left(\xi -  e^{- \frac{\tau - s }{2}} z  \right) e^{ -\frac{\big| \xi - e^{-\frac{\tau - s}{2}} z \hspace{1pt} \big|^2}{4a\left(\tau-s\right)}} \mathrm{d} z \\[2mm]
&\hspace{15pt}= \int_0^{\alpha_n \tau} \dfrac{e^{\frac{3-n}{2}s} e^{-\frac{\tau}{2}}}{ a\left(\tau - s\right)^{1+\frac{n}{2}}} \hspace{2pt} \mathrm{d} s\int_{\mathbb{R}^n} \Big[ U \nabla V - U^{(1)} \nabla V^{(1)}\Big]\left(z,s\right)  \cdot \int_0^{e^{- \frac{\tau - s}{2}}} \dfrac{\mathrm{d}}{\mathrm{d} \nu} \left[ \big( \xi - \nu z \big)e^{ - \frac{\left|\hspace{0.5pt} \xi - \nu z \hspace{0.5pt}\right|^2}{4\hspace{0.5pt}\left(1 - \nu^2\right)}}\right] \mathrm{d} \nu \hspace{2pt}\mathrm{d} z.
\end{align*}
Thus, utilizing (\ref{id_8_01}), (\ref{estimate_grad_v}), (\ref{estimate of weighted of U}) and (\ref{estimate of differentiation}), we obtain
\begin{align}\label{id_8_04}
&\left| \int_0^{\alpha_n \tau} \dfrac{e^{\frac{3-n}{2}s} e^{-\frac{\tau}{2}}}{ a\left(\tau - s\right)^{1+\frac{n}{2}}} \hspace{2pt} \mathrm{d} s\int_{\mathbb{R}^n} \Big[ U \nabla V  - U^{(1)}  \nabla V^{(1)} \Big]\left(z,s\right) \cdot \left(\xi -  e^{- \frac{\tau - s }{2}} z  \right) e^{ -\frac{\big| \xi - e^{-\frac{\tau - s}{2}} z \hspace{1pt} \big|^2}{4a\left(\tau-s\right)}} \mathrm{d} z  \right| \notag\\[2mm]
&\hspace{30pt}\lesssim\hspace{2pt}\dfrac{e^{- \tau}}{ a\big(\left(1 - \alpha_n\right)\tau\big)^{ 1 + \frac{n}{2}}} \int_0^{\alpha_n \tau} e^{\frac{4 - n}{2} s } \left[ \hspace{2pt}\left\| \nabla V - \nabla V^{(1)} \right\|_\infty(s)\hspace{1pt} \int_{\mathbb{R}^n} \big(1 + |\hspace{0.5pt}z\hspace{0.5pt}|\hspace{1pt}\big) \hspace{1pt}U(z, s) \hspace{2pt} \mathrm{d} z \hspace{1pt}\right] \hspace{2pt}\mathrm{d} s \notag\\[2mm]
&\hspace{30pt}+ \dfrac{e^{- \tau}}{ a\big(\left(1 - \alpha_n\right)\tau\big)^{ 1 + \frac{n}{2}}} \int_0^{\alpha_n \tau} e^{\frac{4 - n}{2} s }  \left[\hspace{2pt} \left\|   U - U^{(1)} \right\|_1(s)\hspace{1pt} \sup_{z \hspace{0.5pt}\in\hspace{0.5pt} \mathbb{R}^n}  \big(1 + |\hspace{0.5pt}z\hspace{0.5pt}|\hspace{1pt}\big) \hspace{1pt}\left| \nabla V^{(1)} \right|(z, s)  \hspace{1pt}\right] \hspace{2pt}\mathrm{d} s \notag\\[2mm]
&\hspace{30pt}\leq \frac{C \hspace{1pt} e^{-\tau}}{a\big( (1-\alpha_n)\tau \big)^{1+\frac{n}{2}}} \int_0^{\alpha_n \tau} e^{\left( \frac{3-n}{2} - \beta_n \right) s} \mathrm{d} s \hspace{2pt}\leq\hspace{2pt} C\hspace{1pt} e^{-\tau}, \hspace{15pt}\text{for all $\tau \geq 1$.}
\end{align}
The estimates (\ref{id_8_03}) and (\ref{id_8_04}) then yield
\begin{align}\label{id_8_05}
\big\|\hspace{1pt}U_{2, 4}\hspace{1pt}\big\|_\infty\hspace{2pt}\leq\hspace{2pt} C\hspace{1pt} e^{-\tau}, \hspace{15pt} \text{for all $\tau \geq 1$ and $n \geq 3$.}
\end{align}
We can also integrate the left-hand side of (\ref{id_8_03}) and (\ref{id_8_04}) with respect to the $\xi$ variable. By Fubini's theorem and similar estimates as in  (\ref{id_8_03})--(\ref{id_8_04}), it turns out
\begin{align}\label{id_8_06}
\big\|\hspace{1pt}U_{2, 4}\hspace{1pt}\big\|_1\hspace{2pt}\leq\hspace{2pt} C\hspace{1pt} e^{-\tau}, \hspace{15pt} \text{for all $\tau \geq 1$ and $n \geq 3$.}
\end{align}

In light of (\ref{id_8_07})--(\ref{id_8_12}) and (\ref{id_8_05})--(\ref{id_8_06}), the proof is completed.
\end{proof}


\section{\large Long-time stability: $2$D case}\label{sec_5}\vspace{0.5pc}

This section is devoted to studying the long-time behavior of global mild solutions to (\ref{eqn_0_0_2}). Our main result is
\begin{prop}\label{weight asymptotics of u L1}
Suppose that $u$ is a non-negative classical global mild solution to (\ref{eqn_0_0_2}). If there is a constant $K > 0$ so that \begin{align}\label{decay of u at t infinty 2d}[\hspace{0.5pt}u\hspace{0.5pt}]_\infty := \left(1 + t \right) \big\| \hspace{1pt}u\left(\cdot, t\right)\big\|_{\infty}\hspace{2pt}\leq \hspace{2pt}K, \hspace{10pt}\text{for all $t > 0$,}\end{align}then the total mass $M$ of $u$ satisfies $M \in [ \hspace{1pt}0,8\pi)$. Moreover, it holds 
\begin{align}\label{general p_L1_n=2}
\lim_{t \rightarrow \infty} t^{ 1 - \frac{1}{p}} \hspace{2pt}\left\| u(\cdot, t) - \dfrac{1}{t} G_M\left(\dfrac{\cdot}{\sqrt{t}}\right) \right\|_p = 0, \qquad \text{for any $1 \leq p \leq \infty$.}
\end{align}
The function $G_M$ in (\ref{general p_L1_n=2}) is a strictly positive function with the total mass $M$. Among all non-negative functions with the total mass $M$, it is the unique stationary solution to the following equation with  finite free energy:
\begin{align}\label{eqn_1_2d}
\partial_{\hspace{0.5pt}\tau}U  + \nabla_\xi \cdot \big( U \nabla_\xi E_2 * U \big)  = L_2 \hspace{0.5pt}U  \hspace{15pt}\text{on $\mathbb{R}_+^{2+1}$, where $L_2 := \Delta_\xi + \dfrac{1}{2} \hspace{1pt}\xi \cdot \nabla_\xi + 1$.}
\end{align}The free energy associated with (\ref{eqn_1_2d}) is given in (\ref{free energy in 2D}).
\end{prop} 
\noindent The proof of $2 \Rightarrow 3$ in Theorem \ref{thm_1_2D} follows easily from Proposition \ref{weight asymptotics of u L1}.

\subsection{\normalsize Strong $L^1$-compactness of  $\big\{ U\left(\cdot, \tau\right)\big\}$ as $\tau \rightarrow \infty$}\label{sec_5_1}\vspace{0.5pc}

Similarly as in the higher dimensional case, to prove Proposition \ref{weight asymptotics of u L1}, we need to study the strong $L^1$-compactness of the flow $\big\{ U\left(\cdot, \tau\right)\big\}$ as $\tau \rightarrow \infty$. 

\begin{lem}\label{lem_5_2}
Let $n = 2$ and suppose that $U$ is given in (\ref{def_U}) with $u \geq 0$ a classical global mild solution to (\ref{eqn_0_0_2}). If $U$ satisfies (\ref{unif_bound_U})--(\ref{mass_conservation}), then  for any $\epsilon > 0$, there is positive radius  $r_{\epsilon}$ depending on $\epsilon$, $[\hspace{0.5pt}u\hspace{0.5pt}]_\infty$ and $M$ such that 
\begin{align}\label{unif small outside a unif ball}
\sup_{\tau \hspace{1pt}> \hspace{1pt}0} \int_{B\mystrut^c_{r_{\epsilon}}} U\left( \xi , \tau \right) \mathrm{d}\xi \hspace{2pt}\leq\hspace{2pt} \epsilon.
\end{align}
As a consequence, for any sequence $\big\{\tau_k\big\}$ which diverges to $\infty$ as $k \rightarrow \infty$, there is a subsequence, still denoted by $\big\{\tau_k\big\}$, such that  $ \big\{ U\left(\cdot, \tau_k\right)\big\}$ converges strongly in $L^1$ as $k \rightarrow \infty$.
\end{lem} 

\begin{proof}[\bf Proof] We still use the decomposition $U = U_1 + U_2$ given in (\ref{id_7_01})--(\ref{def_U_2}). Here $n = 2$ and $f_2 \equiv 1$. Moreover, the function $U_1$ satisfies (\ref{id_7_04_1}) as well in the current proof. Letting $R > R_\epsilon$, where $R_\epsilon$ is given in (\ref{id_7_04_1}), we can define $\tau_0(R)$ as follows:
\begin{align}\label{def of tau 0}
\tau_0\left(R\right) := \sup \left\{ T > 0 : \int_{B\mystrut^c_R} U\left(\xi,\tau\right) \mathrm{d}\xi  \hspace{2pt}<\hspace{2pt} \epsilon, \mbox{ for all } \tau \in [\hspace{1pt}0,T\hspace{1pt}] \right\}.
\end{align}The $\tau_0(R)$ must be positive or $\infty$. In the next, we prove $\tau_0\left(R\right) = \infty$, provided that $\epsilon$ is small and $R$ is large. The smallness of $\epsilon$ depends on $M$ and $[\hspace{0.5pt}u\hspace{0.5pt}]_\infty$, while the largeness of $R$ depends on $\epsilon$, $M$ and $[\hspace{0.5pt}u\hspace{0.5pt}]_\infty$. \vspace{0.3pc}

To be simple, we use $\tau_0$ to denote $\tau_0(R)$. Suppose that $\tau_0 < \infty$. Then at $\tau_0$, it holds
\begin{align}\label{id_7_08}
\int_{B\mystrut^c_R} \big|\hspace{1pt}U_2\left(\xi,\tau_0\right) \big|\hspace{2pt} \mathrm{d}\xi \hspace{2pt}\lesssim\hspace{2pt}I, \end{align}where the term $I$ is defined as follows:
\begin{align*}I := \int_{B\mystrut^c_R} \mathrm{d} \xi \int_0^{\tau_0} \dfrac{e^{-\frac{\tau_0 - s}{2}}}{ a\left(\tau_0 - s\right)^2}  \hspace{1pt}\mathrm{d} s \int_{\mathbb{R}^2} U\left(\eta,s\right) \Big|\nabla_\eta E_2 * U \hspace{1pt}\Big| \left(\eta,s\right)  \left|\hspace{1pt}\xi -  e^{- \frac{\tau_0 - s }{2}} \eta \hspace{1pt}\right| e^{ -\frac{\big| \xi - e^{-\frac{\tau_0 - s}{2}} \eta \hspace{1pt} \big|^2}{4a\left(\tau_0-s\right)}} \mathrm{d}\eta. \end{align*}The main part of the proof is to estimate $I$. \vspace{0.3pc}\\ [2mm]
\textbf{Step 1.} For any $\epsilon > 0$,  there exists a unique $t_\epsilon > 0$ so that  \begin{align}\label{def of t epsilon}M^{\frac{3}{2}} [ \hspace{1pt}u \hspace{1pt}]_\infty^{\frac{1}{2}} \int_0^{\log \left(1 + t_\epsilon\right)} \dfrac{\mathrm{d}s}{\sqrt{a(s)}} \hspace{2pt}=\hspace{2pt} \epsilon^2.
\end{align}The function $a(s)$ is defined in (\ref{semi group of U variable}). By Fubini's theorem, (\ref{estimate_grad_v}) and (\ref{unif_bound_U})--(\ref{mass_conservation}), it turns out
\begin{align}\label{less than log 1 + tepsilon}
I &\hspace{2pt}\leq\hspace{2pt} \int_{\mathbb{R}^2} \mathrm{d} \xi \int_{0}^{\tau_0} \mathrm{d} s \int_{\mathbb{R}^2} \dfrac{U\left(\eta,s\right) }{ a\left(\tau_0 - s\right)^2} \Big|\nabla_\eta E_2 * U \hspace{1pt}\Big|\left(\eta,s\right)  \left|\hspace{1pt}\xi -  e^{-\frac{\tau_0 - s }{2}} \eta \hspace{1pt}\right| e^{ -\frac{\big| \xi - e^{-\frac{\tau_0 - s}{2}} \eta \hspace{1pt} \big|^2}{4a\left(\tau_0-s\right)}} \mathrm{d} \eta  \notag\\[2mm]
&\hspace{2pt}\lesssim \int_{0}^{\tau_0} \mathrm{d} s \int_{\mathbb{R}^2} \dfrac{U\left(\eta,s\right)}{ a\left(\tau_0 - s\right)^{\frac{1}{2}}}  \Big|\nabla_\eta E_2 * U \hspace{1pt}\Big|\left(\eta,s\right)  \mathrm{d} \eta \notag\\[2mm]
&\hspace{2pt}\lesssim\hspace{2pt} M^{\frac{3}{2}} [ \hspace{1pt}u \hspace{1pt}]_\infty^{\frac{1}{2}} \int_{0}^{\tau_0} \dfrac{\mathrm{d}s}{ a\left(\tau_0 - s\right)^{\frac{1}{2}}}   \hspace{2pt}\leq\hspace{2pt} M^{\frac{3}{2}} [ \hspace{1pt}u \hspace{1pt}]_\infty^{\frac{1}{2}} \int_0^{\log \left(1 + t_\epsilon\right)} \dfrac{\mathrm{d}s}{\sqrt{a(s)}} \hspace{2pt}=\hspace{2pt} \epsilon^2, \hspace{20pt}\text{if $\tau_0 \leq \log\left(1 + t_\epsilon\right)$.}
\end{align} 

Now we assume $\tau_0 > \log \left( 1 + t_\epsilon\right)$ and split the integral $I$ into \begin{align}\label{decomp of I}I = I_1 + I_2.
\end{align}Here, $I_1$ represents the following integral:
\begin{align*}
 \int_{B\mystrut^c_R} \mathrm{d} \xi \int_{\tau_0 - \log\left(1+t_\epsilon\right)}^{\tau_0} \dfrac{e^{-\frac{\tau_0 - s}{2}}}{ a\left(\tau_0 - s\right)^2}  \hspace{1pt}\mathrm{d} s \int_{\mathbb{R}^2} U\left(\eta,s\right) \Big|\nabla_\eta E_2 * U\hspace{1pt}\Big| \left(\eta,s\right)  \left|\hspace{1pt}\xi -  e^{- \frac{\tau_0 - s }{2}} \eta \hspace{1pt}\right| e^{ -\frac{\big| \xi - e^{-\frac{\tau_0 - s}{2}} \eta \hspace{1pt} \big|^2}{4a\left(\tau_0-s\right)}} \mathrm{d} \eta.
\end{align*}
By using the same derivations for (\ref{less than log 1 + tepsilon}), $I_1$ can be estimated by \begin{align}\label{est of I1}
I_1 &\hspace{2pt}\leq\hspace{2pt} \int_{\mathbb{R}^2} \mathrm{d} \xi \int_{\tau_0 - \log\left(1+t_\epsilon\right)}^{\tau_0} \mathrm{d} s\int_{\mathbb{R}^2} \dfrac{U\left(\eta,s\right) }{ a\left(\tau_0 - s\right)^2} \hspace{1pt}\Big|\nabla_\eta E_2 * U \hspace{1pt}\Big|\left(\eta,s\right)  \left|\hspace{1pt}\xi -  e^{-\frac{\tau_0 - s }{2}} \eta \hspace{1pt}\right| e^{ -\frac{\big| \xi - e^{-\frac{\tau_0 - s}{2}} \eta \hspace{1pt} \big|^2}{4a\left(\tau_0-s\right)}} \mathrm{d} \eta \notag\\[2mm]
&\hspace{2pt}\lesssim\hspace{2pt} \int_{\tau_0 - \log\left(1+t_\epsilon\right)}^{\tau_0} \mathrm{d} s \int_{\mathbb{R}^2} \dfrac{U\left(\eta,s\right)}{ a\left(\tau_0 - s\right)^{\frac{1}{2}}} \hspace{1pt}  \Big|\nabla_\eta E_2 * U \hspace{1pt}\Big| \left(\eta,s\right)  \hspace{2pt}\mathrm{d} \eta \notag\\[2mm]
&\hspace{2pt}\lesssim\hspace{2pt} M^{\frac{3}{2}} [ \hspace{1pt}u \hspace{1pt}]_\infty^{\frac{1}{2}} \int_{\tau_0 - \log\left(1+t_\epsilon\right)}^{\tau_0} \dfrac{\mathrm{d} s}{ a\left(\tau_0 - s\right)^{\frac{1}{2}}} \hspace{2pt}=\hspace{2pt} M^{\frac{3}{2}} [ \hspace{1pt}u \hspace{1pt}]_\infty^{\frac{1}{2}} \int_0^{\log \left(1 + t_\epsilon\right)}\dfrac{\mathrm{d}s}{\sqrt{a(s)}} \hspace{2pt}=\hspace{2pt} \epsilon^2.
\end{align}The last equality in (\ref{est of I1}) holds by (\ref{def of t epsilon}). \vspace{0.3pc}\\[2mm]
\textbf{Step 2.} As for the estimate of $I_2$ in (\ref{decomp of I}), we introduce another positive radius \begin{align}\label{def of R 1 epsilon}
R_{1, \epsilon} := M \hspace{1pt}[ \hspace{1pt}u \hspace{1pt}]_\infty^{-\frac{1}{2}} \hspace{1pt}\epsilon^{-\frac{1}{2}}.
\end{align}With $R_{1, \epsilon}$, we split the integral \begin{align*}
I_2 = \int_{B\mystrut^c_R} \mathrm{d} \xi \int^{\tau_0 - \log\left(1+t_\epsilon\right)}_0 \dfrac{e^{-\frac{\tau_0 - s}{2}}}{ a\left(\tau_0 - s\right)^2}  \hspace{1pt}\mathrm{d} s \int_{\mathbb{R}^2} U\left(\eta,s\right) \Big|\nabla_\eta E_2 * U\hspace{1pt}\Big| \left(\eta,s\right)  \left|\hspace{1pt}\xi -  e^{- \frac{\tau_0 - s }{2}} \eta \hspace{1pt}\right| e^{ -\frac{\big| \xi - e^{-\frac{\tau_0 - s}{2}} \eta \hspace{1pt} \big|^2}{4a\left(\tau_0-s\right)}} \mathrm{d} \eta
\end{align*}into \begin{align}\label{decomp of I2}I_2 = I_{2, 1} + I_{2, 2},
\end{align}where $I_{2, 1}$ represents the integral given as follows:
\begin{align*}
\int_{B\mystrut^c_R} \mathrm{d} \xi \int_0^{\tau_0 - \log\left(1+t_\epsilon\right)} \dfrac{e^{-\frac{\tau_0 - s}{2}}}{ a\left(\tau_0 - s\right)^2} \hspace{1pt}\mathrm{d} s\int_{B\mystrut^{c}_{R + R_{1, \epsilon}}}  U\left(\eta,s\right) \Big|\nabla_\eta E_2 * U\hspace{1pt}\Big|\left(\eta,s\right)  \left|\hspace{1pt}\xi -  e^{- \frac{\tau_0 - s }{2}} \eta \hspace{1pt}\right| e^{ -\frac{\big| \xi - e^{-\frac{\tau_0 - s}{2}} \eta \hspace{1pt} \big|^2}{4a\left(\tau_0-s\right)}} \mathrm{d} \eta.
\end{align*} By H\"{o}lder's inequality and (\ref{unif_bound_U})--(\ref{mass_conservation}),
\begin{align}\label{id_7_10}
\Big| \nabla_\eta E_2 * U\hspace{1pt}\Big|&\left(\eta,\tau\right) \hspace{2pt}\lesssim\hspace{2pt} \int_{B_{R_{1, \epsilon}}(\eta)} \dfrac{U\left( \xi, \tau \right)}{\big|\hspace{1pt} \eta - \xi \hspace{1pt}\big|} \hspace{2pt}  \mathrm{d}\xi + \int_{B\mystrut^c_{R_{1, \epsilon}}(\eta)} \dfrac{U\left( \xi, \tau \right)}{\big|\hspace{1pt} \eta - \xi \hspace{1pt}\big|} \hspace{2pt}  \mathrm{d}\xi \notag\\[2mm]
&\hspace{2pt}\lesssim\hspace{2pt} \big\|\hspace{1pt} U\left( \cdot, \tau \right) \big\|_{3; B_{R_{1, \epsilon}}(\eta)} \hspace{1pt} \Big\|\hspace{1pt} \big| \eta - \,\cdot\, \big|^{-1} \Big\|_{\frac{3}{2};B_{R_{1, \epsilon}}(\eta)}  +  \big\|\hspace{1pt} U\left( \cdot, \tau \right) \big\|_{\frac{3}{2};B\mystrut^c_{R_{1, \epsilon}}(\eta)} \hspace{1pt} \Big\|\hspace{1pt} \big| \eta - \,\cdot\, \big|^{-1} \Big\|_{3;B\mystrut^{c}_{R_{1, \epsilon}}(\eta)} \notag\\[2mm] 
&\hspace{2pt}\lesssim\hspace{2pt} R_{1, \epsilon}^\frac{1}{3} \hspace{1pt} \big\|\hspace{1pt} U\left( \cdot , \tau \right) \big\|_\infty^\frac{2}{3} \big\|\hspace{1pt} U\left( \cdot , \tau \right) \big\|_{1; B_{R_{1, \epsilon}}(\eta)}^\frac{1}{3} + R_{1, \epsilon}^{-\frac{1}{3}} \big\|\hspace{1pt} U\left( \cdot , \tau \right) \big\|_\infty^\frac{1}{3} \hspace{1pt}\big\|\hspace{1pt} U\left( \cdot , \tau \right) \big\|_1^\frac{2}{3} \notag\\[2mm]
&\hspace{2pt}\lesssim\hspace{2pt} R_{1, \epsilon}^\frac{1}{3} \hspace{1pt} [ \hspace{1pt}u \hspace{1pt}]_\infty^\frac{2}{3} \hspace{1pt} \big\|\hspace{1pt} U\left( \cdot , \tau \right) \big\|_{1; B_{R_{1, \epsilon}}(\eta)}^\frac{1}{3} + R_{1, \epsilon}^{-\frac{1}{3}} \hspace{1pt}M^\frac{2}{3} \hspace{1pt}[ \hspace{1pt}u \hspace{1pt}]_\infty^\frac{1}{3}. 
\end{align}Using triangle inequality, we have $B_{R_{1, \epsilon}}(\eta) \subset B\mynewstrut^c_R$, for all $\eta \in B\mynewstrut^c_{R + R_{1, \epsilon}}$. In light of the definition of $\tau_0$ in (\ref{def of tau 0}), it then follows
\begin{align*}
\big\|\hspace{1pt} U\left( \cdot , \tau \right) \big\|_{1; \hspace{0.5pt}B_{R_{1, \epsilon}}(\eta)} \hspace{2pt}\leq\hspace{2pt} \int_{B\mystrut^c_R} U\left(\xi,\tau\right) \mathrm{d}\xi \hspace{2pt}\leq\hspace{2pt} \epsilon, \hspace{15pt}\text{for any $\eta \in B\mynewstrut^c_{R + R_{1, \epsilon}}$ and $\tau \leq \tau_0$.}
\end{align*}Applying the last estimate to (\ref{id_7_10}) and recalling the definition of $R_{1, \epsilon}$ in (\ref{def of R 1 epsilon}), we obtain
\begin{align}\label{estimate of nabla V outside R + R 1 epsilon}
\Big| \nabla_\eta E_2 * U\hspace{1pt}\Big|\left(\eta,\tau\right)  \hspace{2pt}\lesssim\hspace{2pt} M^\frac{1}{3} \hspace{1pt} [ \hspace{1pt}u \hspace{1pt}]_\infty^{\frac{1}{2}} \hspace{1.5pt}\epsilon^{\frac{1}{6}},  \hspace{15pt}\text{for any $\eta \in B\mynewstrut^c_{R + R_{1, \epsilon}}$ and $\tau \leq \tau_0$.}
\end{align}
By Fubini's theorem and (\ref{estimate of nabla V outside R + R 1 epsilon}), the $I_{2, 1}$ given at the beginning of this step can be estimated as follows:
\begin{align}\label{id_7_11}
I_{2,1} &\hspace{2pt}\lesssim\hspace{2pt} \int_0^{\tau_0 - \log\left(1+t_\epsilon\right)} \dfrac{e^{-\frac{\tau_0 - s}{2}}}{ a\left(\tau_0 - s\right)^\frac{1}{2}}  \hspace{1pt}\mathrm{d} s\int_{B\mystrut^c_{R+R_{1, \epsilon}}} U\left(\eta,s\right) \Big|\nabla_\eta E_2 * U \Big|\left(\eta,s\right)  \hspace{2pt} \mathrm{d} \eta \notag\\[2mm]
&\hspace{2pt}\lesssim\hspace{2pt} M^\frac{1}{3} \hspace{1pt} [ \hspace{1pt}u \hspace{1pt}]_\infty^{\frac{1}{2}} \hspace{1pt} \epsilon^{\frac{1}{6}} \int_0^{\tau_0 - \log\left(1+t_\epsilon\right)}  \dfrac{e^{-\frac{\tau_0 - s}{2}}}{ a\left(\tau_0 - s\right)^\frac{1}{2}} \mathrm{d}s \int_{B\mystrut^{c}_{R+R_{1, \epsilon}}} U\left(\eta,s\right)  \hspace{2pt} \mathrm{d} \eta \notag\\[2mm]
&\hspace{2pt}\lesssim\hspace{2pt} M^\frac{1}{3} \hspace{1pt} [ \hspace{1pt}u \hspace{1pt}]_\infty^{\frac{1}{2}} \hspace{1pt} \epsilon^{\frac{7}{6}}.
\end{align}
In the last estimate of (\ref{id_7_11}), we also have used (\ref{def of tau 0}) so that \begin{align*}
\int_{B\mystrut^c_{R+R_{1, \epsilon}}} U\left(\eta,s\right) \mathrm{d} \eta \hspace{2pt}\leq\hspace{2pt} \int_{B\mystrut^c_R} U\left(\eta,s\right) \mathrm{d} \eta \hspace{2pt}\leq\hspace{2pt} \epsilon \qquad \mbox{for all } s \leq \tau_0.
\end{align*}
\textbf{Step 3.} Given $t_\epsilon$ and $R_{1, \epsilon}$ (see (\ref{def of t epsilon}) and (\ref{def of R 1 epsilon})), we can take $R_{2, \epsilon} > R_\epsilon$ sufficiently large such that
\begin{align}\label{R2 satisfied estimate}
\sigma_\epsilon := 1 - \big(1+t_\epsilon\big)^{-\frac{1}{2}} \left( 1 + \dfrac{R_{1, \epsilon}}{R_{2, \epsilon}}\right)\hspace{2pt}> \hspace{2pt}0 \quad \mbox{and} \quad \sigma_\epsilon R_{2, \epsilon} \hspace{2pt}\geq\hspace{2pt} \sqrt{2}.
\end{align}
Hence for any $\xi \in B\mynewstrut^c_R$, $\eta \in B_{R + R_{1, \epsilon}}$ and $s \leq \tau_0 - \log\left( 1 + t_\epsilon \right)$, it satisfies
\begin{align}\label{id_7_12}
\left|\hspace{0.5pt}\xi -  e^{- \frac{\tau_0 - s }{2}} \eta \hspace{0.5pt}\right| &\hspace{2pt}\geq\hspace{2pt}|\hspace{0.5pt}\xi\hspace{0.5pt}| - \big(1 + t_\epsilon\big)^{-\frac{1}{2}} | \hspace{0.5pt}\eta\hspace{0.5pt}| \hspace{2pt}\geq\hspace{2pt} |\hspace{0.5pt}\xi\hspace{0.5pt}| - \big(1 + t_\epsilon\big)^{-\frac{1}{2}} \big( R + R_{1, \epsilon}\big) \notag\\[2mm]
&\hspace{2pt}=\hspace{2pt}|\hspace{0.5pt}\xi\hspace{0.5pt}| - R\hspace{1pt}\big(1 + t_\epsilon\big)^{-\frac{1}{2}} \left( 1 + \dfrac{R_{1, \epsilon}}{R}\right) \notag\\[2mm]
&\hspace{2pt}\geq\hspace{2pt} |\hspace{0.5pt}\xi\hspace{0.5pt}| \left( 1 - \big(1+t_\epsilon\big)^{-\frac{1}{2}} \left( 1 + \dfrac{R_{1, \epsilon}}{R}\right)\right) \hspace{2pt}\geq\hspace{2pt} \sigma_\epsilon\hspace{0.5pt}|\hspace{0.5pt}\xi\hspace{0.5pt}| \hspace{2pt}\geq\hspace{2pt} \sigma_\epsilon\hspace{0.5pt}R_{2, \epsilon} \geq \sqrt{2}, \hspace{15pt}\text{for all $R \geq R_{2, \epsilon}$.}
\end{align}
Since the function $r e^{-\frac{r^2}{4}}$ is monotonically decreasing on $\left[ \sqrt{2} , \infty \right)$, (\ref{id_7_12}) then yields \begin{align*}&\left|\hspace{1pt}\xi -  e^{- \frac{\tau_0 - s }{2}} \eta \hspace{1pt}\right| e^{ -\frac{\big| \xi - e^{-\frac{\tau_0 - s}{2}} \eta \hspace{1pt} \big|^2}{4a\left(\tau_0-s\right)}} \hspace{2pt}\leq\hspace{2pt}\sigma_\epsilon\hspace{1pt}|\hspace{0.5pt}\xi\hspace{0.5pt}| \hspace{1pt} e^{- \frac{\sigma_\epsilon^2 \hspace{1pt}|\hspace{0.5pt}\xi\hspace{0.5pt}|^2}{4a\left(\tau_0 - s \right)}}, \\[2mm]
&\hspace{150pt}\text{for any $R \geq R_{2, \epsilon}$, $\xi \in B\mynewstrut^c_R$, $\eta \in B_{R + R_{1, \epsilon}}$ and $s \leq \tau_0 - \log\left( 1 + t_\epsilon \right)$.}
\end{align*}
By the last estimate, (\ref{unif_bound_U})--(\ref{mass_conservation}), (\ref{estimate_grad_v}) and Fubini's theorem,  for all $R \geq R_{2, \epsilon}$, the term $I_{2,2}$ in (\ref{decomp of I2}) can be estimated as follows:
\begin{align}\label{id_7_13}
I_{2,2} &\hspace{2pt}\lesssim\hspace{2pt}  \sigma_\epsilon\int_0^{\tau_0 - \log\left(1+t_\epsilon\right)} \dfrac{e^{-\frac{\tau_0 - s}{2}}}{ a\left(\tau_0 - s\right)^2}  \hspace{1pt}\mathrm{d} s \int_{B_{R+R_{1, \epsilon}}} U\left(\eta,s\right) \Big|\nabla_\eta E_2 * U \hspace{1pt}\Big|\left(\eta,s\right)  \hspace{2pt}\mathrm{d} \eta\int_{B\mystrut^c_R}  |\hspace{0.5pt}\xi \hspace{0.5pt}| \hspace{1pt} e^{-\frac{\sigma_\epsilon^2\hspace{0.5pt}|\hspace{0.5pt}\xi\hspace{0.5pt}|^2}{4a\left(\tau_0-s\right)}} \hspace{2pt}\mathrm{d}\xi \notag\\[2mm]
&\hspace{2pt}=\hspace{2pt} \sigma_\epsilon \int_0^{\tau_0 - \log\left(1+t_\epsilon\right)} \dfrac{e^{-\frac{\tau_0 - s}{2}}}{ a\left(\tau_0 - s\right)^{\frac{1}{2}}}  \hspace{1pt}\mathrm{d} s\int_{B_{R+R_{1, \epsilon}}} U\left(\eta,s\right) \Big|\nabla_\eta E_2 * U\hspace{1pt}\Big|\left(\eta,s\right) \hspace{2pt}\mathrm{d} \eta \int_{B\mystrut^c_{R a(\tau_0-s)^{ - \frac{1}{2} }}}  |\hspace{0.5pt}\xi\hspace{0.5pt}| \hspace{1pt} e^{-\frac{\sigma_\epsilon^2\hspace{0.5pt}|\hspace{0.5pt}\xi\hspace{0.5pt}|^2}{4}} \mathrm{d}\xi \notag\\[2mm]
&\hspace{2pt}\lesssim\hspace{2pt} \sigma_\epsilon \hspace{1pt} M^\frac{3}{2} [ \hspace{1pt}u \hspace{1pt}]_\infty^{\frac{1}{2}} \int_{B\mystrut^c_R} |\hspace{0.5pt}\xi\hspace{0.5pt}| \hspace{1pt} e^{-\frac{\sigma_\epsilon^2|\hspace{0.5pt}\xi\hspace{0.5pt}|^2}{4}} \mathrm{d}\xi \int_0^{\tau_0 - \log\left(1+t_\epsilon\right)} \dfrac{e^{-\frac{\tau_0 - s}{2}}}{ a\left(\tau_0 - s\right)^{\frac{1}{2}}} \hspace{2pt} \mathrm{d}s \notag\\[2mm]
&\hspace{2pt}\lesssim\hspace{2pt} \sigma_\epsilon\hspace{1pt} M^\frac{3}{2} [ \hspace{1pt}u \hspace{1pt}]_\infty^{\frac{1}{2}} \int_{B\mystrut^c_R} |\hspace{0.5pt}\xi\hspace{0.5pt}|\hspace{1pt} e^{-\frac{\sigma_\epsilon^2|\hspace{0.5pt}\xi\hspace{0.5pt}|^2}{4}} \mathrm{d}\xi \longrightarrow 0 \qquad\quad\qquad \mbox{as } R \rightarrow \infty.
\end{align}
\textbf{Step 4.} In light of (\ref{less than log 1 + tepsilon}), it follows $I \lesssim \epsilon^2$, provided that $\tau_0 \leq \log\left(1 + t_\epsilon\right)$. If  $\tau_0 > \log\left(1 + t_\epsilon\right)$, then by (\ref{est of I1}), (\ref{id_7_11}), (\ref{id_7_13}), (\ref{decomp of I}), (\ref{decomp of I2}) and (\ref{id_7_08}), there are $\epsilon$ sufficiently small and $r_\epsilon > R_\epsilon$ sufficiently large so that \begin{align*} \int_{B\mystrut^c_{r_\epsilon}} \big|\hspace{1pt}U_2\left(\xi, \tau_0\right)\big| \hspace{2pt}\leq\hspace{2pt}\dfrac{\epsilon}{4}.
\end{align*}
Here the smallness of $\epsilon$ depends on $M$ and $[\hspace{0.5pt}u\hspace{0.5pt}]_\infty$, while the largeness of $r_\epsilon$ depends on $M$, $[\hspace{0.5pt}u\hspace{0.5pt}]_\infty$ and $\epsilon$. Recalling  (\ref{id_7_04_1}), the decomposition (\ref{id_7_01}) and the last estimate for $U_2$, in any case, we have 
\begin{align*}
\int_{B\mystrut^c_{r_\epsilon}} U\left(\xi,\tau_0\right) \mathrm{d}\xi \hspace{2pt}<\hspace{2pt} \dfrac{3}{4} \epsilon.
\end{align*}However, the above estimate leads a contradiction to the definition of $\tau_0 = \tau_0(r_\epsilon)$. Therefore $\tau_0(r_\epsilon) = \infty$. (\ref{unif small outside a unif ball}) then follows. \\[2mm]
\textbf{Step 5.} Similarly as in the higher dimensional case, for any sequence $\big\{ \tau_k \big\}$ with $\tau_k \rightarrow \infty$ as $k \rightarrow \infty$, we can extract a subsequence, still denoted by $\big\{\tau_k\big\}$ so that (\ref{loc unif conv of U}) holds for $n = 2$. The strong $L^1$-compactness result in this lemma then follows by (\ref{loc unif conv of U}) and (\ref{unif small outside a unif ball}).  \end{proof}

\subsection{\normalsize Limiting flow}\label{sec_5_2}\vspace{0.5pc}

Let $\big\{ \tau_k \big\}$ be a fixed sequence which diverges to $\infty$ as $k \rightarrow \infty$. By the same derivations for (\ref{Holder norm est of U}), it holds \begin{align}\label{holder est for U at large tau} \big\|\hspace{1pt}U\hspace{1pt}\big\|_{C^{0, \frac{1}{4}}\left(P_{\sqrt{2} R} \big(0, \tau_k  + R^2\big)\right)} \hspace{2pt}\leq \hspace{2pt} C, \hspace{10pt}\text{provided that $k$ is large.}
\end{align}Here the constant $C$ depends on $R$, $M$ and $[\hspace{0.5pt}u\hspace{0.5pt}]_\infty$. If we define \begin{align}\label{def of Uk} U_k\left(\xi, \tau\right) := U\left(\xi, \tau_k + \tau\right),
\end{align}then the H\"{o}lder-norm boundedness in (\ref{holder est for U at large tau}) induces \begin{align} \big\|\hspace{1pt}U_k\hspace{1pt}\big\|_{C^{0, \frac{1}{4}}\left(Q_R\right)} \hspace{2pt}\leq \hspace{2pt} C, \hspace{10pt}\text{provided that $k$ is large.}
\end{align}By Arzel\` {a}-Ascoli theorem and a diagonal argument, we can extract a subsequence, still denoted by $\big\{\tau_k\big\}$, so that for some non-negative function $U_\star$ on $\mathbb{R}^{2 + 1}$, it holds \begin{align}\label{conv in spacetime of Uk} U_k \longrightarrow U_{\star} \hspace{15pt}\text{in $L^\infty_{\mathrm{loc}}\left(\mathbb{R}^{2 + 1}\right)$, as $k \rightarrow \infty$.}
\end{align}In light of the strong $L^1$-compactness in Lemma \ref{lem_5_2} and the uniform boundedness of $U$ in (\ref{unif_bound_U}), the convergence in (\ref{conv in spacetime of Uk}) infers \begin{align}\label{Lp conv of Uk} U_k\left(\cdot, \tau\right) \longrightarrow U_\star\left(\cdot, \tau\right) \hspace{20pt}\text{strongly in $L^p$, for all $\tau \in \mathbb{R}$ and $p \in [\hspace{1pt}1, \infty )$.} 
\end{align}Moreover, the limiting flow $U_\star$ satisfies
\begin{align}\label{id_7_34}
\big\| U_\star (\cdot,\tau) \big\|_1 = M \quad \mbox{and} \quad \big\| \hspace{1pt} U_\star (\cdot, \tau) \hspace{1pt} \big\|_{\infty} \leq [ \hspace{1pt}u \hspace{1pt}]_\infty, \quad \mbox{for all}\hspace{4pt} \tau \in \mathbb{R}.
\end{align}
As for the convergence of $\nabla_\xi E_2 * U_k$, by the standard potential estimate and (\ref{Lp conv of Uk}), we have \begin{align} \label{conv of Vk}
\nabla_\xi E_2 * U_k    \longrightarrow \nabla_\xi E_2 * U_\star, \hspace{15pt} \mbox{pointwisely on $\mathbb{R}^{2+1}$ as $k \rightarrow \infty$.}
\end{align} With the previous arguments, the following lemma holds for $U_\star$:
\begin{lem}\label{characterization of limiting flow}The limiting flow $U_\star$ satisfies the integral equation: \begin{align}\label{id_7_33}
U_\star\left(\xi,\tau\right) &\hspace{2pt}=\hspace{2pt} \dfrac{M}{4\pi} e^{-\frac{|\xi|^2}{4}} + \frac{1}{4\pi} \int_{-\infty}^{\tau}  \dfrac{\mathrm{d} s}{ a\left(\tau - s\right)}  \int_{\mathbb{R}^2}  U_\star \left(\eta, s\right) \nabla_\eta E_2 * U_\star \hspace{2pt}\bigg|_{\left(\eta,s\right)}  \cdot \nabla_\eta \hspace{1pt} e^{ -\frac{\big| \xi - e^{-\frac{\tau - s}{2}} \eta \hspace{1pt} \big|^2}{4a\left(\tau-s\right)}} \mathrm{d} \eta.
\end{align}\end{lem}
\begin{proof}[\bf Proof] Using the decomposition of $U$ in (\ref{id_7_01}), for any $\tau \in \mathbb{R}$, it holds 
\begin{align}\label{decomp of Uk}
U_k\left(\xi, \tau\right) = U_1\left(\xi,\tau_k+\tau\right) + U_2\left(\xi, \tau_k+\tau\right).
\end{align}By Lebesgue's dominated convergence theorem,
\begin{align}\label{id_7_30}
U_1\left(\xi,\tau_k+\tau\right) \longrightarrow \dfrac{1}{4\pi} \int_{\mathbb{R}^2} U_0\left(\eta\right)e^{-\frac{|\xi|^2}{4}} \mathrm{d}\eta \hspace{2pt}=\hspace{2pt} \dfrac{M}{4\pi} \hspace{1pt} e^{-\frac{|\xi|^2}{4}} \qquad \mbox{as } k \rightarrow \infty.
\end{align}Now we fix $\big(\tau, l \big)  \in \mathbb{R} \times \mathbb{R}_+$ arbitrarily and decompose $U_2\left(\xi, \tau_k + \tau\right)$ into \begin{align}\label{decomp of U2 tauk} U_2\left(\xi, \tau_k + \tau\right) = \dfrac{1}{4\pi} \Big(\overline{U}_{2, 1}\left(\xi, \tau\right) + \overline{U}_{2, 2}\left(\xi,  \tau\right)\Big),
\end{align}where \begin{align*} \overline{U}_{2, 1} :=  \int_0^{\tau_k+\tau-l} \dfrac{\mathrm{d} s}{ a\left(\tau_k+\tau - s\right)}  \int_{\mathbb{R}^2}  U\left(\eta,s\right) \nabla_\eta E_2 * U\hspace{1pt}\Big|_{\left(\eta,s\right)} \cdot \nabla_\eta\hspace{1pt} e^{ -\frac{\big| \xi - e^{-\frac{\tau_k+\tau - s}{2}} \eta \hspace{1pt} \big|^2}{4a\left(\tau_k+\tau-s\right)}} \mathrm{d} \eta.
\end{align*}Utilizing (\ref{estimate_grad_v}) and (\ref{unif_bound_U})--(\ref{mass_conservation}), we can take $k$ sufficiently large and estimate $\overline{U}_{2,1}$ pointwisely as follows: \begin{align}\label{id_7_31}
\big|\hspace{1pt}\overline{U}_{2, 1}\left(\xi, \tau\right)\big| &\hspace{2pt}\lesssim\hspace{2pt} \int_0^{\tau_k+\tau-l} \dfrac{e^{-\frac{\tau_k+\tau - s}{2}}}{ a\left(\tau_k+\tau - s\right)^\frac{3}{2}} \hspace{2pt} \mathrm{d} s \int_{\mathbb{R}^2}  U\left(\eta,s\right) \Big| \nabla_\eta E_2 * U \hspace{1pt}\Big| \left(\eta,s\right)  \hspace{2pt} \mathrm{d} \eta \notag\\[2mm]
&\hspace{2pt}\lesssim_{\hspace{1pt}M, \hspace{1pt}[\hspace{1pt}u\hspace{1pt}]_\infty}\hspace{2pt} \int_0^{\tau_k+\tau-l} \dfrac{e^{-\frac{\tau_k+\tau - s}{2}}}{ a\left(\tau_k+\tau - s\right)^\frac{3}{2}} \mathrm{d}s \notag\\[1mm]
&\hspace{2pt}=\hspace{2pt} 2 \int_{e^{-\frac{\tau_k+\tau}{2}}}^{e^{-\frac{l}{2}}} \frac{\mathrm{d} t}{\left(1 - t^2\right)^{\frac{3}{2}}} \hspace{2pt}\lesssim\hspace{2pt}  \int_0^{e^{-\frac{l}{2}}} \frac{\mathrm{d} t}{\left(1 - t^2\right)^{\frac{3}{2}}}.
\end{align}As for $\overline{U}_{2, 2}$, it satisfies \begin{align*} \overline{U}_{2, 2} =  \int_{\tau-l}^{\tau} \dfrac{\mathrm{d} s}{ a\left(\tau - s\right)}  \int_{\mathbb{R}^2}  U_k\left(\eta,s\right) \nabla_\eta E_2 * U_k\hspace{1pt}\bigg|_{\left(\eta,s\right)} \cdot \nabla_\eta\hspace{1pt} e^{ -\frac{\big| \xi - e^{-\frac{\tau - s}{2}} \eta \hspace{1pt} \big|^2}{4a\left(\tau-s\right)}} \mathrm{d} \eta.
\end{align*}Fubini's theorem then induces \begin{align*} &\int_{\mathbb{R}^2} \left| \hspace{1pt}\overline{U}_{2, 2} -  \int_{\tau-l}^{\tau} \dfrac{\mathrm{d} s}{ a\left(\tau - s\right)}  \int_{\mathbb{R}^2}  U_\star\left(\eta,s\right) \nabla_\eta E_2 * U_\star\hspace{1pt}\bigg|_{ \left(\eta,s\right)} \cdot \nabla_\eta\hspace{1pt} e^{ -\frac{\big| \xi - e^{-\frac{\tau - s}{2}} \eta \hspace{1pt} \big|^2}{4a\left(\tau-s\right)}} \mathrm{d} \eta \hspace{1pt} \right| \hspace{2pt}\mathrm{d} \xi \notag\\[2mm]
&\hspace{15pt}\lesssim\hspace{2pt}\int_{\tau - l}^{\tau} \dfrac{e^{-\frac{ \tau - s}{2}}}{ a\left( \tau - s\right)^\frac{1}{2}} \hspace{2pt} \mathrm{d} s \int_{\mathbb{R}^2} \Big|\hspace{1pt}U_k\left(\cdot, s\right) - U_\star\left(\cdot, s \right) \Big| \hspace{2pt} \Big|\hspace{1pt}\nabla E_2 * U_k\hspace{1pt}\Big| \left(\cdot,s\right)  \notag\\[2mm]
&\hspace{15pt} + \hspace{2pt} \int_{\tau - l}^{\tau} \dfrac{e^{-\frac{ \tau - s}{2}}}{ a\left( \tau - s\right)^\frac{1}{2}} \hspace{2pt} \mathrm{d} s \int_{\mathbb{R}^2} \Big|\hspace{1pt}U_\star\left(\cdot, s \right) \Big| \hspace{2pt} \Big|\hspace{1pt}\nabla E_2 * U_k - \nabla E_2 * U_\star\hspace{1pt}\Big| \left(\cdot,s\right).
\end{align*}Utilizing (\ref{Lp conv of Uk})--(\ref{conv of Vk}), (\ref{estimate_grad_v}) and (\ref{unif_bound_U})--(\ref{mass_conservation}), we can apply Lebesgue's dominated convergence theorem to the right-hand side above and infer \begin{align*}
\int_{\mathbb{R}^2} \left| \hspace{1pt}\overline{U}_{2, 2} -  \int_{\tau-l}^{\tau} \dfrac{\mathrm{d} s}{ a\left(\tau - s\right)}  \int_{\mathbb{R}^2}  U_\star\left(\eta,s\right) \nabla_\eta E_2 * U_\star\hspace{1pt}\Big|_{ \left(\eta,s\right)} \cdot \nabla_\eta\hspace{1pt} e^{ -\frac{\big| \xi - e^{-\frac{\tau - s}{2}} \eta \hspace{1pt} \big|^2}{4a\left(\tau-s\right)}} \mathrm{d} \eta \hspace{1pt} \right| \hspace{2pt}\mathrm{d} \xi \longrightarrow 0, \hspace{10pt}\text{as $ k \rightarrow \infty$.}\end{align*}Recalling (\ref{decomp of U2 tauk}), (\ref{id_7_31}) and the last convergence, we can take $k \rightarrow \infty$ and $l \rightarrow \infty$ successively. Hence \begin{align*} U_2\left(\xi, \tau_k + \tau\right) \xrightarrow{k \hspace{1pt}\rightarrow \hspace{1pt}\infty} \dfrac{1}{4\pi}\int_{- \infty}^{\tau} \dfrac{\mathrm{d} s}{ a\left(\tau - s\right)}  \int_{\mathbb{R}^2}  U_\star\left(\eta,s\right)\nabla_\eta E_2 * U_\star\hspace{1pt}\Big|_{ \left(\eta,s\right)} \cdot \nabla_\eta\hspace{1pt} e^{ -\frac{\big| \xi - e^{-\frac{\tau - s}{2}} \eta \hspace{1pt} \big|^2}{4a\left(\tau-s\right)}} \mathrm{d} \eta,
\end{align*} for all $\tau \in \mathbb{R}$ and a.\hspace{0.5pt}e. $\xi \in \mathbb{R}^2$. In light of (\ref{id_7_30}), the last convergence and (\ref{conv in spacetime of Uk}), the proof is completed by taking $k \rightarrow \infty$ in (\ref{decomp of Uk}). 
\end{proof}

Associated with the flow $\big\{U_\star\left(\cdot, \tau\right)\big\}$, we define
\begin{align}\label{def_u_star}
u_\star(x,t) = \frac{1}{t} \hspace{2pt} U_\star \left( \dfrac{x}{\sqrt{t}}, \hspace{2pt} \log t \right), \hspace{15pt}\text{for any $x \in \mathbb{R}^2$ and $t > 0$.}
\end{align}
(\ref{id_7_34}) then implies
\begin{align}\label{id_7_36}
\big\| u_\star (\cdot,t) \big\|_1 = M \quad \mbox{and} \quad \big\| u_\star (\cdot,t) \big\|_{\infty} \leq \frac{[ \hspace{1pt}u \hspace{1pt}]_\infty}{t}, \quad \mbox{for any} \hspace{3pt}t > 0.
\end{align}
\begin{lem}
The $u_\star$ defined in (\ref{def_u_star}) is a global mild solution to (\ref{eqn_0_0_2}) on $\mathbb{R}^{2+1}_+$. Moreover, \begin{align}\label{delta initial data}u_\star(\cdot,t) \overset{*}{\longrightharpoonup} M\delta_0 \hspace{15pt}\text{ as $t \rightarrow 0^+$.}\end{align} Here $\delta_0$ is the Dirac measure concentrated at $0$.
\end{lem}
\begin{proof}[\bf Proof] By (\ref{def_u_star}), the equation  (\ref{id_7_33}) can be equivalently written as follows:
\begin{align}\label{id_7_35}
u_\star(x,t) = \frac{M}{4\pi t} \hspace{2pt} e^{-\frac{|x|^2}{4t}} + \frac{1}{4\pi} \int_0^t  \dfrac{\mathrm{d}s}{ t - s}   \int_{\mathbb{R}^2}u_\star (z,s) \hspace{2pt} \nabla E_2 * u_\star\hspace{1pt}\bigg|_{ \left(z,s\right)} \cdot \nabla_z \hspace{1pt} e^{-\frac{|x-z|^2}{4(t-s)}} \hspace{2pt}\mathrm{d}z. 
\end{align} Since for all $t > 0$, the total mass of $u_\star\left(\cdot, t\right)$ is $M$, this equality infers \begin{align*} \int_{\mathbb{R}^2}\mathrm{d} x\int_0^t  \dfrac{\mathrm{d}s}{ t - s}   \int_{\mathbb{R}^2}u_\star (z,s) \hspace{2pt} \nabla E_2 * u_\star\hspace{1pt}\bigg|_{ \left(z,s\right)} \cdot \nabla_z \hspace{1pt} e^{-\frac{|x-z|^2}{4(t-s)}} \hspace{2pt}\mathrm{d}z = 0, \hspace{15pt}\text{for all $t > 0$.}
\end{align*}Given $\varphi$ a smooth function compactly supported on $\mathbb{R}^2$, the last equality yields\begin{align} II &\hspace{2pt}:= \hspace{2pt}\int_{\mathbb{R}^2}\varphi\left(x\right)\mathrm{d} x\int_0^t  \dfrac{\mathrm{d}s}{ t - s}   \int_{\mathbb{R}^2}u_\star (z,s) \hspace{2pt} \nabla E_2 * u_\star\hspace{1pt}\bigg|_{ \left(z,s\right)} \cdot \nabla_z \hspace{1pt} e^{-\frac{|x-z|^2}{4(t-s)}} \hspace{2pt}\mathrm{d}z \notag\\[2mm]
&\hspace{5pt}=\hspace{2pt}  \int_{\mathbb{R}^2}\big(\varphi(x) - \varphi(0)\big) \hspace{2pt}\mathrm{d} x\int_0^t  \dfrac{\mathrm{d}s}{ t - s}   \int_{\mathbb{R}^2}u_\star (z,s) \hspace{2pt} \nabla E_2 * u_\star\hspace{1pt}\bigg|_{ \left(z,s\right)} \cdot \nabla_z \hspace{1pt} e^{-\frac{|x-z|^2}{4(t-s)}} \hspace{2pt}\mathrm{d}z. 
\end{align}Using (\ref{id_7_36}), (\ref{estimate_grad_v}) and change of variables, we have 
\begin{align}\label{est of II}
\big|\hspace{1pt}II\hspace{1pt}\big| &\hspace{2pt}\lesssim_{\hspace{1pt}M, \hspace{1pt}[\hspace{1pt}u\hspace{1pt}]_\infty}\hspace{2pt}\int_{\mathbb{R}^2}\big|\varphi(x) - \varphi(0)\big| \hspace{2pt}\mathrm{d} x\int_0^t  \dfrac{\mathrm{d}s}{ \left(t - s\right)^2 s^{\frac{1}{2}}}   \int_{\mathbb{R}^2}u_\star (z,s) \hspace{2pt}  \big| x - z \big| \hspace{1pt} e^{-\frac{|x-z|^2}{4(t-s)}} \hspace{2pt}\mathrm{d}z \\[2mm]
&\hspace{2pt}=\hspace{2pt}\int_{\mathbb{R}^2}\hspace{2pt}|\hspace{1pt}y\hspace{1pt}| \hspace{1pt} e^{-\frac{|\hspace{0.5pt} y \hspace{0.5pt}|^2}{4}} \hspace{2pt}\mathrm{d} y\int_0^t  \dfrac{\mathrm{d}s}{ \left(t - s\right)^{\frac{1}{2}} s^{\frac{1}{2}}}   \int_{\mathbb{R}^2}u_\star (z,s) \hspace{2pt} \left|\varphi\left(z + \big(t - s\big)^{\frac{1}{2}} y\right) - \varphi(0)\right|  \hspace{2pt}\mathrm{d}z  \notag\\[2mm]
&\hspace{2pt}=\hspace{2pt}\int_{\mathbb{R}^2}\hspace{2pt}|\hspace{1pt}y\hspace{1pt}| \hspace{1pt} e^{-\frac{|\hspace{0.5pt} y \hspace{0.5pt}|^2}{4}} \hspace{2pt}\mathrm{d} y\int_0^1  \dfrac{ \mathrm{d}\tau}{ \left(1 -  \tau\right)^{\frac{1}{2}}  \tau^{\frac{1}{2}}}   \int_{\mathbb{R}^2} \dfrac{1}{\tau} \hspace{1pt}U_\star \left( \dfrac{z'}{\sqrt{\tau}}, \hspace{1pt}\log\left(t\hspace{0.5pt}\tau\right) \right)\hspace{2pt} \left|\varphi\left(t^\frac{1}{2}z' + \big(t - t \hspace{0.5pt}\tau\big)^{\frac{1}{2}} y\right) - \varphi(0)\right|  \hspace{2pt}\mathrm{d}z' .  \notag \end{align}The last equality in (\ref{est of II}) has used (\ref{def_u_star}). In light of (\ref{unif small outside a unif ball}) in Lemma \ref{lem_5_2}, for any $\epsilon > 0$, there  is a $r_\epsilon$ sufficiently large so that \begin{align*}\int_{B\mystrut^c_{r_\epsilon}} U\left(\xi, \tau_k + \tau \right) \mathrm{d} \xi \hspace{2pt}\leq \hspace{2pt}\epsilon, \hspace{15pt}\text{for all $\tau \in \mathbb{R}$ and $k$ sufficiently large.}
\end{align*}Recall (\ref{def of Uk}) and (\ref{conv in spacetime of Uk}). We then can take $k \rightarrow \infty$ in the above estimate. By Fatou's lemma, it turns out\begin{align}\label{uniform small Ustar outside a ball}\int_{B\mystrut^c_{r_\epsilon}} U_\star\left(\xi,  \tau \right) \mathrm{d} \xi \hspace{2pt}\leq \hspace{2pt}\epsilon, \hspace{15pt}\text{for all $\tau \in \mathbb{R}$.}
\end{align}Meanwhile, we can also take $r_\epsilon$ sufficiently large so that \begin{align}\label{uniform small gaussian func outside a ball}\int_{B\mystrut^c_{r_\epsilon}}\hspace{2pt}|\hspace{1pt}y\hspace{1pt}| \hspace{1pt} e^{-\frac{|\hspace{0.5pt} y \hspace{0.5pt}|^2}{4}} \hspace{2pt}\mathrm{d} y \hspace{2pt}\leq \hspace{2pt}\epsilon.
\end{align}Now we start estimating the last integral on the right-hand side of (\ref{est of II}). Firstly, it holds by (\ref{uniform small gaussian func outside a ball}) that \begin{align}\label{est of II.1}&\int_{B\mystrut^c_{r_\epsilon}}\hspace{2pt}|\hspace{1pt}y\hspace{1pt}| \hspace{1pt} e^{-\frac{|\hspace{0.5pt} y \hspace{0.5pt}|^2}{4}} \hspace{2pt}\mathrm{d} y\int_0^1  \dfrac{ \mathrm{d}\tau}{ \left(1 -  \tau\right)^{\frac{1}{2}}  \tau^{\frac{1}{2}}}   \int_{\mathbb{R}^2} \dfrac{1}{\tau} \hspace{1pt}U_\star \left( \dfrac{z'}{\sqrt{\tau}}, \hspace{1pt}\log\left(t\hspace{0.5pt}\tau\right) \right)\hspace{2pt} \left|\varphi\left(t^\frac{1}{2}z' + \big(t - t \hspace{0.5pt}\tau\big)^{\frac{1}{2}} y\right) - \varphi(0)\right|  \hspace{2pt}\mathrm{d}z' \notag\\[2mm]
&\hspace{100pt}\lesssim\hspace{2pt}M \hspace{1pt}\| \hspace{1pt}\varphi\hspace{1pt}\|_{\infty}\hspace{1pt}\int_{B\mystrut^c_{r_\epsilon}}\hspace{2pt}|\hspace{1pt}y\hspace{1pt}| \hspace{1pt} e^{-\frac{|\hspace{0.5pt} y \hspace{0.5pt}|^2}{4}} \hspace{2pt}\mathrm{d} y \hspace{2pt}\leq M \hspace{1pt}\| \hspace{1pt}\varphi\hspace{1pt}\|_{\infty}\hspace{1pt}\epsilon.
\end{align}Moreover, by (\ref{uniform small Ustar outside a ball}), we can infer \begin{align}\label{est of II.2}&\int_{B_{r_\epsilon}}\hspace{2pt}|\hspace{1pt}y\hspace{1pt}| \hspace{1pt} e^{-\frac{|\hspace{0.5pt} y \hspace{0.5pt}|^2}{4}} \hspace{2pt}\mathrm{d} y\int_0^1  \dfrac{ \mathrm{d}\tau}{ \left(1 -  \tau\right)^{\frac{1}{2}}  \tau^{\frac{1}{2}}}   \int_{B\mystrut^c_{r_\epsilon}} \dfrac{1}{\tau} \hspace{1pt}U_\star \left( \dfrac{z'}{\sqrt{\tau}}, \hspace{1pt}\log\left(t\hspace{0.5pt}\tau\right) \right)\hspace{2pt} \left|\varphi\left(t^\frac{1}{2}z' + \big(t - t \hspace{0.5pt}\tau\big)^{\frac{1}{2}} y\right) - \varphi(0)\right|  \hspace{2pt}\mathrm{d}z' \notag\\[2mm]
&\hspace{50pt}\lesssim\hspace{2pt} \| \hspace{1pt}\varphi\hspace{1pt}\|_{\infty}\hspace{1pt}\sup_{\tau' \hspace{1pt}\in \hspace{1pt}\mathbb{R}} \int_0^1 \dfrac{ \mathrm{d}\tau}{ \left(1 -  \tau\right)^{\frac{1}{2}}  \tau^{\frac{1}{2}}}  \int_{B\mystrut^c_{r_\epsilon \hspace{0.5pt}\tau^{- 1/2}}}\hspace{2pt}U_\star\left(\cdot, \tau'\right) \notag\\[2mm]
&\hspace{50pt}\lesssim\hspace{2pt} \| \hspace{1pt}\varphi\hspace{1pt}\|_{\infty}\hspace{1pt}\sup_{\tau' \hspace{1pt}\in \hspace{1pt}\mathbb{R}}\int_{B\mystrut^c_{r_\epsilon}}\hspace{2pt}U_\star\left(\cdot, \tau'\right) \leq \hspace{2pt} \| \hspace{1pt}\varphi\hspace{1pt}\|_{\infty}\hspace{1pt}\epsilon.
\end{align}By the above arguments, we can localize the spatial integration domains in the last integral of (\ref{est of II}) on the ball $B_{r_\epsilon}$. By taking $t \rightarrow 0$, it satisfies \begin{align*}&\int_{B_{r_\epsilon}}\hspace{2pt}|\hspace{1pt}y\hspace{1pt}| \hspace{1pt} e^{-\frac{|\hspace{0.5pt} y \hspace{0.5pt}|^2}{4}} \hspace{2pt}\mathrm{d} y\int_0^1  \dfrac{ \mathrm{d}\tau}{ \left(1 -  \tau\right)^{\frac{1}{2}}  \tau^{\frac{1}{2}}}   \int_{B_{r_\epsilon}} \dfrac{1}{\tau} \hspace{1pt}U_\star \left( \dfrac{z'}{\sqrt{\tau}}, \hspace{1pt}\log\left(t\hspace{0.5pt}\tau\right) \right)\hspace{2pt} \left|\varphi\left(t^\frac{1}{2}z' + \big(t - t \hspace{0.5pt}\tau\big)^{\frac{1}{2}} y\right) - \varphi(0)\hspace{1pt}\right|  \hspace{2pt}\mathrm{d}z' \notag\\[2mm]
&\hspace{60pt}\lesssim\hspace{2pt} M\hspace{1pt}\int_0^1  \dfrac{ \mathrm{d}\tau}{ \left(1 -  \tau\right)^{\frac{1}{2}}  \tau^{\frac{1}{2}}} \hspace{3pt}\sup_{\left(y, \hspace{0.5pt}z'\right)  \hspace{1pt}\in \hspace{1pt}B_{r_\epsilon} \times B_{r_\epsilon}} \left|\varphi\left(t^\frac{1}{2}z' + \big(t - t \hspace{0.5pt}\tau\big)^{\frac{1}{2}} y\right) - \varphi(0)\hspace{1pt}\right| \longrightarrow 0, \hspace{15pt}\text{as $t \rightarrow 0$.}
\end{align*}Applying the last estimate in conjunction with (\ref{est of II.1})--(\ref{est of II.2}) to (\ref{est of II}) then yields \begin{align*}II = \int_{\mathbb{R}^2}\varphi\left(x\right)\mathrm{d} x\int_0^t  \dfrac{\mathrm{d}s}{ t - s}   \int_{\mathbb{R}^2}u_\star (z,s) \hspace{2pt} \nabla E_2 * u_\star\hspace{1pt}\bigg|_{ \left(z,s\right)} \cdot \nabla_z \hspace{1pt} e^{-\frac{|x-z|^2}{4(t-s)}} \hspace{2pt}\mathrm{d}z  \longrightarrow 0, \hspace{10pt}\text{as $t \rightarrow 0$.}\end{align*}  By (\ref{id_7_35}) and the above convergence, (\ref{delta initial data}) holds. 
\end{proof}

\subsection{\normalsize Proof of Proposition \ref{weight asymptotics of u L1}}\label{sec_5_3}\vspace{0.5pc}

In this section, we finish the proof of Proposition \ref{weight asymptotics of u L1}. Notice that Proposition \ref{weight asymptotics of u L1} contains two parts. The first part is that $M < 8 \pi$, where $M$ is the total mass of the global mild solution $u$. The second part is the long-time asymptotics of  $u$. We discuss these two parts seperately in the following two lemmas.
\begin{lem} \label{M less 8pi}The total mass $M$ of $u_\star$ satisfies $M < 8 \pi$.  Here $u_\star$ is the limiting flow obtained in Sect.\hspace{0.5pt}5.2.
\end{lem}
\begin{proof}[\bf Proof] Since $u_\star$ is a global mild solution to (\ref{eqn_0_0_2}) with $\| u_\star(\cdot, t) \|_1 = M$, we have $M \leq 8\pi$. See \cite{W18}. In the remainder of the proof, we show $M \neq 8 \pi$.\vspace{0.3pc}

Let $\eta$ be a smooth radial function compactly supported on $B_2$. Moreover, $0 \leq \eta \leq 1$ on $\mathbb{R}^2$ and equivalently equals to $1$ on $B_1$. Fixing an arbitrary $R > 0$, we define \begin{align*}\eta_R(x) :=  \eta\left(\frac{x}{R}\right) \hspace{10pt}\text{and}\hspace{10pt}\eta^*_R(x) := \eta_R(x) \hspace{1pt} | x |^2, \qquad \text{ for any $x \in \mathbb{R}^2$.}\end{align*}
By multiplying $\eta^*_R$ on both sides of (\ref{eqn_0_0_2}) (here $u = u_\star$) and integrating over $\mathbb{R}^2$, it turns out
\begin{align}\label{id_7_38}
\dfrac{\mathrm{d}}{\mathrm{d} t} \int_{\mathbb{R}^2} u_\star \hspace{1.5pt}\eta^*_R = \int_{\mathbb{R}^2} u_\star \hspace{1.5pt} \Delta \eta^*_R - \frac{1}{4\pi} \int_{\mathbb{R}^2} \dfrac{\big( \nabla \eta^*_R(x) - \nabla \eta^*_R(y)\big) \cdot (x-y)}{|x-y|^2} \hspace{1.5pt}u_\star(x,t) \hspace{1.5pt} u_\star(y,t) \hspace{1.5pt}\mathrm{d}x \hspace{1.5pt}\mathrm{d}y.
\end{align}
The $L^\infty$\hspace{0.5pt}-\hspace{0.5pt}norms of the second-order partial derivatives of $\eta^*_R$ are uniformly bounded from above by constants  independent of $R$. Therefore, 
\begin{align*}
\dfrac{\mathrm{d}}{\mathrm{d}t} \int_{\mathbb{R}^2} u_\star \hspace{1.5pt}\eta^*_R \hspace{2pt}\lesssim\hspace{2pt} M + M^2.
\end{align*}
For any $0 < t_0 < t$, we integrate the above ODE inequality over the interval $[\hspace{0.5pt}t_0,t \hspace{0.5pt}]$. It follows
\begin{align*}
\int_{\mathbb{R}^2} u_\star(x,t)\hspace{1pt} \eta^*_R(x) \hspace{1pt} \mathrm{d}x \hspace{2pt}\leq\hspace{2pt} \int_{\mathbb{R}^2} u_\star(x,t_0) \hspace{1pt} \eta^*_R(x) \hspace{1pt} \mathrm{d}x + C\left( M + M^2 \right) t.
\end{align*}
Utilizing (\ref{delta initial data}), we can take $t_0 \rightarrow 0$ and $R \rightarrow \infty$ successively in the above estimate and obtain \begin{align}\label{second moment est of ustar}
\int_{\mathbb{R}^2} u_\star(x,t)\hspace{1pt} |\hspace{0.5pt}x\hspace{0.5pt}|^2 \hspace{1pt} \mathrm{d}x \hspace{2pt}\leq\hspace{2pt} C\left( M + M^2 \right) t.
\end{align}Here we have also used $\eta^*_R(0) = 0$. \vspace{0.3pc}

In light of (\ref{second moment est of ustar}), the second moment of $u_\star$ is finite for any $t > 0$. Suppose that $M = 8\pi$. Then the second moment of $u_\star(\cdot, t)$ must be a constant independent of $t > 0$. Therefore, the left-hand side of (\ref{second moment est of ustar}) is a constant. By taking $t \rightarrow 0$, the right-hand side of (\ref{second moment est of ustar}) converges to $0$. Hence, the second moment of $u_\star$ (also $u_\star$ itself) must be $0$ for all $t > 0$. But this is impossible since the total mass of $u_\star$ equals to $8 \pi$ at any $t > 0$. The proof is completed.
\end{proof}

Note that $u_\star$ is a global mild solution to (\ref{eqn_0_0_2}) with the initial data $M\delta_0$. In addition,  $M < 8\pi$ by Lemma \ref{M less 8pi}. Using the uniqueness result in Theorem 2 of \cite{BM14}, we conclude that
\begin{align*}
u_\star(x,t) = \frac{1}{t} \hspace{1pt} G_M \left( \dfrac{x}{\sqrt{t}} \right) \hspace{15pt}\text{on $\mathbb{R}_+^{2+1}$,}
\end{align*}
where $G_M$ is the unique stationary solution to (\ref{eqn_1_2d}) with the total mass $M$ and a finite free energy. Since the sequence $\big\{ \tau_k\big\}$ is arbitrary, the $\omega$-limit set of $\big\{ U\left(\cdot, \tau\right)\big\}$ contains only one element. More precisely, it equals to $\big\{ G_M\big\}$.  Therefore, the strong $L^1$-compactness result in Lemma \ref{lem_5_2} yields 
\begin{align*}
\lim_{\tau\rightarrow\infty} \big\| U(\cdot,\tau) - G_M \big\|_1 = 0.
\end{align*}This convergence shows the $p = 1$ case in (\ref{general p_L1_n=2}). We are left to prove \begin{lem}\label{conv of U in infty case}Suppose that $u$ satisfies the assumptions in Proposition \ref{weight asymptotics of u L1}. $U$ is the self-similar representation of $u$ given in (\ref{def_U}). Then we have
\begin{align*}
\lim_{k\rightarrow\infty} \big\|\hspace{1pt} U(\cdot,\tau) - G_M \hspace{1pt}\big\|_\infty = 0.
\end{align*}The last convergence and (\ref{def_U}) yield the $p = \infty$ case in (\ref{general p_L1_n=2}).
\end{lem}\noindent By this lemma and an interpolation argument, (\ref{general p_L1_n=2}) holds for general $p \in (1, \infty)$.

\begin{proof}[\bf Proof of Lemma \ref{conv of U in infty case}]  We still use the decomposition $U = U_1 + U_2$ in  (\ref{id_7_01}). 
For any $\tau \geq 1$, $R > 0$ and $\xi \in B\mynewstrut^c_R$, the function $U_1$ can be estimated by
\begin{align*}
U_1(\xi,\tau) &\hspace{2pt}\lesssim\hspace{2pt} \int_{\mathbb{R}^2} U_0\left(\eta\right)e^{-\frac{\big|\hspace{0.5pt} \xi - e^{- \frac{\tau}{2}} \hspace{0.5pt} \eta \hspace{0.5pt}\big|^2}{4} } \mathrm{d}\eta \hspace{2pt}=\hspace{2pt} \int_{B_{R/2}} U_0\left(\eta\right)e^{-\frac{\big|\hspace{0.5pt} \xi - e^{- \frac{\tau}{2}} \hspace{0.5pt} \eta \hspace{0.5pt}\big|^2}{4} } \mathrm{d}\eta + \int_{B\mystrut^c_{R/2}} U_0\left(\eta\right)e^{-\frac{\big|\hspace{0.5pt} \xi - e^{- \frac{\tau}{2}} \hspace{0.5pt} \eta \hspace{0.5pt}\big|^2}{4} } \mathrm{d}\eta \\[2mm]
&\hspace{2pt}\leq \hspace{2pt}\int_{B_{R/2}} U_0\left(\eta\right)e^{-\frac{|\xi|^2}{16}} \mathrm{d}\eta + \int_{B\mystrut^c_{R/2}} U_0\left(\eta\right) \mathrm{d}\eta \hspace{2pt}\leq\hspace{2pt} M e^{-\frac{R^2}{16}} + \int_{B\mystrut^c_{R/2}} U_0\left(\eta\right) \mathrm{d}\eta.
\end{align*}
By taking supreme over $\xi \in  B\mynewstrut^c_R$ and $\tau \in [\hspace{0.5pt} 1, \infty )$, the last estimate infers
\begin{align}\label{id_7_40}
\big\|\hspace{1pt} U_1 \hspace{1pt}\big\|_{\infty; \hspace{1pt}B\mystrut^c_R \times [\hspace{0.5pt}1,\infty)} \longrightarrow 0 \qquad \mbox{as } R \rightarrow \infty.
\end{align} 

Now we estimate $U_2$. For any $\xi \in \mathbb{R}^2$, $\tau \geq 1$ and $0 < \sigma < \tau$, by (\ref{unif_bound_U})--(\ref{mass_conservation}) and (\ref{estimate_grad_v}), it holds
\begin{align*}
& \left|\hspace{1pt} \int_{\tau-\sigma}^{\tau} \int_{\mathbb{R}^2} \dfrac{e^{-\frac{\tau - s}{2}}}{ a\left(\tau - s\right)^2} \hspace{2pt} U\left(\eta,s\right)  \nabla_\eta E_2 *U\hspace{1pt}\bigg|_{ \left(\eta,s\right)} \cdot \left(\xi -  e^{- \frac{\tau - s }{2}} \eta \right) e^{ -\frac{\big| \xi - e^{-\frac{\tau - s}{2}} \eta \hspace{1pt} \big|^2}{4a\left(\tau-s\right)}} \mathrm{d} \eta \hspace{2pt} \mathrm{d}s \hspace{1pt}\right| \\[1mm]
&\hspace{130pt}\lesssim_{\hspace{1pt}M, \hspace{1pt}[ \hspace{1pt}u \hspace{1pt}]_\infty}  \int_{\tau-\sigma}^{\tau} \dfrac{e^{\frac{\tau - s}{2}}}{ a\left(\tau - s\right)^{\frac{1}{2}}} \hspace{2pt}\mathrm{d}s \hspace{2pt}=\hspace{2pt} 2\hspace{1.5pt}\big(e^\sigma - 1\big)^{\frac{1}{2}}.
\end{align*}
Hence, for any given $\epsilon > 0$, there is $\sigma_\epsilon > 0$ sufficiently small such that
\begin{align}\label{id_7_41}
\left\|\hspace{1pt} \int_{\tau-\sigma_\epsilon}^{\tau} \int_{\mathbb{R}^2} \dfrac{e^{-\frac{\tau - s}{2}}}{ a\left(\tau - s\right)^2} \hspace{2pt} U\left(\eta,s\right)  \nabla_\eta E_2 * U\hspace{1pt}\bigg|_{ \left(\eta,s\right)} \cdot \left(\xi -  e^{- \frac{\tau - s }{2}} \eta \right) e^{ -\frac{\big| \xi - e^{-\frac{\tau - s}{2}} \eta \hspace{1pt} \big|^2}{4a\left(\tau-s\right)}} \mathrm{d} \eta \hspace{2pt}\mathrm{d}s \hspace{1pt}\right\|_{\infty; \hspace{1pt}\mathbb{R}^2 \times [\hspace{0.5pt}1, \infty)} < \epsilon.
\end{align}
In the next, we assume that $R > 2\sqrt{2}$. For any $\xi \in B\mynewstrut^c_R$, $\eta \in B_{R/2}$ and $s < \tau$, it holds \begin{align*} \dfrac{\left|\hspace{1pt} \xi - e^{-\frac{\tau - s}{2}} \eta \hspace{1pt} \right|}{a(\tau - s)^{\frac{1}{2}}} \hspace{2pt}\geq\hspace{2pt}\dfrac{|\hspace{0.5pt}\xi\hspace{0.5pt}|}{2} \hspace{2pt}>\hspace{2pt}\sqrt{2}.
\end{align*} Since the function $r e^{-\frac{r^2}{4}}$ is monotonically decreasing on $\left[ \sqrt{2} , \infty \right)$, the last estimate yields
\begin{align}\label{id_7_42}
& \left|\hspace{1pt} \int_0^{\tau-\sigma_\epsilon} \int_{B_{R/2}} \dfrac{e^{-\frac{\tau - s}{2}}}{ a\left(\tau - s\right)^2} \hspace{2pt} U\left(\eta,s\right) \nabla_\eta E_2 * U\hspace{1pt}\bigg|_{ \left(\eta,s\right)} \cdot \left(\xi -  e^{- \frac{\tau - s }{2}} \eta \right) e^{ -\frac{\big| \xi - e^{-\frac{\tau - s}{2}} \eta \hspace{1pt} \big|^2}{4a\left(\tau-s\right)}} \mathrm{d} \eta \hspace{2pt}\mathrm{d}s \hspace{1pt}\right| \notag\\[2mm]
&\hspace{60pt}\leq\hspace{2pt} |\hspace{0.5pt}\xi\hspace{0.5pt}|\hspace{1pt} e^{-\frac{|\hspace{0.5pt}\xi\hspace{0.5pt}|^2}{16}} \int_0^{\tau-\sigma_\epsilon} \int_{B_{R/2}} \dfrac{e^{-\frac{\tau - s}{2}}}{ a\left(\tau - s\right)^\frac{3}{2}} \hspace{2pt} U\left(\eta,s\right) \Big|  \nabla_\eta E_2 * U\hspace{1pt}\Big|\left(\eta,s\right)  \mathrm{d} \eta \hspace{2pt}\mathrm{d}s \notag\\[2mm]
&\hspace{60pt}\lesssim\hspace{2pt} a(\sigma_\epsilon)^{-\frac{3}{2}} M^{\frac{3}{2}} [ \hspace{1pt}u \hspace{1pt}]_\infty^{\frac{1}{2}} \hspace{1pt}R \hspace{1pt}e^{-\frac{R^2}{16}} \hspace{2pt}<\hspace{2pt} \epsilon, \qquad \mbox{provided that $R$ is sufficiently large}.
\end{align}
Moreover, still using (\ref{estimate_grad_v}), we obtain
\begin{align*} 
&  \left|\hspace{1pt} \int_0^{\tau-\sigma_\epsilon} \int_{B\mystrut^c_{R/2}} \dfrac{e^{-\frac{\tau - s}{2}}}{ a\left(\tau - s\right)^2} \hspace{2pt} U\left(\eta,s\right) \nabla_\eta E_2 * U\hspace{1pt}\bigg|_{ \left(\eta,s\right)} \cdot \left(\xi -  e^{- \frac{\tau - s }{2}} \eta \right) e^{ -\frac{\big| \xi - e^{-\frac{\tau - s}{2}} \eta \hspace{1pt} \big|^2}{4a\left(\tau-s\right)}} \mathrm{d} \eta \hspace{2pt}\mathrm{d}s \hspace{1pt}\right| \notag\\[2mm]
&\hspace{60pt}\lesssim\hspace{2pt} \int_0^{\tau-\sigma_\epsilon} \int_{B\mystrut^c_{R/2}} \dfrac{e^{-\frac{\tau - s}{2}}}{ a\left(\tau - s\right)^{\frac{3}{2}}} \hspace{2pt} U\left(\eta,s\right) \Big|\hspace{1pt}\nabla_\eta E_2 * U\hspace{1pt}\Big| \left(\eta,s\right)  \mathrm{d} \eta \hspace{2pt}\mathrm{d}s \notag\\[2mm]
&\hspace{60pt}\lesssim \hspace{2pt} a(\sigma_\epsilon)^{-\frac{3}{2}} M^{\frac{1}{2}} [ \hspace{1pt}u \hspace{1pt}]_\infty^{\frac{1}{2}} \int_0^{\tau-\sigma_\epsilon}e^{-\frac{\tau - s}{2}}  \hspace{2pt}\mathrm{d} s \int_{B\mystrut^c_{R/2}} U\left(\eta,s\right) \mathrm{d} \eta.
\end{align*}If we take $R$ sufficiently large so that $R > 2 r_\epsilon$ (here $r_\epsilon$ is given in (\ref{unif small outside a unif ball})), then the last estimate can be reduced to
\begin{align*} 
&  \left|\hspace{1pt} \int_0^{\tau-\sigma_\epsilon} \int_{B\mystrut^c_{R/2}} \dfrac{e^{-\frac{\tau - s}{2}}}{ a\left(\tau - s\right)^2} \hspace{2pt} U\left(\eta,s\right) \nabla_\eta E_2 * U\hspace{1pt}\bigg|_{ \left(\eta,s\right)} \cdot \left(\xi -  e^{- \frac{\tau - s }{2}} \eta \right) e^{ -\frac{\big| \xi - e^{-\frac{\tau - s}{2}} \eta \hspace{1pt} \big|^2}{4a\left(\tau-s\right)}} \mathrm{d} \eta \hspace{1.5pt} \mathrm{d}s \hspace{1pt}\right|  \hspace{2pt}\lesssim\hspace{2pt} 
\left(\dfrac{M \hspace{1pt} [ \hspace{1pt}u \hspace{1pt}]_\infty}{a(\sigma_\epsilon)^{3}}\right)^{\frac{1}{2}} \hspace{2pt}\epsilon.\end{align*}
Combining (\ref{id_7_40}), (\ref{id_7_41}), (\ref{id_7_42}) and the last estimate, we get
\begin{align*}
\big\|\hspace{1pt} U \hspace{1pt}\big\|_{\infty; \hspace{1pt}B\mystrut^c_R \times [\hspace{0.5pt}1,\infty\hspace{0.5pt})} \longrightarrow 0, \qquad \mbox{as } R \rightarrow \infty.
\end{align*}The proof of this lemma is completed by the last convergence and (\ref{loc unif conv of U}). Notice that here for any $\big\{\tau_k\big\}$ which diverges to $\infty$ as $k \rightarrow \infty$, the sequence of functions $\big\{U\left(\cdot, \tau_k\right) \big\} $ converges locally uniformly to $G_M$ as $k \rightarrow \infty$.
\end{proof}


\vspace{1pc}

\noindent\textbf{Acknowledgement:} Y. Yu is partially supported by RGC grant of Hong Kong, Grant No. 14302718.

\vspace{2pc}
Department of Mathematics, The Chinese University of Hong Kong, Hong Kong \\[1mm]
\textit{E-mail address}: cyhsieh@math.cuhk.edu.hk\\[2mm]
Department of Mathematics, The Chinese University of Hong Kong, Hong Kong\\[1mm]
\textit{E-mail address}: yongyu@math.cuhk.edu.hk

\end{document}